\documentclass[12pt]{elsarticle}

\usepackage{amsmath,amssymb,amsthm}
\usepackage{mathrsfs}
\usepackage[bookmarks=true,
bookmarksnumbered=true,
bookmarksopen=true,
colorlinks,
pdfborder=001,
linkcolor=black]{hyperref}
\usepackage{bm,bbm}
\usepackage{array}
\usepackage{float}
\usepackage{color}
\usepackage{lineno}
\usepackage{caption}
\usepackage{stmaryrd}
\usepackage{extarrows}
\usepackage{subcaption}
\usepackage{graphicx}
\usepackage{multirow}
\usepackage{multicol}
\usepackage{epstopdf}
\usepackage{geometry}

\renewcommand{\vec}[1]{\mbox{\boldmath \small $#1$}}

\theoremstyle{plain}
\newtheorem{theorem}{Theorem}[section]
\newtheorem{example}{Example}[section]
\newtheorem{lemma}{Lemma}[section]

\newtheorem{remark}{Remark}[section]
\newtheorem{definition}{Definition}[section]
\newtheorem{proposition}{Proposition}[section]
\newtheorem{assumption}{Assumption}[section]
\numberwithin{equation}{section}
\numberwithin{figure}{section}
\numberwithin{table}{section}
\numberwithin{theorem}{section}

\newcommand{\defave}[1]{ \left\{ #1 \right\}}
\newcommand{\defjump}[1]{\left[\left[ #1 \right]\right]}

\newcommand{\defin}[1]{#1 ^{\tt in}}
\newcommand{\defout}[1]{#1 ^{\tt out}}

\newcommand{\normcp}[2]{\| #2 \|_{{\small C}^{#1}(\overline{\Omega})}}
\newcommand{\projection}[1]{\Pi_h[ #1 ]}

\def\inerface{\Sigma_{\tt int}}

\renewcommand{\div}{\mbox{div}}

\def\softd{{\leavevmode\setbox1=\hbox{d}%
          \hbox to 1.05\wd1{d\kern-0.4ex{\char039}\hss}}}
\begin{document}
	
\begin{frontmatter}
\title{Convergence of first-order Finite Volume Method based on Exact Riemann Solver for the Complete Compressible Euler Equations}
\author{M\'{a}ria Luk\'{a}\v{c}ov\'{a}-Medvi{\softd}ov\'{a}}
\ead{lukacova@uni-mainz.de}

\author{Yuhuan Yuan}
\ead{yuhuyuan@uni-mainz.de}
\address{Institute of Mathematics, Johannes Gutenberg-University Mainz\\Staudingerweg 9, 55128 Mainz, Germany}

\begin{abstract}
Recently developed concept of dissipative measure-valued solution for compressible flows is a suitable tool to describe oscillations and singularities possibly developed in solutions of multidimensional Euler equations.
In this paper we study the convergence of the first-order finite volume method based on the exact Riemann solver for the complete compressible Euler equations.
Specifically, we derive entropy inequality and prove the consistency of numerical method.  
Passing to the limit, we show the weak and strong convergence of numerical solutions and identify their limit.
The numerical results presented for the spiral, Kelvin-Helmholtz and the Richtmyer-Meshkov problem are consistent with our theoretical analysis.
\end{abstract}


\end{frontmatter}

\section{Introduction}
\noindent Hyperbolic conservation laws play an important role in describing many physical and engineering process.
An iconic example is the nonlinear system of compressible Euler equations, which governs the dynamics of a compressible material and incorporates mass, momentum and energy conservation.

A characteristic feature of nonlinear conservation laws is that discontinuities (shock waves) may develop after finite time even if the initial condition is smooth.
A natural way to overcome this difficulty would be to work with a concept of weak distributional solution.
However, it is well-known that such weak solutions fail to be unique.
Consequently, the second law of thermodynamics have been proposed as a selection criterion to rule out nonphysical solutions.
The entropy production principle have been successfully applied in the scalar multidimensional equations \cite{Kruzkov:1970} and one-dimensional systems \cite{Bressan-Crasta-Piccoli:2000,Bressan-Lewicka:2000}.
Unfortunately, it completely fails for multidimensional systems.
Recently, De Lellis and Sz\'{e}kelyhidi \cite{DeLellis-Szekelyhidi:2010} and Chiodaroli et al.  \cite{Chiodaroli-DeLellis-Kreml:2015} showed non-uniqueness of weak entropy solutions  for the multidimensional isentropic compressible Euler system, see also \cite{DeLellis-Szekelyhidi:2010} for similar non-uniqueness results for the incompressible Euler system.
These results have been extended to the complete Euler system in Feireisl et al.  \cite{Feireisl-Klingenberg-Kreml-Markfelder:2020}.

In order to describe oscillations arising in the limits of singular perturbations of hyperbolic conservation laws, a more generalized solution, i.e. measure-valued (MV) solution, was suggested by DiPerna.
In 1985 he showed for one-dimensional hyperbolic conservation laws that regularized solutions of associated diffusive and dispersive regularized systems converge to a MV solution, as a regularized parameter vanishes \cite{DiPerna:1985a}.
Recently the concept of MV solution was adapted and applied for multidimensional  compressible Euler system \cite{Kroner-Zajaczkowski:1996} and three-dimensional incompressible Euler equations \cite{DiPerna-Majda:1987},
see also \cite{Brenier-DeLellis-Szekelyhidi:2011,Malek-Necas-Rokyta-Ruzicka:1996}.

In this paper we work with the concept of \emph{dissipative measure-valued (DMV) solution} for the complete Euler system,
which enjoys the weak-strong uniqueness principle \cite{Brezina-Feireisl:2018a}.
Analogous concept has been adopted for the insentropic Euler system \cite{Gwiazda-Swierczewska-Wiedemann:2015},   compressible Navier-Stokes system \cite{Feireisl-Gwiazda-Swierczewska-Wiedemann:2016},  elastodynamics \cite{Demoulini-Stuart-Tzavaras:2012} and other related problems.

In the convergence analysis of numerical schemes for hyperbolic conservation laws
the entropy stability plays a crucial role, we refer a reader to a seminal paper of Tadmor \cite{Tadmor:1987}.
For multidimensional hyperbolic conservation laws the convergence analysis to MV solutions  was studied by
Fjordholm et al. 
\cite{Fjordholm-Mishra-Tadmor:2016,Fjordholm-Kappeli-Mishra-Tadmor:2017}.
For the Lax-Friedrichs and vanishing viscosity finite volume method,
Feireisl,  Luk\'{a}\v{c}ov\'{a} and Mizerov{\'a} generalized the above convergence result for the Euler system and proved the convergence to the DMV solution and the classical solution on its lifespan  \cite{Feireisl-Lukacova-Mizerova:2020a,Feireisl-Lukacova-Mizerova:2020}.
For generalization to viscous compressible flows we refer a reader to \cite{Feireisl-Lukacova-Mizerova-She:2019,Feireisl-Lukacova-Mizerova-She:2021}.
In
\cite{Feireisl-Lukacova-Mizerova:2020b,Feireisl-Lukacova-She-Wang:2019} a new tool, $\mathcal{K}$-convergence, has been developed to compute strong limits of oscillatory sequences of numerical solutions.

In the present paper we focus on the first-order finite volume method based on the exact Riemann problem solver for the complete compressible Euler system and show its convergence via
the concept of  DMV solutions.
We will only assume that our numerical solutions stay in a physically non-degenerate region, i.e.~density is bounded from below and energy is bounded from above.
Interestingly, this assumption is equivalent to the strict convexity of the mathematical entropy, see Lemma~\ref{lemma:equivalent-statement}.

The rest of the paper is organized as follows.
In Section~2 we introduce suitable notations and describe the finite volume method with a numerical flux based on the exact solution of the local Riemann problem, cf.~Godunov method.
In Section~3 we analyze the entropy inequality and give an explicit lower bound of the entropy Hessian matrix, see Appendix~\ref{section:lowerbound}.
A crucial step is to show the consistency of our numerical method.
Section~4 is devoted to the limiting process.
We prove that the numerical solutions will generate a weakly-(*) convergent subsequence and a Young measure, a DMV solution, to the Euler system.
Furthermore, employing the theory of $\mathcal{K}$-convergence and DMV--strong uniqueness principle, we obtain the strong convergence of the Ces\`{a}ro averages and strong convergence of numerical solutions to the weak/strong/classical solution.
Finally, in Section~5 we present numerical simulations and illustrate the effects of $\mathcal{K}$-convergence of numerical solutions.
The numerical results clearly demonstrate convergence results being consistent with our analysis.
As far as we know this is the first result in literature where the convergence of a well-known Godunov method has been proved rigorously for multidimensional Euler equations.

\section{Numerical method}
\noindent The complete Euler system can be written in the divergence form
\begin{equation} \label{eq:multiD-Euler}
\partial_t \vec U + \div_{\vec x} \vec F(\vec U) = 0, \quad (t,\vec x) \in (0,T) \times \Omega,
\end{equation}
where $\vec U, \vec F$ denote the conservative vector and flux, defined by
 \begin{equation*}
\vec U= (\rho, \vec m, E)^T, \hspace{1cm}
\vec F_{i}(\vec U)= (\rho u_i,\, u_i \vec m+ p \vec e_i, u_i (E+p))^T,\hspace{1cm}
i = 1, \cdots, d.
 \end{equation*}
Here $\rho$ denotes density, $\vec u:=(u_1, \cdots, u_d)$ velocity, $\vec m:= \rho \vec u$ momentum, $E := \frac{\rho |\vec u|^2}{2} + \rho e$ total energy,  $p$ pressure, $e$ internal energy, and the row vector $\vec e_i$ represents the $i$-th row of the unit matrix of size $d := \mbox{dim} (\vec x ), d=1,2,3$.

Throughout the whole text, we consider the bounded domain $\Omega \subset \mathbb{R}^d$ together with the space-periodic or no-flux boundary condition.
Moreover, the equation of state is restricted to
\begin{equation}
p = (\gamma-1)\rho e,
\end{equation}
where $\gamma \in (1,2]$ is the adiabatic constant.

\begin{remark}
For the Euler system \eqref{eq:multiD-Euler} the mathematical entropy pair can be given by
\begin{equation}\label{eq:mathematical-entropy-pair}
\eta = \frac{-\rho S}{\gamma -1}, \quad \vec q = \frac{- \rho S \vec u }{\gamma -1} = \eta \vec u
\end{equation}
with thermodynamic entropy $S$ defined by
\begin{equation}
S := \ln(p) - \gamma \ln(\rho).
\end{equation}
The corresponding entropy variable and entropy potential are given by
\begin{equation}
\vec \nu  := \nabla_{\vec U} \eta = \begin{pmatrix}
\frac{\gamma - S}{\gamma - 1} - \frac{\rho |\vec u|^2}{2p}\\
\frac{\rho \vec u}{p} \\
-\frac{\rho}{p}
\end{pmatrix},\quad  \vec \psi = \rho \vec u,
\end{equation}
respectively.
We require that in addition to \eqref{eq:multiD-Euler} the entropy inequality holds, i.e.
\begin{equation}\label{eq:Euler-entropy-inequality}
\partial_t  \, \eta(\vec U ) +  \rm{div}_{\vec x} \,\vec q(\vec U) \leq 0.
\end{equation}
\end{remark}

\subsection{Spatial discretization and notations}
\noindent The computational domain $\Omega$ consists of rectangular meshes $\overline{\Omega} := \bigcup_{K} \overline{K}$.
We denote the set of all mesh cells as $\mathcal{T}_h$, $\sigma$ stands for cell face, a unit normal vector to $\sigma$ is $\vec n$,
the face between two neighboring cells $K$ and $L$ is denoted by $\sigma_{K,L} = K|L$,
and the set of all interior faces is given as $\inerface = \{ \sigma \in \Sigma ~ : ~ \sigma \cap \partial \Omega = \emptyset\}$.
Moreover, for $\vec x \in \sigma$ we define
\begin{equation}
\defout{a}(\vec x) = \lim\limits_{\delta \rightarrow 0+} a(\vec x + \delta \vec n), \quad \defin{a}(\vec x) = \lim\limits_{\delta \rightarrow 0+} a(\vec x - \delta \vec n)
\end{equation}
and introduce the standard average- and jump-operators
\begin{align}
& \defave{a} = \frac{ \defin{a} + \defout{a} }{2},
\quad \defjump{a} = \defout{a} - \defin{a}.
\end{align}
We use the notation $a \lesssim b$ if there exists a generic constant $C > 0$ independent on the mesh discretization, such that $a \leq C\cdot b$ for $a,b\in \mathbb{R}$.

\subsection{Discrete function space and finite volume method}
\noindent Consider the space of piecewise constant functions
\begin{equation}
\mathcal{Q}_h(\Omega) = \{ v :  v|_{K^o} = \mbox{constant}, ~ \mbox{for all}~ K \in \mathcal{T}_h \}.
\end{equation}
We can define the projection
\begin{equation}
\Pi_h :  L^1(\Omega) \rightarrow \mathcal{Q}_h(\Omega), \quad \projection{\phi}_K = \frac{1}{|K|} \int_K \phi (x) ~dx,
\end{equation}
where $|K|$ is the Lebesgue measure of $K$.

Let $\vec U_h \in  \mathcal{Q}_h(\Omega;\mathbb{R}^{d+2})$.
We consider a semi-discrete finite volume method
\begin{subequations}
\begin{align}
& \frac{d }{d t} \vec U_K(t) + (\div_h \vec F_h(t))_K = 0, \quad t>0, ~ K \in \mathcal{T}_h, \label{eq:semi-discrete}\\
& \vec U_K(0) = \varPi_h(\vec U^0)_K, \quad  K \in \mathcal{T}_h,
\end{align}
\end{subequations}
where the numerical flux function $\vec F_h$ is given by means of the exact Riemann solver \cite{Toro:2009}.
Consequently, \eqref{eq:semi-discrete} can be rewritten as
\begin{equation} \label{eq:semi-discrete-RPflux}
\int_{K} \frac{d }{d t} \vec U_K(t) ~ d\vec x +  \sum_{\partial K \cap \partial L \neq \emptyset} \int_{\partial K \cap \partial L} \vec F(\vec U^{\it RP}_{\sigma_{K,L}}) \cdot \vec n_K  ~dS_{\vec x} = 0,
\end{equation}
where $\vec U^{\it RP}_{\sigma_{K,L}}$ is the solution of a local Riemann problem at face $\sigma_{K,L}$.
A piecewise constant solution $\vec U_h = \{ \vec U_K; K \in \mathcal{T}_h \}$ satisfies an equivalent weak form
\begin{equation} \label{eq:semi-discrete-RPflux-phi}
\int_{\Omega} \phi  \frac{d }{d t} \vec U_h ~ d\vec x -  \sum_{\sigma \in \inerface} \int_{\sigma} \vec F(\vec U^{\it RP}_{\sigma}) \cdot \vec n \defjump{\phi} ~dS_{\vec x} = 0,
\end{equation}
where $\phi \in \mathcal{Q}_h(\Omega)$.

\begin{remark}
	From here on, we write $(x,y)^T$ rather than $\vec x$ for $d=2$, and $(x,y,z)^T$  for $d=3$.
	For a two-dimensional uniform mesh a general cell is denoted by $K_{ij} = (x_{i-\frac12}, x_{i+\frac12}) \times  (y_{j-\frac12}, y_{j+\frac12})$
	with $(x_i,y_j)$ as the center, $\Delta x =  x_{i+\frac12} - x_{i-\frac12}$ and $\Delta y =  y_{j+\frac12} - y_{j-\frac12}$.
	Hence, \eqref{eq:semi-discrete-RPflux} can be rewritten as
	\begin{equation*}
	\frac{d }{d t} \vec U_{{ij}}(t)  + \frac{1}{\Delta x} \left( \vec F_1(\vec U^{\it RP}_{i+\frac12,j}) - \vec F_1(\vec U^{\it RP}_{i-\frac12,j}) \right)   + \frac{1}{\Delta y} \left( \vec F_2(\vec U^{\it RP}_{i,j+\frac12}) - \vec F_2(\vec U^{\it RP}_{i,j-\frac12}) \right) = 0,
	\end{equation*}
	where  $\vec U_{ij} := \vec U_{K_{ij}}$ and $\vec U^{\it RP}_{i+\frac12,j}$ is the solution at $(x_{i+\frac12}, t+0)$ of the following local Riemann problem
	\begin{equation*}
	\begin{cases}
	\partial_{\tau} \vec U + \partial_{x} \vec F_1 = 0, \quad \tau > t, \\
	\vec U(x,t) = \begin{cases}
	\vec U_{{i,j}}(t), & \mbox{ if }~ x < x_{i+\frac12}, \\
	\vec U_{{i+1,j}}(t), & \mbox{ if }~ x > x_{i+\frac12}
	\end{cases}
	\end{cases}
	\end{equation*}
	and $\vec U^{\it RP}_{i,j+\frac12}$ is the solution at $(y_{j+\frac12},t+0)$ of local Riemann problem
	\begin{equation*}
	\begin{cases}
	\partial_{\tau} \vec U + \partial_{y} \vec F_2 = 0, \quad \tau > t, \\
	\vec U(y,t) = \begin{cases}
	\vec U_{{i,j}}(t), & \mbox{ if }~ y < y_{j+\frac12}, \\
	\vec U_{{i,j+1}}(t), & \mbox{ if }~ y > y_{j+\frac12}.
	\end{cases}
	\end{cases}
	\end{equation*}
\end{remark}

\section{Consistency}
\noindent Before proving the consistency of finite volume method  \eqref{eq:semi-discrete-RPflux-phi} we formulate a physically reasonable assumption.
\begin{assumption}\label{assumption}
We assume that
\begin{equation}
0 < \underline{\rho} \leq \rho_h, \quad
E_h \leq \overline{E} \quad \mbox{uniformly for } h \rightarrow 0,~ t\in[0,T],
\end{equation}
where $ \underline{\rho},  \overline{E}$ are positive constants.
\end{assumption}

\begin{lemma}\label{lemma:lemma-U-bounded}
Under Assumption~\ref{assumption}
it holds
\begin{align}
&0 < \underline{\rho} \leq \rho_h \leq \overline{\rho}, \quad
|\vec u| \leq \overline{u},\quad
0 < \underline{p} \leq p_h \leq \overline{p}, ~ \\
& |\vec m| \leq \overline{m},\quad
0 < \underline{E} \leq E_h \leq \overline{E}, \quad
0 < \underline{\vartheta} \leq \vartheta_h \leq \overline{\vartheta}
\end{align}
uniformly for $h \rightarrow 0, t\in[0,T]$ with positive constants $\underline{\rho},\,\overline{\rho},\, \overline{u},\, \underline{p},\, \overline{p},\, \overline{m},\, \underline{E},\,\overline{E},\, \underline{\vartheta},\,\overline{\vartheta}$, where $\vartheta := \frac{p}{\rho}$ is the temperature.
\end{lemma}

\begin{proof}
Here we only give the idea and framework of the proof. More details could be found in \cite{Feireisl-Lukacova-Mizerova:2020a}.

It is easy to observe that
\begin{align*}
& p = (\gamma-1) \left[E - \frac{|\vec m|^2}{2\rho} \right] \leq (\gamma-1)E \leq (\gamma-1) \overline{E} , \quad
|\vec u|^2 \leq \frac{2E}{\rho} \leq \frac{2\overline{E}}{\underline{\rho}},
\end{align*}
which implies
\begin{equation*}
\overline{p} = (\gamma-1) \overline{E}, \quad
\overline{u} = \sqrt{ \frac{2\overline{E}}{\underline{\rho}}}, \quad
\overline{\vartheta} = \frac{\overline{p}}{\underline{\rho}}.
\end{equation*}
On the other hand, we use the fact that the finite volume method \eqref{eq:semi-discrete-RPflux-phi} is entropy stable, cf.  \cite{Harten-Lax-VanLeer:1983,Chen-Shu:2017}.
Thus, using the renormalized entropy
\begin{align*}
& \eta = \rho \chi(S) , \quad \chi '(z) \geq 0, \quad
 \chi (z) = \begin{cases}
< 0, & \mbox{\it if }~z \leq z_0, \\
= 0, & \mbox{\it if }~z \geq z_0,\\
\end{cases}
  \\
& z_0 = (\gamma-1) \ln(1/\overline{C}), \quad
	\overline{C} := \max_{K} \frac{\rho_K(0)}{\vartheta_K(0)^{1/(\gamma-1)}}
\end{align*}
implies
\begin{equation}
0 < \rho_h \leq \overline{C}\vartheta_h^{1/(\gamma-1)}.
\end{equation}
Consequently, we obtain
\begin{align*}
0 < \underline{\rho}^{\gamma} \leq \rho_h^{\gamma}
\leq   \overline{C}^{\gamma-1} \rho_h  \vartheta_h = \overline{C}^{\gamma-1}  p_h \leq \overline{C}^{\gamma-1}  (\gamma-1) E_h \leq \overline{C}^{\gamma-1}  (\gamma-1) \overline{E},
\end{align*}
which gives
\begin{align*}
& \overline{\rho} = \big(\overline{C}^{\gamma-1}  (\gamma-1) \overline{E}\big)^{1/\gamma}, \quad
	\overline{m} = \sqrt{\overline{\rho}\overline{E} }, \quad
 \underline{p} = \frac{\underline{\rho}^{\gamma}}{\overline{C}^{\gamma-1} }, \quad
	\underline{E} =  \frac{\underline{\rho}^{\gamma}}{(\gamma-1)\overline{C}^{\gamma-1} }, \quad
	\underline{\vartheta} = \frac{\underline{p}}{\overline{\rho}}
\end{align*}
and concludes the proof.
\end{proof}

Note that Assumption~\ref{assumption} is equivalent to the strict convexity of mathematical entropy function \eqref{eq:mathematical-entropy-pair}, see Appendix~\ref{section:lowerbound}. 

\begin{lemma} \label{lemma:usefulestimation}
Under Assumption~\ref{assumption}
it holds
\footnote{We use the notations $\|\cdot\|,\, \|\cdot\|_2$ and $|\cdot|$ for
 the $L^1$-, $L^2$-norm and the absolute value, respectively.}
\begin{itemize}
\item[(1)]  $\|\defjump{\vec F_h} \cdot \vec n\| \lesssim \|\defjump{\vec U_h}\| < \infty$,
$\|\defjump{\vec V_h}\| \lesssim \|\defjump{\vec U_h}\| < \infty$,~$\|\defjump{\vec U_h}\| \lesssim \|\defjump{\vec V_h}\| < \infty$.
\item[(2)] $\nabla_{\vec U}^2 \eta(\vec U_h) \geq \underline{\eta} I $ with a positive constant $\underline{\eta}$.
\end{itemize}
Here $\vec V = (\rho, \vec u, p)^T$ and $\vec F_h = \vec F(\vec U_h), ~ \vec \nu_h = \vec \nu(\vec U_h),~  \vec V_h = \vec V(\vec U_h)$.
\end{lemma}

\begin{proof} Lemma~\ref{lemma:lemma-U-bounded} implies the boundedness of $\{\vec U_h\}$ and $\{\vec V_h\}$.
\begin{itemize}
\item[(1)] Since $\vec F(\vec U), \vec \nu(\vec U), \vec V(\vec U)$ are smooth with respect to $\vec U$ and $\vec U(\vec V)$ is smooth with respect to $\vec V$, the inequalities in (1) hold.

\item[(2)]  In \cite{Harten:1983a} Harten proved that $\nabla_{\vec U}^2 \eta(\vec U)$ is positive definite for all $\vec U$.
Assume that in (2) $\underline{\eta} = 0$. Then there exists a subsequence $\vec U_{h_n}$ satisfying $\nabla_{\vec U}^2 \eta(\vec U_{h_n}) \leq \frac1n$. Since $\{\vec U_{h_n}\}$ is a bounded set, there exists a  subsequence $\vec U_{h_{n_k}}$ which converges to some $\vec U$. Hence, we get $\nabla_{\vec U}^2 \eta(\vec U) = 0$, which is a contradiction.
\end{itemize}
\end{proof}

\begin{remark}
Detailed calculations to the proof of statement (1) for Lemma~\ref{lemma:usefulestimation} are presented in Appendix~\ref{section:Lipschitzcontinuity}.
The explicit expression of the lower bound $\underline{\eta}$ of $\nabla_{\vec U}^2 \eta(\vec U_h) $ is derived in Appendix~\ref{section:lowerbound}.
\end{remark}

\subsection{Entropy inequality}
\noindent In \cite{Harten-Lax-VanLeer:1983} Harten proved that the Godunov finite volume method based on the exact solution of the Riemann problem satisfies the entropy inequality (in one-dimensional case)
\begin{equation*}
\eta( \vec U_i(t_{n+1})) - \eta( \vec U_i(t_{n})) \leq \frac{\tau}{\Delta x}\left(-\vec q( \vec U^{\it RP}_{i+1/2}) +  \vec q( \vec U^{\it RP}_{i-1/2})\right),
\end{equation*}
where $\tau= t_{n+1}-t_n$ is the time-step.
Generalizing to multi-dimensions and writing it in the semi-discretization form give
\begin{equation} \label{eq:semi-discrete-entropy}
\int_{K} \frac{d }{d t} \eta_K(t) ~ d\vec x +  \sum_{\partial K \cap \partial L \neq \emptyset} \int_{\partial K \cap \partial L} \vec q(\vec U^{\it RP}_{\sigma_{K,L}}) \cdot \vec n_K  ~dS_{\vec x} \leq 0.
\end{equation}
Further, we get the equivalent weak form
\begin{equation} \label{eq:semi-discrete-entropy-phi}
\int_{\Omega} \phi  \frac{d }{d t} \eta_h ~ d\vec x -  \sum_{\sigma \in \inerface} \int_{\sigma} \vec q(\vec U^{\it RP}_{\sigma}) \cdot \vec n \defjump{\phi} ~dS_{\vec x} \leq 0
\end{equation}
with $\phi \in \mathcal{Q}_h(\Omega), \phi \geq 0$ and $\eta_h = \eta(\vec U_h)$.

On the other hand, Chen and Shu in \cite{Chen-Shu:2017} showed that the Godunov finite volume method is entropy stable, i.e.
\begin{align} \label{eq:semi-discrete-entropy-1}
& \int_{K} \frac{d }{d t} \eta_K(t) ~ d\vec x +  \sum_{\partial K \cap \partial L \neq \emptyset} \int_{\partial K \cap \partial L} \vec Q_{\sigma} \cdot \vec n_K  ~dS_{\vec x} \leq 0
\end{align}
with
\begin{align*}
	& \vec Q_{\sigma} \cdot \vec n_K= \defave{\vec \nu}_{\sigma}  \cdot (\vec F(\vec U^{\it RP}_{\sigma}) \cdot \vec n_K)-  \defave{\vec \psi}_{\sigma} \cdot \vec n_K.
\end{align*}
Writing explicitly the entropy dissipation we get
\begin{align} \label{eq:semi-discrete-entropy-2}
	& \int_{K} \frac{d }{d t} \eta_K(t) ~ d\vec x +  \sum_{\partial K \cap \partial L \neq \emptyset} \int_{\partial K \cap \partial L} \vec Q_{\sigma} \cdot \vec n_K  ~dS_{\vec x} = \frac12\sum_{\partial K \cap \partial L \neq \emptyset} \int_{\partial K \cap \partial L} r_{\sigma}  ~dS_{\vec x},
\end{align}
where
\begin{align*}
r_{\sigma}
& = \defjump{\vec \nu}_{\sigma}  \cdot (\vec F(\vec  U^{\it RP}_{\sigma})  \cdot \vec n)-  \defjump{\vec \psi}_{\sigma} \cdot \vec n \\
& = \defjump{\vec \nu}_{\sigma}  \cdot  \int_{-1/2}^{1/2} \bigg[ \vec F(\vec  U(\vec \nu^{\it RP}_{\sigma}) )  - \vec F(\vec U(\vec \nu(s)) \bigg]  \cdot \vec n ~ds := \defjump{\vec \nu}_{\sigma}  	\cdot \vec D_{\sigma} \cdot \defjump{\vec \nu}_{\sigma} ,
\end{align*}
and $\vec \nu(s) = \defave{\vec \nu}_{\sigma} + s \defjump{\vec \nu}_{\sigma} $.
Note that matrix $\vec D_{\sigma}$ is symmetric because $\frac{d \vec F_i}{d \vec \nu}, (i =1,\cdots, d)$ is symmetric.
Hence, \eqref{eq:semi-discrete-entropy-1} implies $r_{\sigma} \leq 0$, i.e., $\vec v \cdot D_{\sigma} \cdot \vec v \leq 0$ for any $\vec v \in \mathbb{R}^{d+2}$.
Combining Lemma~\ref{lemma:lemma-U-bounded} we have  $\vec v \cdot D_{\sigma} \cdot \vec v\leq \underline{d} < 0$ uniformly for $h \rightarrow 0$.
Therefore, it holds
\begin{align*}
	& \int_{K} \frac{d }{d t} \eta_K(t) ~ d\vec x +  \sum_{\partial K \cap \partial L \neq \emptyset} \int_{\partial K \cap \partial L} \vec Q_{\sigma} \cdot \vec n_K  ~dS_{\vec x} \\
\leq ~	&  \frac{\underline{d}}2 \sum_{\partial K \cap \partial L \neq \emptyset} \int_{\partial K \cap \partial L}  \| \defjump{\vec \nu}_{\sigma} \|_2^2  ~dS_{\vec x}
	 \leq \frac{\underline{d} \underline{\eta}}2 \sum_{\partial K \cap \partial L \neq \emptyset} \int_{\partial K \cap \partial L}  \|\defjump{\vec U}_{\sigma} \|_2^2 ~dS_{\vec x} ,
\end{align*}
which implies the weak BV estimate
\begin{equation}\label{eq:U-l2norm-bound}
\int_{0}^{\tau}\sum_{\sigma \in \inerface} \int_{\sigma} \|\defjump{\vec U}_{\sigma} \|_2^2 ~dS_{\vec x} dt\leq - \frac{1}{\underline{d} \underline{\eta}} \bigg( \int_{\Omega}  \eta_h(0) ~ d\vec x - \int_{\Omega} \eta_h(\tau) ~ d\vec x\bigg) \leq C,
\end{equation}
where $C$ is a positive constant depending on $\underline{d}, \underline{\eta}, \underline{\rho}, \overline{E}$.
Realizing that {\small$\sum\limits_{\sigma \in \inerface} \int_{\sigma} \|\defjump{\vec U}_{\sigma} \|_2^2/h ~dS_{\vec x} $} represents $H_0^1$-seminorm,
we have the following interpretation of  \eqref{eq:U-l2norm-bound}.
It tells us that
\begin{equation}
\|\vec U\|_{L^2(0,T;H^1_0(\Omega))} \lesssim h^{-1/2}.
\end{equation}

\subsection{Difference between $\vec U^{\it RP}$ and $\vec U_L, \vec U_R$} \label{section:difference}
\noindent For the sake of convenience, we write $ (u,v,w)$ in place of $\vec u$.
Taking $x$-direction as an example, our strategy is to study the difference between $\vec U(x/t;\vec U_L,\vec U_R)$  and $\vec U_K, (K = L,R)$, where $\vec U(x/t;\vec U_L,\vec U_R)$ is the solution along the line $x/t$ of the following Riemann problem
\begin{equation} \label{eq:Euler-RP}
\begin{cases}
\partial _t \vec U + \partial_{ x} \vec F_1 = 0, \quad  t > 0,\\
\vec U(x, 0)= \begin{cases}
\vec U_L, & \mbox{if} ~  x < 0,\\
\vec U_R, & \mbox{if} ~ x > 0,\\
\end{cases}
\end{cases}
\end{equation}
where $\vec U_{L}, \vec U_{R}$ are constant.
Once $\|\vec U(x/t; \vec U_L,\vec U_R)-\vec U_K\|$ is clearly known, then with the definition of $\vec U^{\it RP} = \vec U(0;\vec U_L,\vec U_R)$, we can directly estimate $\|\vec U^{\it RP} -  \vec U_K\|$.

We divide the $x$-$t$ domain into four parts separated by the left and right (non-linear) waves and the middle (linear-degenerated) waves.
Figure \ref{figure:RP} shows the possible wave patterns including left rarefaction, right rarefaction;  left shock, right shock; left rarefaction, right shock; left shock, right rarefaction.
Due to  Lemma~\ref{lemma:usefulestimation}, the estimate of $\|\vec U(x/t; \vec U_L,\vec U_R)-\vec U_K\|$ reduces to estimating $\|\vec V_{*M} - \vec V_K\|$ and $\| \vec V_{\it fan}-\vec V_K\|$ with $M = L,R$ and $K = L,R$.

Moreover, the Riemann Invariants and Rankine-Hugoniot conditions imply that $v = v_L, w = w_L$ before the middle wave and $v = v_R, w = w_R$ after the  middle wave.
Hence, we only need to study the change of $\rho, u, p$.

\begin{figure}[htbp]
	\setlength{\abovecaptionskip}{0.cm}
	\setlength{\belowcaptionskip}{-0.cm}
	\centering
	\begin{subfigure}{0.22\textwidth}
		\includegraphics[width=\textwidth]{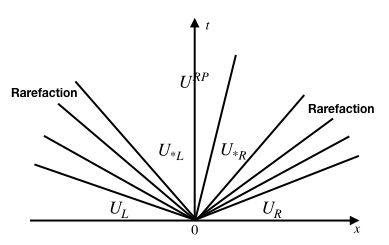}
		\caption{}
	\end{subfigure}	
	\begin{subfigure}{0.22\textwidth}
		\includegraphics[width=\textwidth]{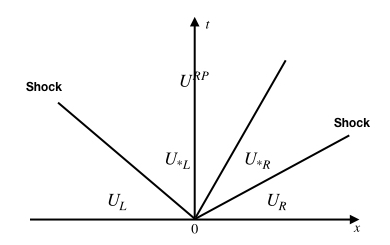}
		\caption{}
	\end{subfigure}	
	\begin{subfigure}{0.22\textwidth}
		\includegraphics[width=\textwidth]{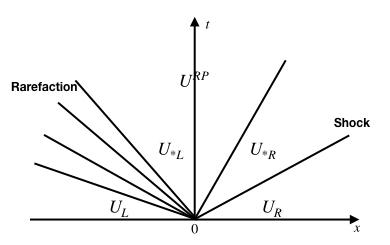}
		\caption{}
	\end{subfigure}	
	\begin{subfigure}{0.22\textwidth}
		\includegraphics[width=\textwidth]{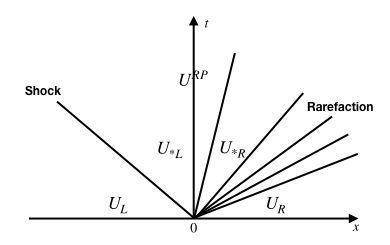}
		\caption{}
	\end{subfigure}	
	\caption{\small{Possible wave patterns in the solution of the Riemann problem: (a) left rarefaction, right rarefaction;
			(b) left shock, right shock;
			(c) left rarefaction, right shock;
			(d) left shock, right rarefaction. }}\label{figure:RP}
\end{figure}

\subsubsection{Preliminaries}
\noindent Here we list some results which will be used for the estimates of $\|\vec V_{*M} - \vec V_K\|$ and $\| \vec V_{\it fan}-\vec V_K\|$, cf.  \cite{Feistauer, Toro:2009}.

\begin{proposition}(solution for $p_*$ and $u_*$)
	The solution for pressure $p_*$ of the Riemann problem \eqref{eq:Euler-RP} is given by the root of the algebraic equation
	\begin{equation}
	f(p,\vec U_L,\vec U_R) := f_L(p,\vec U_L) + f_R(p,\vec U_R) + \Delta u= 0, \quad \Delta u = u_R - u_L,
	\end{equation}
	where the function $f_K, (K = L \mbox{ or } R)$ is given by
	\begin{equation*}
		f_K(p, \vec U_K) =
		\begin{cases}
			f_s(p, \vec U_K), & \mbox{if }~ p>p_K ~ \mbox{shock},\\
			f_r(p, \vec U_K), & \mbox{if }~ p\leq p_K ~ \mbox{rarefaction}\\
		\end{cases}
	\end{equation*}
	with
	\begin{align}
		& f_s(p, \vec U_K)  = (p-p_K) \left(\frac{A_K}{p+B_K}\right)^{1/2}, \quad
		A_K = \frac{2}{(\gamma+1)\rho_K}, \label{eq:RPsolver-fshock}\\
		& f_r(p, \vec U_K) = \frac{2a_K}{(\gamma-1)} \left[\left(\frac{p}{p_K}\right)^{\frac{\gamma-1}{2\gamma}} - 1\right], \quad
		B_K = \frac{\gamma-1}{\gamma+1} p_K. \label{eq:RPsolver-frarefaction}
	\end{align}
	The solution for the velocity $u_*$ in the star region is
	\begin{equation}
	u_* = \frac12 (u_L+u_R) +  \frac12 [f_R(p_*) - f_L(p_*)].
	\end{equation}
\end{proposition}

\begin{remark}
	More precisely, $u_*, p_*$ satisfy
	\begin{equation}
	u_* = u_L - f_L(p_*, \vec U_L), \quad u_* = u_R - f_R(p_*, \vec U_R).
	\end{equation}
	For the right shock the density $\rho_{*R}$ is found to be
	\begin{equation}
	\rho_{*R} = \rho_R \left[ \frac{ \frac{p_*}{p_R} + \frac{\gamma-1}{\gamma+1} }{ \frac{\gamma-1}{\gamma+1} \frac{p_*}{p_R}  + 1} \right]
	\end{equation}
	and the shock speed is
	\begin{equation}
	s_{R} = u_R + a_R \left[ \frac{\gamma+1}{2\gamma} \frac{p_*}{p_R} + \frac{\gamma-1}{2\gamma}   \right]^{1/2}.
	\end{equation}
\end{remark}

\begin{proposition}\label{proposition:fk}
	It holds
	\begin{equation}
	f'_K > 0,  ~ f''_K < 0, \quad K = L,  R.
	\end{equation}
\end{proposition}

\subsubsection{Estimates}
\noindent We start by studying the solution inside the rarefaction fan.
\begin{lemma} \label{lemma:CRW}
	Assume $\vec U_L$ and $\vec U_R$ are connected with the left rarefaction.
	Denote $$\vec U_{\it Lfan}: =  (\rho_*, \rho_* u_*, E(\rho_*,u_*,p_*))^T$$ as the value inside left rarefaction fan.
	Then it holds
	\begin{align}
		& \rho_L \geq \rho_* \geq \rho_R,  \quad
		  u_L \leq u_* \leq u_R,  \quad
		 p_L \geq p_* \geq p_R.
	\end{align}
\end{lemma}

\begin{proof}
	Since $\vec U_L$ and $\vec U_R$ are connected with the left rarefaction,
	then it holds
	\begin{align*}
		&  p_L \geq p_* \geq p_R, \quad 
		u_* = u_L - f_r(p_*, \vec U_L), \quad  u_R = u_L - f_r(p_R, \vec U_L),
	\end{align*}
	where $f_r$ is defined in \eqref{eq:RPsolver-frarefaction}.
	From $p_L \geq p_* \geq p_R$ and $f_r(p_L,\vec U_L) = 0, ~f_r' > 0$ we have
	\begin{equation*}
		u_L \leq u_* \leq u_R.
	\end{equation*}
	On the other hand, with the Riemann Invariant $S:=\ln(p/\rho^{\gamma})$ we have
	\begin{equation*}
		\rho_L \geq \rho_* \geq \rho_R,
	\end{equation*}
	which concludes the proof.
\end{proof}

Consequently, the value inside left-rarefaction-fan ($\vec V_{Lfan}$) can be bounded by the left and right side values ($\vec V_L,\vec V_R$).
Clearly, it is also true for the right rarefaction wave.
Thus, $\| \vec V_{\it fan}-\vec V_K\|$ can be controlled by $\|\vec V_{*M} - \vec V_K\|$.
Hence, we concentrate on  $\|\vec V_{*M} - \vec V_K\|$
and analyze it case by case:
\begin{multicols}{2}
\begin{itemize}
	\item[(a)] left rarefaction, right rarefaction;
	\item[(b)] left shock, right shock;
	\item[(c)] left rarefaction, right shock;
	\item[(d)] left shock, right rarefaction.
\end{itemize}
\end{multicols}

\begin{lemma}[Left rarefaction, right rarefaction]\label{lemma:RPRR}
	Assume the initial data $\vec U_L, \vec U_R$ generate left and right rarefaction waves.
	Then it holds
	\begin{align}
		& |u_* - u_K| \leq \defjump{u},  \label{lemma:RPRR:jumpu}\\
		& 0 \leq p_K - p_{*}  \leq \rho_K a_K  \defjump{u}, \label{lemma:RPRR:jumpp}\\
		&  |\rho_{*K} - \rho_{M}| \leq \rho_K a_K^{-1}   \defjump{u} + |\defjump{\rho}|, \label{lemma:RPRR:jumprho}
	\end{align}
	where $K = L, R$ and $M = L, R$.
\end{lemma}

\begin{proof}
	Since the initial data $\vec U_L, \vec U_R$ generate left and right rarefaction waves, we have
	\begin{align}
		& p_* \leq p_L, \quad p_* \leq p_R \label{eq:RPRR-condition-p}
	\end{align}
	and
	\begin{equation}
	u_* = u_L - f_r(p_*,\vec U_L), \quad u_* = u_R + f_r(p_*,\vec U_R).
	\end{equation}
	Because of $f_r(p_K, \vec U_K) = 0, (K = L,R)$ and $f'_{r} > 0$,  it holds
	\begin{equation}
	u_* \geq u_L - f_r(p_L,\vec U_L) = u_L, \quad u_* \leq u_R + f_r(p_R,\vec U_R) = u_R, \label{eq:RPRR-condition-u}
	\end{equation}
	which gives
	\begin{equation}
	0 \leq u_* - u_L \leq \defjump{u}, \quad 0 \leq u_R - u_* \leq \defjump{u}.
	\end{equation}

	Consider the left rarefaction wave.
	The Riemann Invariant
	\begin{equation*}
		u_L + \frac{2a_L}{\gamma-1} = u_* + \frac{2a_{*L}}{\gamma-1}
	\end{equation*}
	gives
	\begin{align} \label{eq:RPRR:IR-a}
		& 0 \leq a_{L} - a_{*L} = \frac{\gamma-1}{2} (u_*-u_L)  \leq \frac{\gamma-1}{2} \defjump{u}.
	\end{align}
	On the other hand, with the help of the Riemann Invariant $S$ we have
	{\small
	\begin{align*}
		& a_{L} - a_{*L}
		= \gamma^{1/2} e^{S_L/(2\gamma)} \left(p_L^{(\gamma-1)/(2\gamma)} - p_*^{(\gamma-1)/(2\gamma)} \right)
		 = \gamma^{1/2} e^{S_L/(2\gamma)}   \frac{\gamma-1}{2\gamma}  p_1^{-(\gamma+1)/(2\gamma)}  (p_L-p_*)\\
\geq 	&~  \gamma^{1/2} e^{S_L/(2\gamma)}  \cdot \frac{\gamma-1}{2\gamma}  \cdot p_L^{-(\gamma+1)/(2\gamma)} ) (p_L-p_*)
		 =  \frac{\gamma^{1/(\gamma-1)}  (\gamma-1)}2  e^{S_L/(\gamma-1)}  a_L^{-(\gamma+1)/(\gamma-1)}  (p_L-p_*)
	\end{align*}	}%
and
	{\small
	\begin{align*}
		&	a_{L} - a_{*L}
		= \gamma^{1/2} e^{S_L/2} \left( \rho_L^{(\gamma-1)/2} -  \rho_{*L}^{(\gamma-1)/2} \right)
		 = \gamma^{1/2} e^{S_L/2}\frac{\gamma-1}{2}  \rho_1^{(\gamma-3)/2}   ( \rho_L - \rho_{*L}) \\
\geq	&~  \gamma^{1/2} e^{S_L/2}\frac{\gamma-1}{2}  \rho_L^{(\gamma-3)/2}   ( \rho_L - \rho_{*L})
		 	= \frac{\gamma^{1/(\gamma-1)} (\gamma-1) }2  e^{S_L/(\gamma-1)} a_L^{(\gamma-3)/(\gamma-1)}   ( \rho_L - \rho_{*L}),
	\end{align*}}%
	where $p_1 \in (p_*, p_L), ~ \rho_1 \in (\rho_{*L}, \rho_L)$.
	Hence, we obtain
	\begin{align*}
		& 0 \leq p_L - p_{*}
		\leq \gamma^{-1/(\gamma-1)}  \cdot e^{-S_L/(\gamma-1)}  a_L^{(\gamma+1)/(\gamma-1)} \defjump{u} = \rho_L a_L \defjump{u}, \\
		&0 \leq \rho_L - \rho_{*L}
		\leq  \gamma^{-1/(\gamma-1)} e^{-S_L/(\gamma-1)} a_L^{(3-\gamma)/(\gamma-1)}  \defjump{u} = \rho_L a_L^{-1} \defjump{u}.
	\end{align*}
	
	Analogously, for the right rarefaction wave  we obtain
	\begin{align*}
		& 0 \leq p_R - p_{*}  \leq \rho_R a_R  \defjump{u}, \quad
		 0 \leq \rho_R - \rho_{*R}  \leq \rho_R a_R^{-1}  \defjump{u}.
	\end{align*}
	Consequently,
	\begin{align*}
		& |\rho_{R} - \rho_{*L}| \leq |\rho_{R} - \rho_{L}| + |\rho_{L} - \rho_{*L}| \leq \rho_L a_L^{-1}  \defjump{u} + |\defjump{\rho}|, \\
		& |\rho_{L} - \rho_{*R}| \leq |\rho_{R} - \rho_{L}| + |\rho_{R} - \rho_{*R}| \leq \rho_R a_R^{-1}  \defjump{u} + |\defjump{\rho}|,
	\end{align*}
	which concludes the proof.
\end{proof}

\begin{lemma}[Left shock, right shock]\label{lemma:RPSS}
	Assume the initial data $\vec U_L, \vec U_R$ generate left and right shock waves.
	Then it holds
	\begin{align}
		& |u_* - u_K| < |\defjump{u}|,  \label{lemma:RPSS:jumpu}\\
		& 0 < p_* - p_{K}  <  \left(\gamma\rho_K |\defjump{u}|+ \rho_Ka_K\right)  |\defjump{u}|, \label{lemma:RPSS:jumpp}\\
		&  |\rho_{*K} - \rho_{M}| < a_K^{-2} \left(\gamma\rho_K |\defjump{u}|+ \rho_Ka_K\right)  |\defjump{u}| + |\defjump{\rho}|, \label{lemma:RPSS:jumprho}
	\end{align}
	where $K = L, R$ and $M = L, R$.
\end{lemma}

\begin{proof}
	Since the initial data $\vec U_L, \vec U_R$ generate left and right shock waves,
	we have
	\begin{align}
		& p_* > p_L, \quad p_* > p_R, \label{eq:RPSS-condition-p} 
	\end{align}
	where $s_L, s_R$ are the velocities of the left and right shocks respectively.
	According to
	\begin{equation*}
		u_* = u_L - f_s(p_*,\vec U_L), \quad u_* = u_R + f_s(p_*,\vec U_R)
	\end{equation*}
	and $f_s(p_K,\vec U_K) = 0, (K = L,R), ~ f'_{s} > 0$,  we obtain
	\begin{equation}
	u_* < u_L- f_s(p_L,\vec U_L) = u_L, \quad u_* > u_R + f_s(p_*, \vec U_R) = u_R, \label{eq:RPSS-condition-u}
	\end{equation}
	which means
	\begin{equation}
	0 < u_L - u_* < -\defjump{u}, \quad 0 < u_* - u_R < -\defjump{u}.
	\end{equation}

	Consider the right shock wave.
	With $u_* = u_R + f_s(p_*, \vec U_R)$ and
	\begin{equation*}
		s_R = u_R + a_R \left[ \frac{(\gamma+1)}{2\gamma}\frac{p_*}{p_R} + \frac{(\gamma-1)}{2\gamma} \right]^{1/2} <  u_R + a_R \left(\frac{p_*}{p_R} \right)^{1/2},
	\end{equation*}
	consequently we derive after some algebraic manipulations
	\begin{align*}
		p_* - p_R
		& = \left[ \frac{ (\gamma+1)p_* + (\gamma-1)p_R}{  2/\rho_{R} } \right]^{1/2}  (u_* - u_R ) = \rho_R(s_R-u_R)(u_* - u_R )  \\
		& <  \rho_R(s_R-u_R) |\defjump{u}|
		 < \rho_Ra_R |\defjump{u}|\left(\frac{p_*}{p_R} \right)^{1/2}.
	\end{align*}
	Thus,
	\begin{equation*}
		\frac{p_*}{p_R}  - \frac{\rho_Ra_R |\defjump{u}|}{p_R} \left(\frac{p_*}{p_R} \right)^{1/2} - 1 <0,
	\end{equation*}
	which gives
	\begin{align*}
		\left(\frac{p_*}{p_R} \right)^{1/2}
		& < \frac{1}{2} \left[ \frac{\rho_Ra_R |\defjump{u}|}{p_R}  + \sqrt{\left(\frac{\rho_Ra_R |\defjump{u}|}{p_R} \right)^2+4} \right]
		 < \frac{\rho_Ra_R |\defjump{u}|}{p_R} + 2.
	\end{align*}
	Hence, we obtain
	\begin{align*}
		p_* - p_R
		 <  \left(\gamma\rho_R |\defjump{u}|+ \rho_Ra_R\right) |\defjump{u}|.
	\end{align*}
	On the other hand, using
	\begin{equation*}
		\rho_{*R} = \rho_{R} \left[ \frac{  \frac{p_*}{p_R} + \frac{\gamma-1}{\gamma+1} }{ \frac{\gamma-1}{\gamma+1}\frac{p_*}{p_R} + 1 } \right]
	\end{equation*}
	we obtain
	\begin{align*}
		\rho_{*R} - \rho_{R}
		& =   \left[ \frac{  2\rho_{R} }{ (\gamma-1)p_* + (\gamma+1)p_R} \right]  (p_* - p_R)  < \left[ \frac{  2\rho_{R} }{ (\gamma-1)p_R + (\gamma+1)p_R} \right]  (p_* - p_R)  \\
		& = a_R^{-2}  (p_* - p_R) \leq a_R^{-2} \left(\gamma\rho_R |\defjump{u}|+ \rho_Ra_R\right)  |\defjump{u}|,
	\end{align*}
	which yields $\rho_{*R} - \rho_{R}   > 0$.

	Analogously analyzing the left shock wave we obtain
	\begin{align*}
		&0 < p_* - p_L < \left(\gamma\rho_L |\defjump{u}|+ \rho_La_L\right)  |\defjump{u}|,  \\
		&0 < \rho_{*L} - \rho_{L}  < a_L^{-2} \left(\gamma\rho_L |\defjump{u}|+ \rho_La_L\right) |\defjump{u}|.
	\end{align*}
	Further, we have
	\begin{align*}
		& | \rho_{*L} - \rho_{R} | < a_L^{-2}  \left(\gamma\rho_L |\defjump{u}|+ \rho_La_L\right)  |\defjump{u}| + |\defjump{\rho}|,\\
		& | \rho_{*R} - \rho_{L} | < a_R^{-2}  \left(\gamma\rho_R |\defjump{u}|+ \rho_Ra_R\right)  |\defjump{u}| + |\defjump{\rho}|,
	\end{align*}
	which concludes the proof.
\end{proof}

\begin{lemma}[Left rarefaction, right shock]\label{lemma:RPRS}
	Assume the initial data $\vec U_L, \vec U_R$ generate left rarefaction waves and right shock waves.
	Then it holds
	\begin{align}
		& 0 \leq u_* - u_K < (\rho_Ra_R)^{-1} |\defjump{p}| + |\defjump{u}|,  \label{lemma:RPRS:jumpu}\\
		& |p_K - p_{*} | < |\defjump{p}|, \label{lemma:RPRS:jumpp}\\
		&  |\rho_{*K} - \rho_{M}| \leq \left(p_K/p_R \right)^{(\gamma-1)/\gamma} a_K^{-2}  |\defjump{p}| + |\defjump{\rho}|, \label{lemma:RPRS:jumprho}
	\end{align}
	where $K = L, R$ and $M = L, R$.
\end{lemma}

\begin{proof}
	Since the initial data $\vec U_L, \vec U_R$ generate left rarefaction waves and right shock waves,
	we have
	\begin{align}
	& p_* \leq p_L, \quad p_* > p_R, \quad
 		u_L - a_L  \leq  u_* - a_{*L}, \quad  u_* + a_{*R} > S_R > u_R + a_{R}
	\end{align}
	and
	\begin{equation}
	u_* = u_L - f_r(p_*,\vec U_L), \quad u_* = u_R + f_s(p_*,\vec U_R).
	\end{equation}
	This leads to
	\begin{equation}
	0 \leq p_L - p_* \leq -\defjump{p}, \quad 0 < p_* - p_R \leq -\defjump{p}
	\end{equation}
	and
	\begin{equation}
	u_* \geq u_L - f_r(p_L,\vec U_L) = u_L, \quad u_* > u_R + f_r(p_R,\vec U_R) = u_R.
	\end{equation}

	Consider the right shock wave.
	Realizing that
	\begin{equation*}
		\rho_{*R} = \rho_{R} \left[ \frac{  \frac{p_*}{p_R} + \frac{\gamma-1}{\gamma+1} }{ \frac{\gamma-1}{\gamma+1}\frac{p_*}{p_R} + 1 } \right],
	\end{equation*}
	we obtain
	\begin{align*}
		\rho_{*R} - \rho_{R}
		& =   \left[ \frac{  2\rho_{R} }{ (\gamma-1)p_* + (\gamma+1)p_R} \right]  (p_* - p_R) \\
		& < \left[ \frac{  2\rho_{R} }{ (\gamma-1)p_R + (\gamma+1)p_R} \right]  (p_* - p_R)  = a_R^{-2}  (p_* - p_R) \leq a_R^{-2} |\defjump{p}|,
	\end{align*}
	which also implies $\rho_{*R} - \rho_{R}   > 0$.
	Since $u_*$ satisfies $u_* = u_R + f_s(p_*,U_R)$, we have
	{\small
	\begin{align*}
		u_* - u_R
		& = \left[ \frac{  2/\rho_{R} }{ (\gamma+1)p_* + (\gamma-1)p_R} \right]^{1/2}  (p_* - p_R)
		 < \left[ \frac{  2/\rho_{R} }{ (\gamma+1)p_R + (\gamma-1)p_R} \right]^{1/2}  (p_* - p_R) \\
		& = (\gamma\rho_{R} p_R)^{-1/2}  (p_* - p_R) \leq (\rho_Ra_R)^{-1} |\defjump{p}|.
	\end{align*}}%

	Consider the left rarefaction wave.
	Clearly,
	\begin{align*}
		u_* - u_L & < |u_* - u_R| + |\defjump{u}| < (\rho_Ra_R)^{-1} |\defjump{p}| + |\defjump{u}|.
	\end{align*}
	On the other hand, with the help of the Riemann Invariant $S$ we have
	{\small
	\begin{align*}
		\rho_{L} - \rho_{*L}
		& = \exp^{-S_L/\gamma} (p^{1/\gamma}_L  - p^{1/\gamma}_{*} ) > 0, \\
		\rho_{L} - \rho_{*L}
		& = \exp^{-S_L/\gamma} \cdot (1/\gamma) \cdot p_1^{(1-\gamma)/\gamma} (p_L  - p_{*})
		 < \exp^{-S_L/\gamma} \cdot (1/\gamma) \cdot p_R^{(1-\gamma)/\gamma} (p_L  - p_{*})   \\
		 & < \gamma^{-1}\exp^{-S_L/\gamma}  \cdot p_R^{(1-\gamma)/\gamma} |\defjump{p}|
		 = \left(\frac{\rho_La_L^2}{\rho_Ra_R^2} \right)^{(\gamma-1)/\gamma} a_L^{-2} =  \left(\frac{p_L}{p_R} \right)^{(\gamma-1)/\gamma} a_L^{-2} |\defjump{p}|,
	\end{align*}}%
	where $p_1 \in (p_*,p_L) \subset (p_R, p_L)$.
	Further, we have
	\begin{align*}
		& | \rho_{*L} - \rho_{R} | < \left(p_L/p_R \right)^{(\gamma-1)/\gamma} a_L^{-2}  |\defjump{p}| + |\defjump{\rho}|,\\
		& | \rho_{*R} - \rho_{L} | <  a_R^{-2} |\defjump{p}|+ |\defjump{\rho}|,
	\end{align*}
	which concludes the proof.
\end{proof}

Analogously to Lemma~\ref{lemma:RPRS} the following result holds.
\begin{lemma}[Left shock, right rarefaction]\label{lemma:RPSR}
	Assume the initial data $\vec U_L, \vec U_R$ generate left shock waves and right rarefaction  waves.
	Then it holds
	\begin{align}
		& 0 \leq u_* - u_K < (\rho_La_L)^{-1} |\defjump{p}| + |\defjump{u}|,  \label{lemma:RPSR:jumpu}\\
		& |p_K - p_{*} | < |\defjump{p}|, \label{lemma:RPSR:jumpp}\\
		&  |\rho_{*K} - \rho_{M}| \leq \left(p_K/p_L \right)^{(\gamma-1)/\gamma} a_K^{-2}  |\defjump{p}| + |\defjump{\rho}|, \label{lemma:RPSR:jumprho}
	\end{align}
	where $K = L, R$ and $M = L, R$.
\end{lemma}

Combining Lemma~\ref{lemma:usefulestimation} and Lemma~\ref{lemma:RPRR} - \ref{lemma:RPSR}  we finally obtain the following bounds for the Riemann problem solution.
\begin{lemma} \label{lemma:usefulestimation2}
	Under Assumption~\ref{assumption}
	it holds
	\begin{equation}
		\|\vec U_L - \vec U^{\it RP}_{\sigma}\| \lesssim \|\defjump{\vec U_h}\|, ~ \|\vec U_R - \vec U^{\it RP}_{\sigma}\| \lesssim \|\defjump{\vec U_h}\|
	\end{equation}
	with $\sigma := L|R$.
\end{lemma}

\begin{remark}
Lemma~\ref{lemma:usefulestimation2} implies $\|\vec U^{\it RP}_{\sigma}\| \lesssim \|U_L\| + \|U_R\| \leq C$ with $C = C(\underline{\rho}, \overline{E})$.
\end{remark}

\subsection{Consistency}
\noindent
The aim of this section is to prove the consistency of the finite volume method \eqref{eq:semi-discrete-RPflux-phi}.
\begin{theorem}(Consistency Formulation) \label{theorem:consistency}
	\noindent Let $\vec U_h$ be the unique solution of the finite volume scheme \eqref{eq:semi-discrete-RPflux-phi} on the time interval $[0,T]$ with the initial data $\vec U_{0,h}$. Under the Assumption~\ref{assumption}
	we have the following results for all $\tau \in (0,T)$:
	\begin{itemize}
	\item for all $\phi \in C^1([0,T]\times \overline{\Omega})$
	\begin{equation}\label{eq:consistency-density}
		\left[ \int_{\Omega} \rho_h \phi ~d \vec x  \right]_{t = 0}^{t = \tau} = \int_{0}^{\tau} \int_{\Omega} \rho_h \partial_t \phi +   \vec m_h \cdot \nabla_{\vec x} \phi   ~d \vec x dt + \int_{0}^{\tau} e_{\rho,h}(t,\phi) ~dt;
	\end{equation}
	
	\item for all $\vec \phi \in C^1([0,T]\times \overline{\Omega}; \mathbb{R}^d)$
	\begin{align}\label{eq:consistency-momentum}
	\left[ \int_{\Omega} \vec m_h \vec \phi ~d \vec x  \right]_{t = 0}^{t = \tau} =
	&\int_{0}^{\tau} \int_{\Omega} \vec m_h \partial_t \vec \phi +   \frac{\vec m_h \otimes\vec m_h }{\rho_h} : \nabla_{\vec x} \vec \phi \nonumber\\
	& \hspace{1cm} + p_h  \rm{div} _{\vec x} \vec \phi   ~d \vec x dt + \int_{0}^{\tau} e_{\vec m,h}(t,\vec \phi) ~dt;
	\end{align}

	\item for all $\phi \in C^1([0,T]\times \overline{\Omega}),\, \phi \geq 0$
	\begin{equation}\label{eq:consistency-entropy}
	\left[ \int_{\Omega} \eta_h \phi ~d \vec x  \right]_{t = 0}^{t = \tau} \leq \int_{0}^{\tau} \int_{\Omega} \eta_h \partial_t \phi + \vec q_h \cdot \nabla_{\vec x} \phi ~d \vec x dt + \int_{0}^{\tau} e_{\eta,h}(t,\phi) ~dt;
	\end{equation}
	
	\item
	\begin{equation}\label{eq:consistency-energy}
		\int_{\Omega} E_h(\tau) ~d \vec x = \int_{\Omega} E_{0,h} ~d \vec x.
	\end{equation}
	\end{itemize}

	The errors $e_{j,h}, (j = \rho, \vec m, \eta)$ are bounded by
	\begin{equation}\label{eq:consistency-errors}
	\|e_{j,h}\|_{L^1(0,T)} \lesssim h^{1/2} \| \phi \|_{C^1([0,T]\times \overline{\Omega})} \left(\int_{0}^{\tau} \sum_{\sigma \in \inerface} \int_{\sigma} \|\defjump{U_h}_{\sigma}\|_2^2 ~ dS_{\vec x}  dt\right)^{1/2}.
	\end{equation}
\end{theorem}

\begin{proof}
	Energy conservation for \eqref{eq:semi-discrete-RPflux-phi} follows directly by integrating the discrete energy equation in time and applying the boundary conditions.
	We proceed by proving  \eqref{eq:consistency-density} - \eqref{eq:consistency-entropy}.
	
	\textbf{Step 1}: We prove \eqref{eq:consistency-density} and \eqref{eq:consistency-momentum} by showing
	\begin{align}
	& \left[ \int_{\Omega} \vec U_h \vec \phi ~d \vec x  \right]_{t = 0}^{t = \tau} = \int_{0}^{\tau} \int_{\Omega} \vec U_h \partial_t \vec \phi +   \vec F: \nabla_{\vec x} \vec \phi   ~d \vec x dt + \int_{0}^{\tau} e_{h}(t,\phi) ~dt
	\end{align}
	for all $\vec \phi \in C^1([0,T]\times \overline{\Omega}; \mathbb{R}^{d+2})$ with
	\begin{align} \label{eq:theorem-errors-U}
		& \|e_h\|_{L^1(0,T)} \lesssim h^{1/2} \| \phi \|_{C^1([0,T]\times \overline{\Omega})} \left(\int_{0}^{T} \sum_{\sigma \in \inerface} \int_{\sigma} \|\defjump{U_h}_{\sigma}\|^2 ~ dS_{\vec x}  dt\right)^{1/2}.
	\end{align}

	Realizing that
	\begin{align}
	& \defjump{ab} = \defave{a} \defjump{b} + \defjump{a} \defave{b},
	\end{align}
	we obtain after some manipulations
	\begin{align*}
	& \int_{\Omega} \vec F_h:\nabla_{\vec x}  \vec \phi ~ d\vec x
	=  \sum_{K} \int_{K}  \vec F_h:\nabla_{\vec x}  \vec \phi ~ d\vec x\\
	=  & \sum_{K} \int_{\partial K}  \vec F_h \cdot \vec \phi  \cdot \vec n_K   ~ d\vec x
	=   - \sum_{\sigma \in \inerface} \int_{\sigma} \defjump{\vec F_h} \cdot \vec \phi \cdot \vec n~ dS_{\vec x}\\
	= &   - \sum_{\sigma \in \inerface} \int_{\sigma} \bigg(\defjump{\vec F_h} \cdot \left( \vec \phi - \defave{\projection{\vec \phi}} \right) + \defjump{\vec F_h}  \cdot  \defave{\projection{\vec \phi}} \bigg)  \cdot \vec n ~ dS_{\vec x} \\
	= &   - \sum_{\sigma \in \inerface} \int_{\sigma} \bigg(\defjump{\vec F_h} \cdot \left( \vec \phi - \defave{\projection{\vec \phi}} \right) - \defjump{\projection{\vec \phi}}  \cdot  \defave{\vec F_h}  + \defjump{\vec F_h \cdot \projection{\vec \phi}} \bigg)  \cdot \vec n ~ dS_{\vec x} \\
	= & \sum_{\sigma \in \inerface} \int_{\sigma}  \vec  F(\vec U^{\it RP}_{\sigma}) \cdot \vec n \cdot \defjump{\projection{\vec \phi}} ~ dS_{\vec x}   - \sum_{\sigma \in \inerface} \int_{\sigma} \defjump{\vec F_h}\cdot \vec n \cdot \left( \vec \phi - \defave{\projection{\vec \phi}} \right)~ dS_{\vec x} \\
	& + \sum_{\sigma \in \inerface} \int_{\sigma}  \big( \defave{\vec F_h}  -  \vec  F(\vec U^{\it RP}_{\sigma}) \big) \cdot \vec n \cdot  \defjump{\projection{\vec \phi}} ~ dS_{\vec x}.
	\end{align*}
	For the last equality, we have used the Gauss theorem and the no-flux or periodic boundary condition
	\begin{equation} \label{eq:divergence-theorem}
	\sum_{\sigma \in \inerface} \int_{\sigma} \defjump{\vec F_h \cdot \projection{\vec \phi}} \cdot \vec n ~dS_{\vec x} = \int_{\partial\Omega} \vec F_h\cdot  \projection{\vec \phi} \cdot \vec n ~dS_{\vec x}  = 0.
	\end{equation}
	
	Let us now consider the error terms
	\begin{align*}
	& e_1 = \sum_{\sigma \in \inerface} \int_{\sigma} \defjump{\vec F_h}\cdot \vec n \cdot \left( \vec \phi - \defave{\projection{\vec \phi}} \right)~ dS_{\vec x}, \\
	& e_2 = -\sum_{\sigma \in \inerface} \int_{\sigma}   \big( \defave{\vec F_h}  -  \vec  F(\vec U^{\it RP}_{\sigma}) \big) \cdot \vec n \cdot  \defjump{\projection{\vec \phi}} ~ dS_{\vec x}.
	\end{align*}
	Applying Lemma~\ref{lemma:usefulestimation}, i.e. $\| \defjump{\vec F_h} \cdot \vec n \|  \lesssim  |  \defjump{\vec U_h}|$ and the fact
	\begin{align*}
	& \| \vec \phi - \defave{ \projection{\vec \phi}} \| \lesssim h \normcp{1}{\vec \phi}
	\quad \mbox{for all } \vec x \in \sigma \in \inerface,
	\end{align*}
	we have the following estimate
	\begin{align*}
	| e_1 | & \lesssim h \normcp{1}{\vec \phi}\sum_{\sigma \in \inerface} \int_{\sigma} \| \defjump{\vec F_h}\cdot \vec n\| ~ dS_{\vec x} \lesssim h \normcp{1}{\vec \phi}\sum_{\sigma \in \inerface} \int_{\sigma}   \| \defjump{\vec U_h}\| ~ dS_{\vec x} \\
	& \lesssim h \normcp{1}{\vec \phi} \left(\sum_{\sigma \in \inerface} \int_{\sigma}  \| \defjump{\vec U_h}\|^2 ~ dS_{\vec x}\right)^{1/2}  \left(\sum_{\sigma \in \inerface}\int_{\sigma} 1 ~ dS_{\vec x} \right)^{1/2}  \\
	& \lesssim h^{1/2}\normcp{1}{\vec \phi}  \left(\sum_{\sigma \in \inerface} \int_{\sigma}   \| \defjump{\vec U_h}\|^2 ~ dS_{\vec x}\right)^{1/2}.
	\end{align*}
	Realizing that
	\begin{equation*}
		\| \defjump{ \projection{\phi}} \|  \lesssim h \normcp{1}{\vec \phi},
	\end{equation*}
	we derive
	\begin{align*}
	|e_2| &   \leq h \normcp{1}{\vec \phi} \sum_{\sigma \in \inerface} \int_{\sigma}    \big\| (\defave{\vec F_h}  -   \vec F(\vec U^{\it RP}_{\sigma}))\cdot \vec n \big\| ~ dS_{\vec x}\\
	&  \lesssim h\normcp{1}{\phi}  \sum_{\sigma:= L|R \in \inerface} \int_{\sigma}    \big\| (\vec F(\vec U_L)  -   \vec F(\vec U^{\it RP}_{\sigma}))\cdot \vec n \big\| + \big\| (\vec F(\vec U_R)  -   \vec F(\vec U^{\it RP}_{\sigma}))\cdot \vec n \big\| ~ dS_{\vec x}\\
	&  \lesssim h \normcp{1}{\vec \phi} \sum_{\sigma:= L|R \in \inerface} \int_{\sigma}  \| \vec U_L-\vec U^{\it RP}_{\sigma}\|  +   \| \vec U_R-\vec U^{\it RP}_{\sigma}\|   ~ dS_{\vec x} \\
	& \lesssim h \normcp{1}{\vec \phi} \sum_{\sigma \in \inerface} \int_{\sigma}  \| \defjump{\vec U_h}\|    ~ dS_{\vec x}
	  \lesssim h^{1/2}\normcp{1}{\vec \phi}  \sum_{\sigma \in \inerface} \int_{\sigma}   \| \defjump{\vec U_h}\|^2 ~ dS_{\vec x}.
	\end{align*}
		
	Hence, we can obtain
	\begin{align*}
	\left[ \int_{\Omega} \vec U_h \cdot \vec \phi ~d \vec x  \right]_{t = 0}^{t = \tau}
	=  & \int_{0}^{\tau} \int_{\Omega} \frac{d}{dt} (\vec U_h \cdot \vec \phi) ~d \vec x
	= \int_{0}^{\tau} \int_{\Omega} \vec U_h \cdot \partial_t \vec \phi + \vec \phi \cdot \frac{d}{dt} \vec U_h ~d {\vec x} dt  \\
	= & \int_{0}^{\tau} \int_{\Omega} \vec U_h \cdot \partial_t \vec \phi ~d {\vec x} dt  +   \int_{0}^{\tau}  \sum_{\sigma \in \inerface} \int_{\sigma} \vec F(\vec U^{\it RP}_{\sigma}) \cdot \vec n \defjump{\phi} ~dS_{\vec  x} \\
	= & \int_{0}^{\tau} \int_{\Omega} \vec U_h \cdot \partial_t \vec \phi + \vec F_h : \nabla_{\vec x}  \vec \phi ~d \vec x dt + \int_{0}^{\tau} e_h(t,\vec \phi) ~dt,
	\end{align*}
	where $e_h = e_1 + e_2$ satisfies \eqref{eq:theorem-errors-U} with the help of $\|\vec U\|^2 \lesssim \|\vec U\|_2^2$.
	
	\textbf{Step 2}:
	Indeed, using the same techniques as Step 1 to analyze \eqref{eq:semi-discrete-entropy-phi} we obtain \eqref{eq:consistency-entropy}, which concludes the proof.	
\end{proof}

\section{Convergence}
\noindent
In order to keep the paper self-contained we present the definition of a dissipative measure-valued solution for the Euler system \eqref{eq:multiD-Euler}, \eqref{eq:Euler-entropy-inequality}, cf. \cite{Feireisl-Lukacova-Mizerova-She:2021}.
\begin{definition} \label{def:DMV}
Let $\Omega \subset \mathbb{R}^{d}$ be a bounded domain.
A parametrized probability measure $\{\mathcal{V}_{t,\vec x}\}_{(t,\vec x)\in (0,T)\times\Omega}$,
\begin{equation*}
\mathcal{V}_{t,\vec x} \in L^{\infty}((0,T)\times\Omega, \mathcal{P}(\mathbb{R}^{d+2})), \quad
\mathbb{R}^{d+2} =\{ (\tilde{\rho}, \tilde{\vec m}, \tilde{\eta}):  \tilde{\rho} \in \mathbb{R},~ \tilde{\vec m} \in \mathbb{R}^d, ~ \tilde{\eta} \in \mathbb{R} \}
\end{equation*}
is called a dissipative measure-valued (DMV) solution of the Euler system \eqref{eq:multiD-Euler}, \eqref{eq:Euler-entropy-inequality} with the space-periodic or no-flux boundary condition and initial condition $(\rho_0, \vec m_0, \eta_0)$ if the following holds:
\begin{itemize}
	\item (\textbf{lower bound on density and entropy})
	\begin{equation}
	\mathcal{V}_{t,\vec x}\left[ \left\{ \tilde{\rho} \geq 0, ~ \tilde{\eta} \geq \underline{S} \tilde{\rho} \right\} \right] = 1 \quad \mbox{for a.a.}~ (t,\vec x) \in (0,T)\times\Omega;
	\end{equation}
	
	\item (\textbf{energy inequlity})  the integral inequality
	\begin{equation} \label{eq:DMV-energy}
		\int_{\Omega} \langle \mathcal{V}_{\tau,\vec x};\, E(\tilde{\rho}, \tilde{\vec m}, \tilde{\eta}) \rangle ~d \vec x
		+ \int_{\Omega} d \mathfrak{E}_{cd}(\tau)
		\leq \int_{\Omega} E(\rho_0, \vec m_0, \eta_0) ~d \vec x
	\end{equation}
	holds for a.a. $0\leq \tau \leq T$ with the energy concentration defect
	\begin{equation*}
		\mathfrak{E}_{cd}\in  L^{\infty}(0,T; \mathcal{M}^+(\overline{\Omega}));
		\footnote{ $\mathcal{M}^+(\overline{\Omega})$ denotes the set of positive Radon measures on $\overline{\Omega}$.}
	\end{equation*}
	
	\item (\textbf{equation of continuity})
	\begin{equation*}
		\langle \mathcal{V}_{t,\vec x};\, \widetilde{\rho} \rangle \in C_{weak}([0,T];L^{\gamma}(\Omega)), \quad \langle \mathcal{V}_{0,\vec x};\, \widetilde{\rho} \rangle = \rho_0 \quad \mbox{for a.a.}~ \vec x \in \Omega
	\end{equation*}
	and the integral equality
	\begin{equation}
	\left[ \int_{\Omega}  \langle \mathcal{V}_{t,\vec x};\, \widetilde{\rho} \rangle \phi ~d \vec x  \right]_{t = 0}^{t = \tau} = \int_{0}^{\tau} \int_{\Omega}  \langle \mathcal{V}_{t,\vec x};\, \widetilde{\rho} \rangle \partial_t \phi +    \langle \mathcal{V}_{t,\vec x};\, \widetilde{\vec m} \rangle \cdot \nabla_{\vec x} \phi   ~d \vec x dt
	\end{equation}	
	for any $0\leq \tau \leq T$ and any $\phi \in W^{1,\infty}((0,T)\times \Omega)$;
	
	\item (\textbf{momentum equation})
	\begin{equation*}
	\langle \mathcal{V}_{t,\vec x};\, \widetilde{\vec m} \rangle \in C_{weak}([0,T];L^{\frac{2\gamma}{\gamma+1}}(\Omega; \mathbb{R}^d)), \quad \langle \mathcal{V}_{0,\vec x};\, \widetilde{\vec m} \rangle= m_0 \quad \mbox{for a.a.}~ \vec x \in \Omega
	\end{equation*}
	and the integral equality
	{\small
	\begin{align} \label{eq:DMV-momentum}
		& \left[ \int_{\Omega}  \langle \mathcal{V}_{t,\vec x};\, \widetilde{\vec m} \rangle \vec \phi ~d \vec x  \right]_{t = 0}^{t = \tau}
		 =  \int_{0}^{\tau} \int_{\Omega}  \langle \mathcal{V}_{t,\vec x};\, \widetilde{\vec m} \rangle \partial_t \vec \phi +   \langle  \mathcal{V}_{t,\vec x};\, \frac{ \widetilde{\vec m} \otimes\widetilde{\vec m} }{\widetilde{ \rho}}   \rangle : \nabla_{\vec x} \vec \phi  ~d \vec x dt \nonumber\\
		 &\hspace{2cm} + \int_{0}^{\tau} \int_{\Omega}    \langle \mathcal{V}_{t,\vec x};\,  p(\tilde{\rho},  \tilde{\eta})   \rangle \rm{div}_{\vec x} \vec \phi ~d \vec x dt
		+    \int_{0}^{\tau} \int_{\overline{\Omega}} \nabla_{\vec x} \vec \phi   : d\mathfrak{R}_{cd}(t) dt
	\end{align}}%
	for any $0\leq \tau \leq T$ and any $\vec \phi \in C^1([0,T]\times \overline{\Omega}; \mathbb{R}^d)$, ($\vec \phi$ also satisfies $\vec \phi \cdot \vec n |_{\partial \Omega} = 0$ when no-flux boundary condition is used),
	where the Reynolds concentration defect
	 \begin{equation*}
	 \mathfrak{R}_{cd} \in  L^{\infty}(0,T; \mathcal{M}^+(\overline{\Omega}; \mathbb{R}^{d\times d}_{\mbox{sym}}))
	 \end{equation*}
	 satisfies
	 \begin{equation}
	 	\underline{d} \mathfrak{E}_{cd}\leq \mbox{tr}[\mathfrak{R}_{cd}] \leq \overline{d} \mathfrak{E}_{cd} \quad \mbox{for some constants} ~ 0 < \underline{d} \leq \overline{d};
	 \end{equation}
	
	\item (\textbf{entropy balance})
	{\small
	\begin{align*}
	&& \int_{\Omega} \langle \mathcal{V}_{\tau\pm,\vec x};\, \tilde{\eta} \rangle \phi ~d \vec x  \equiv \lim\limits_{t \rightarrow \tau\pm} \int_{\Omega} \langle \mathcal{V}_{t,\vec x};\, \tilde{\eta} \rangle \phi ~d \vec x \quad \mbox{exists for any}~ 0 \leq \tau < T, \\
	&& \int_{\Omega} \langle \mathcal{V}_{0+,\vec x};\, \tilde{\eta} \rangle \phi ~d \vec x  \equiv \int_{\Omega} S_0  \phi ~d \vec x \quad \mbox{for any}~ \phi \in C(\overline{\Omega}),
	\end{align*}}%
	and the integral inequality
	{\small
	\begin{equation}
	\left[ \int_{\Omega} \langle \mathcal{V}_{t,\vec x};\, \tilde{\eta} \rangle \phi ~d \vec x   \right]^{t = \tau_2+}_{t = \tau_1-} \leq \int_{\tau_1}^{\tau_2} \int_{\Omega} \langle \mathcal{V}_{t,\vec x};\, \tilde{\eta}  \rangle \partial_t \phi + \langle \mathcal{V}_{t,\vec x};\, \frac{\tilde{\vec m}}{\tilde{\rho}} \tilde{\eta} \rangle \cdot \nabla_{\vec x} \phi ~d \vec x dt
	\end{equation}}%
	for any $0\leq \tau \leq T$ and any $\phi \in W^{1,\infty}((0,T)\times \Omega),\, \phi \geq 0$.
\end{itemize}

\end{definition}

\begin{remark}
Consider a family $\{ \rho_h, \vec m_h, E_h \}_{h \downarrow 0}$ of numerical solutions generated by our finite volume method \eqref{eq:semi-discrete-RPflux-phi}.
We note that a sequence $\{ \rho_h, \vec m_h, E_h \}_{h \downarrow 0}$ can be mapped uniquely to a sequence $\{ \rho_h, \vec m_h, \eta_h \}_{h \downarrow 0}$.
Due to Theorem~\ref{theorem:consistency} $\{ \rho_h, \vec m_h, \eta_h \}_{h \downarrow 0}$ is a consistent approximation  of complete Euler system.
Consequently, up to a subsequence $\{ \rho_{h_n}, \vec m_{h_n}, \eta_{h_n} \}_{h_n \downarrow 0}$ generates the Young measure  $\{\mathcal{V}_{t,\vec x}\}_{(t,\vec x)\in (0,T)\times\Omega}$, which is a disspative measure-valued solution of the Euler system in the  sense of Definition \ref{def:DMV}.
Following \cite{Feireisl-Lukacova-Mizerova-She:2021} the concentration defects are
{\small 
\begin{align*}
& \mathfrak{E}_{cd}= \overline{E(\rho, \vec m, \eta)} - \langle \mathcal{V}_{\tau,\vec x};\, E(\tilde{\rho}, \tilde{\vec m}, \tilde{\eta}) \rangle, \\ 
&\mathfrak{R}_{cd} = \overline{ \frac{{\vec m} \otimes{\vec m} }{{ \rho}}  + p(\rho,\eta) \mathbb{I} }  -   \left\langle  \mathcal{V}_{t,\vec x};\, \frac{ \widetilde{\vec m} \otimes\widetilde{\vec m} }{\widetilde{ \rho}} + p(\tilde{\rho},\tilde{\eta}) \mathbb{I}    \right\rangle
\end{align*}}%
with
{\small
\begin{align*}
&& E(\rho_{h_n}, \vec m_{h_n}, \eta_{h_n}) \longrightarrow~  \overline{E(\rho, \vec m, \eta)}  \quad \mbox{weakly-}(*)~\mbox{in~}  \mathcal{M}(\overline{\Omega}),\\
&&  \frac{ {\vec m_{h_n}} \otimes {\vec m_{h_n}} }{{ \rho_{h_n}}}  + p(\rho_{h_n},\eta_{h_n}) \mathbb{I} \longrightarrow~
  \overline{ \frac{{\vec m} \otimes{\vec m} }{{ \rho}}  + p(\rho,\eta) \mathbb{I} }
\quad \mbox{weakly-}(*)~\mbox{in~}   \mathcal{M}(\overline{\Omega}; \mathbb{R}^{d\times d}_{\mbox{sym}}).
\end{align*}}%
\end{remark}

\begin{theorem}\rm\label{thm:weak-convergence}
(\textbf{Weak convergence})
	
\noindent Let $\{ \rho_h, \vec m_h, \eta_h \}_{h \downarrow 0}$ be the family of numerical solutions obtained by the finite volume method \eqref{eq:semi-discrete-RPflux-phi}.
Let Assumption~\ref{assumption} hold, i.e.
$0 < \underline{\rho} \leq \rho_h,
E_h \leq \overline{E}$ for some $\underline{\rho},\, \overline{E}$.
Then there exists a subsequence
$\{ \rho_{h_n}, \vec m_{h_n}, \eta_{h_n} \}_{{h_n} > 0}$,
such that
\begin{align*}
&& (\rho_{h_n}, \vec m_{h_n}, \eta_{h_n}) ~ \longrightarrow~ \langle \mathcal{V}_{t,\vec x};\, (\widetilde{\rho}, \widetilde{\vec m} , \widetilde{\eta}  )\rangle \quad \mbox{weakly-}(*)~\mbox{in~}  L^{\infty }((0,T)\times \Omega; \mathbb{R}^{d+2}),
\end{align*}
where $\{\mathcal{V}_{t,\vec x}\}_{(t,\vec x)\in (0,T)\times\Omega}$ is a DMV solution of the complete Euler system \eqref{eq:multiD-Euler}, \eqref{eq:Euler-entropy-inequality} with
\begin{equation}
\mathfrak{E}_{cd}\equiv 0, \quad \mathfrak{R}_{cd} \equiv 0.
\end{equation}
\end{theorem}

\begin{proof}
Under Assumption~\ref{assumption} Lemma~\ref{lemma:lemma-U-bounded} gives
\begin{align*}
& \rho_h \in L^{\infty}((0,T) \times \Omega), ~ \vec m_h \in L^{\infty}((0,T)\times \Omega), ~\eta_h \in L^{\infty}((0,T) \times \Omega),  ~ E_h \in L^{\infty}((0,T) \times \Omega), \\
& \frac{\vec m_h \otimes\vec m_h }{\rho_h} \in L^{\infty}((0,T) \times \Omega), ~ p_h \in L^{\infty}((0,T) \times \Omega), ~
\vec q_h \in L^{\infty}((0,T) \times \Omega).
\end{align*}
Applying the Fundamental Theorem on Young Measure \cite{Ball:1989} implies the existence of a convergent subsequence and a parameterized probability measure $\{\mathcal{V}_{t,\vec x}\}_{(t,\vec x)\in (0,T)\times\Omega}$ satisfying that $(\rho_{h_n}, \vec m_{h_n} \eta_{h_n})$ weakly-(*) converges to $(\langle \mathcal{V}_{t,\vec x};\, \widetilde{\rho} \rangle, \langle  \mathcal{V}_{t,\vec x};\, \widetilde{\vec m}  \rangle, \langle \mathcal{V}_{t,\vec x};\, \widetilde{\eta} \rangle)$ in $L^{\infty }((0,T)\times \Omega)$.
Moreover,
\begin{equation*}
\frac{\vec m_{h_n}  \otimes\vec m_{h_n}  }{\rho_{h_n} }, \quad
p_{h_n} :=p(\rho_{h_n} ,\eta_{h_n} ),\quad
E_{h_n}  := E(\rho_{h_n} ,\vec m_{h_n} ,\eta_{h_n} ),\quad
 \vec q_{h_n}  := \vec q(\rho_{h_n} ,\vec m_{h_n} ,\eta_{h_n} )
\end{equation*}
weakly-(*) converge to
\begin{align*}
&\langle  \mathcal{V}_{t,\vec x};\, \frac{ \widetilde{\vec m} \otimes\widetilde{\vec m} }{\widetilde{\vec \rho}}   \rangle,\quad
\langle \mathcal{V}_{t,\vec x};\, p(\widetilde{\rho}, \widetilde{\eta}) \rangle, \quad
\langle \mathcal{V}_{t,\vec x};\, E(\widetilde{\rho}, \tilde{\vec m},\widetilde{\eta}) \rangle, \quad
\langle \mathcal{V}_{t,\vec x};\, \vec q(\widetilde{\rho}, \tilde{\vec m},\widetilde{\eta}) \rangle
\end{align*}
in $L^{\infty }((0,T)\times \Omega)$, respectively.
Consequently, the concentration defects vanish, i.e. $\mathfrak{E}_{cd}\equiv 0, ~ \mathfrak{R}_{cd} \equiv 0$.

Hence, passing to the limit $h \rightarrow 0$, \eqref{eq:consistency-density} in Theorem~\ref{theorem:consistency} gives
\begin{equation}
\left[ \int_{\Omega}  \langle \mathcal{V}_{t,\vec x};\, \widetilde{\rho} \rangle \phi ~d \vec x  \right]_{t = 0}^{t = \tau} = \int_{0}^{\tau} \int_{\Omega}  \langle \mathcal{V}_{t,\vec x};\, \widetilde{\rho} \rangle \partial_t \phi +    \langle \mathcal{V}_{t,\vec x};\, \widetilde{\vec m} \rangle \cdot \nabla_{\vec x} \phi   ~d \vec x dt
\end{equation}	
for $\phi \in W^{1,\infty}((0,T) \times \Omega)$.
Analogously, \eqref{eq:consistency-momentum} and  \eqref{eq:consistency-entropy} in Theorem~\ref{theorem:consistency} yield
\begin{align*}
\left[ \int_{\Omega}  \langle \mathcal{V}_{t,\vec x};\, \widetilde{\vec m} \rangle \vec \phi ~d \vec x  \right]_{t = 0}^{t = \tau} \nonumber
= & \int_{0}^{\tau} \int_{\Omega}  \langle \mathcal{V}_{t,\vec x};\, \widetilde{\vec m} \rangle \partial_t \vec \phi +   \langle  \mathcal{V}_{t,\vec x};\, \frac{ \widetilde{\vec m} \otimes\widetilde{\vec m} }{\widetilde{ \rho}}   \rangle : \nabla_{\vec x} \vec \phi  ~d \vec x dt\nonumber\\
+ & \int_{0}^{\tau} \int_{\Omega}    \langle \mathcal{V}_{t,\vec x};\, p(\tilde{\rho},  \tilde{\eta})   \rangle \div_{\vec x} \vec \phi ~d \vec x dt
\end{align*}
for $\vec \phi \in C^{1}([0,T] \times \overline{\Omega};\mathbb{R}^d)$ and $\vec \phi \cdot \vec n = 0$ for no-flux boundary condition,
and
\begin{equation}
\left[ \int_{\Omega} \langle \mathcal{V}_{t,\vec x};\, \tilde{\eta} \rangle \phi ~d \vec x   \right]^{t = \tau_2+}_{t = \tau_1-} \leq \int_{\tau_1}^{\tau_2} \int_{\Omega} \langle \mathcal{V}_{t,\vec x};\, \tilde{\eta}  \rangle \partial_t \phi + \langle \mathcal{V}_{t,\vec x};\, \vec q(\tilde{\rho}, \tilde{\vec m}, \tilde{\eta})  \rangle \cdot \nabla_{\vec x} \phi ~d \vec x dt
\end{equation}	
for $\phi \in W^{1,\infty}((0,T) \times \Omega)$, respectively.
Finally, \eqref{eq:consistency-energy} in Theorem~\ref{theorem:consistency} implies
\begin{equation}
 \int_{\Omega} \langle \mathcal{V}_{\tau,\vec x};\, E(\tilde{\rho}, \tilde{\vec m}, \tilde{\eta})  \rangle ~d \vec x   = \int_{\Omega} E(\rho_0, \vec m_0, \eta_0)  ~d \vec x
\end{equation}	
and concludes that $\{\mathcal{V}_{t,\vec x}\}_{(t,\vec x)\in (0,T)\times\Omega}$ is a DMV solution of the complete Euler system.
\end{proof}

Having shown weak convergence to a DMV solution allows us to look for strong convergence to the observable quantities, such as the expected value and first variance.
To this end we apply a novel technique of $\mathcal{K}$-convergence as introduced in \cite{Feireisl-Lukacova-Mizerova:2020b,Feireisl-Lukacova-Mizerova-She:2021}.

\begin{theorem} \rm\label{thm:K-convergence}
(\textbf{$\mathcal{K}$-convergence})

\noindent Let $\{ \rho_h, \vec m_h, \eta_h \}_{h \downarrow 0}$ be the family of numerical solutions obtained by the finite volume method \eqref{eq:semi-discrete-RPflux-phi}.
Assumption~\ref{assumption} holds.
Then there exist a subsequence
$\{ \rho_{h_n}, \vec m_{h_n}, \eta_{h_n} \}_{{h_n} \downarrow 0}$ such that
\begin{itemize}
\item  \textbf{strong convergences of  Ces\`{a}ro average}
			{\small
			\begin{equation*}
			\frac{1}{N} \sum_{n=1}^{N} (\rho_{h_n}, \vec m_{h_n}, \eta_{h_n} ) \longrightarrow \langle \mathcal{V}_{t,\vec x};\, (\widetilde{\rho}, \widetilde{\vec m}, \widetilde{\eta} ) \rangle  \quad \mbox{in~}  L^{q}((0,T)\times \Omega; \mathbb{R}^{d+2})
			\end{equation*}}%
			for $N \to \infty$ and any $1\leq q < \infty$;
\item \textbf{$L^q$ convergence to Young measure}
			{\small
			\begin{equation*}
				d_{W_s} \left[ \frac{1}{N} \sum_{n=1}^{N} \delta_{(\rho_{h_n}, \vec m_{h_n}, \eta_{h_n})} ;  \mathcal{V}_{t,\vec x}\right] \longrightarrow 0 \quad \mbox{in~}  L^{q }((0,T)\times \Omega)
			\end{equation*}}%
			for $N \to \infty$ and any $1 \leq q < s < \infty$;
\item \textbf{$L^1$ convergence of the first variance}
			{\small
			\begin{equation*}
				 \frac{1}{N} \sum_{n=1}^{N} \bigg\|(\rho_{h_n}, \vec m_{h_n}, \eta_{h_n}) -   \frac{1}{N} \sum_{n=1}^{N} (\rho_{h_n}, \vec m_{h_n}, \eta_{h_n})\bigg\| \longrightarrow 0 \quad \mbox{in~}  L^{1 }((0,T)\times \Omega)
			\end{equation*}}%
for $N \to \infty.$ 
\end{itemize}
\end{theorem}

\noindent Applying techniques developed in \cite{Feireisl-Lukacova-Mizerova-She:2021} we directly obtain the following strong convergence results.
\begin{theorem} \rm\label{thm:strong-convergence}
(\textbf{Strong convergence})
	
\noindent 	 Let $\{ \rho_h, \vec m_h, \eta_h \}_{h \downarrow 0}$ be the family of numerical solutions obtained by the finite volume method \eqref{eq:semi-discrete-RPflux-phi} and with initial data $\rho_{0,h} = \Pi_h[\rho_0],\, \vec m_{0,h} = \Pi_h[\vec m_0], \, \eta_{0,h} = \Pi_h[\eta_0]$.
Let Assumption~\ref{assumption} hold.
Let the subsequence
\begin{equation*}
(\rho_{h_n}, \vec m_{h_n}, \eta_{h_n}) \longrightarrow (\rho, \vec m, \eta) \quad \mbox{as}~ h\longrightarrow 0
\end{equation*}
in the sense specified in Theorem~\ref{thm:weak-convergence}, where the barycenters $\rho := \langle \mathcal{V}_{t,\vec x};\, \widetilde{\rho} \rangle$, $\vec m := \langle \mathcal{V}_{t,\vec x};\, \widetilde{\vec m} \rangle$ and $\eta := \langle \mathcal{V}_{t,\vec x};\, \widetilde{\eta} \rangle$.
Then the following holds:
\begin{itemize}
	\item \textbf{weak solution}
	
	If $(\rho, \vec m, \eta)$ is a weak entropy solution of the Euler system with initial data $(\rho_0, \vec m_0, \eta_0)$, then
	\begin{equation*}
		\mathcal{V}_{t,\vec x} = \delta_{(\rho(t,\vec x), \vec m(t,\vec x), \eta(t,\vec x))} \quad \mbox{for a.a.}~ (t,\vec x)\in(0,T)\times\Omega
	\end{equation*}
	and the strong convergence holds, i.e.
	\begin{align*}
		&& (\rho_{h_n}, \vec m_{h_n}, \eta_{h_n} ) \longrightarrow~  (\rho,\vec m, \eta)  \hspace{1cm}\quad \mbox{in~}  L^{q }((0,T)\times \Omega; \mathbb{R}^{d+2}) \\
		&& E(\rho_{h_n}, \vec m_{h_n}, \eta_{h_n}) \longrightarrow ~ E(\rho, \vec m, \eta)  \hspace{1cm}\hspace{0.8cm} \mbox{in~}  L^{q }((0,T)\times \Omega)
	\end{align*}
	for any $1\leq q < \infty$.

	\item \textbf{classical solution}
	
	Let $\Omega \subset \mathbb{R}^d$ be a bounded Lipschitz domain and $ (\rho, \vec m, \eta) $  such that
	\begin{equation*}
	\rho, \eta \in C^{1}([0,T] \times \overline{\Omega}), ~ \vec m  \in C^{1}([0,T] \times \overline{\Omega}; \mathbb{R}^d), ~ \rho \geq \underline{\rho} > 0 ~ \mbox{in} ~ [0,T] \times \overline{\Omega}.
	\end{equation*}
	Then  $ (\rho, \vec m, \eta) $ is a classical solution to the Euler system and
	\begin{align*}
	&(\rho_{h_n}, \vec m_{h_n}, \eta_{h_n} ) \longrightarrow~  (\rho,\vec m, \eta)  \quad \hspace{1cm}\mbox{in~}  L^{q }((0,T)\times \Omega; \mathbb{R}^{d+2})
	\end{align*}
	for any $1\leq q < \infty$.
	
	\item \textbf{strong solution}
	
	Let periodic boundary conditions are applied.
	Suppose that the Euler system admits a strong solution $ (\rho, \vec m, \eta) $ in the class
	\begin{equation*}
		\rho, \eta \in W^{1,\infty}((0,T)\times \Omega), ~ \vec m  \in W^{1,\infty}((0,T)\times \Omega; \mathbb{R}^d), ~ \rho \geq \underline{\rho} > 0 ~ \mbox{in} ~ [0,T)\times\Omega
	\end{equation*}
	emanating from initial data $(\rho_0, \vec m_0, \eta_0)$. Then  it holds
		\begin{align*}
		&& (\rho_{h_n}, \vec m_{h_n}, \eta_{h_n} ) \longrightarrow~  (\rho,\vec m, \eta)  \hspace{1cm}\quad \mbox{in~}  L^{q }((0,T)\times \Omega; \mathbb{R}^{d+2})\\
	&& E(\rho_{h_n}, \vec m_{h_n}, \eta_{h_n}) \longrightarrow ~ E(\rho, \vec m, \eta)  \hspace{1cm}\hspace{0.8cm} \mbox{in~}  L^{q }((0,T)\times \Omega)
	\end{align*}
	for any $1\leq q < \infty$.

\end{itemize}
\end{theorem}

\section{Numerical results}
\noindent In this section we simulate a spiral problem, i.e. two-dimensional Riemann problem, 
Kelvin-Helmholtz problem and Richtmyer-Meshkov problem \cite{Fjordholm-Mishra-Tadmor:2016,Fjordholm-Kappeli-Mishra-Tadmor:2017} to illustrate the weak, strong and $\mathcal{K}$-convergence of the finite volume method \eqref{eq:semi-discrete-RPflux-phi}.

In our computations the computational domain is $[0,1]\times[0,1]$ and divided into $n\times n$ uniform cells.
Denote the Ces\`{a}ro  average of the numerical solutions and their first variance
{
\begin{equation*}
\tilde{U}_{h_n} = \frac1n\sum_{j = 1}^{n} U_{h_j}, \quad \tilde{U}^\dagger_{h_n} = \frac1n\sum_{j = 1}^{n} |U_{h_j}-\tilde{U}_{h_n}|,
\end{equation*}}%
respectively.
Let $U_{h_N}$ be the reference solution computed on the finest mesh with $N \times N$ cells.
Analogously to \cite{Feireisl-Lukacova-She-Wang:2019} we compute four errors 
{
\begin{equation}
	E_1 = \| U_{h_n} - U_{h_N} \|,
	E_2 = \| \tilde{U}_{h_n} - \tilde{U}_{h_N} \|,
 	E_3 = \| U^\dagger_{h_n} - U^\dagger_{h_N} \|,
 	E_4 = \| W_1(\overline{\mathcal{V}}_{t,x}^n, \overline{\mathcal{V}}_{t,x}^N) \|,
\end{equation}}%
where
$\overline{\mathcal{V}}_{t,x}^n$ is the Ces\`{a}ro  average of the Dirac measures concentrated on numerical solution $U_{h_n}$.
In addition, we apply the outflow boundary condition to the spiral problem and periodic boundary condition to the other two problems. 
Moreover, the CFL number is set to $0.9$ and the adiabatic index $\gamma$ is taken as $1.4$.

\begin{example}[\textbf{Spiral problem}]\label{example:2D-Spiral}\rm
	We consider one of the classical 2D Riemann problem with the initial data
	{\small
	\begin{equation*}
		(\rho ,  \vec u , p)(x,0)
		=  \begin{cases}
			(0.5 ,\, 0.5 ,\, -0.5 ,\, 5 ) , ~\mbox{if}~ x > 0.5,\,y>0.5;  \hspace{0.45cm}
			(1 ,\,  0.5,\, 0.5 , \, 5) , ~\mbox{if}~ x<0.5,\,y>0.5; \\
			(1.5 ,\,  -0.5,\, -0.5 , \, 5) ,  ~\mbox{if}~ x>0.5,\,y<0.5;  \hspace{0.15cm}
			(2 ,\,  -0.5,\, 0.5 , \, 5) , ~\mbox{if}~ x<0.5,\,y<0.5.
		\end{cases}
	\end{equation*}}%
	This problem describes the interaction of four contact discontinuities (vortex sheets) with the negative sign.
	As time increases the four initial vortex sheets interact each other to form a spiral with the low density around the center of the domain.
	This is a typical cavitation phenomenon well-known in gas dynamics.
	We compute the solution up to the finite time $T = 2$.
	
	Figure \ref{figure:2D-Spiral-error} shows the errors $E_1, E_2, E_3, E_4$ of $\rho, m_1, m_2, E, S$ obtained on different meshes and the reference solution on a mesh with $2048\times 2048$ cells.
	The errors of $\rho, S$ are specifically listed in Tables \ref{tabel:2D-Spiral-density},  \ref{tabel:2D-Spiral-entropy},  respectively.
	Moreover, Figures \ref{figure:2D-Spiral-densitycontour},  \ref{figure:2D-Spiral-entropycontour} show the contour of $\rho$ and $S$ obtained on different meshes, respectively.
	
	The numerical results show that four errors are all decreasing with the refinement of mesh.
	This together with the pictures of the first variance indicate that $\mathcal{V}_{t,x} = \delta_{(\rho,\vec m, S)(x,t)}$ and the numerical solutions converge to the weak solution.
	This is in accordance with our theoretical results.
	We point out that the convergence rate is 1.
\end{example}

\begin{figure}[htbp]
	\setlength{\abovecaptionskip}{0.cm}
	\setlength{\belowcaptionskip}{-0.cm}
	\centering
	\begin{subfigure}{0.243\textwidth}
		\includegraphics[width=\textwidth]{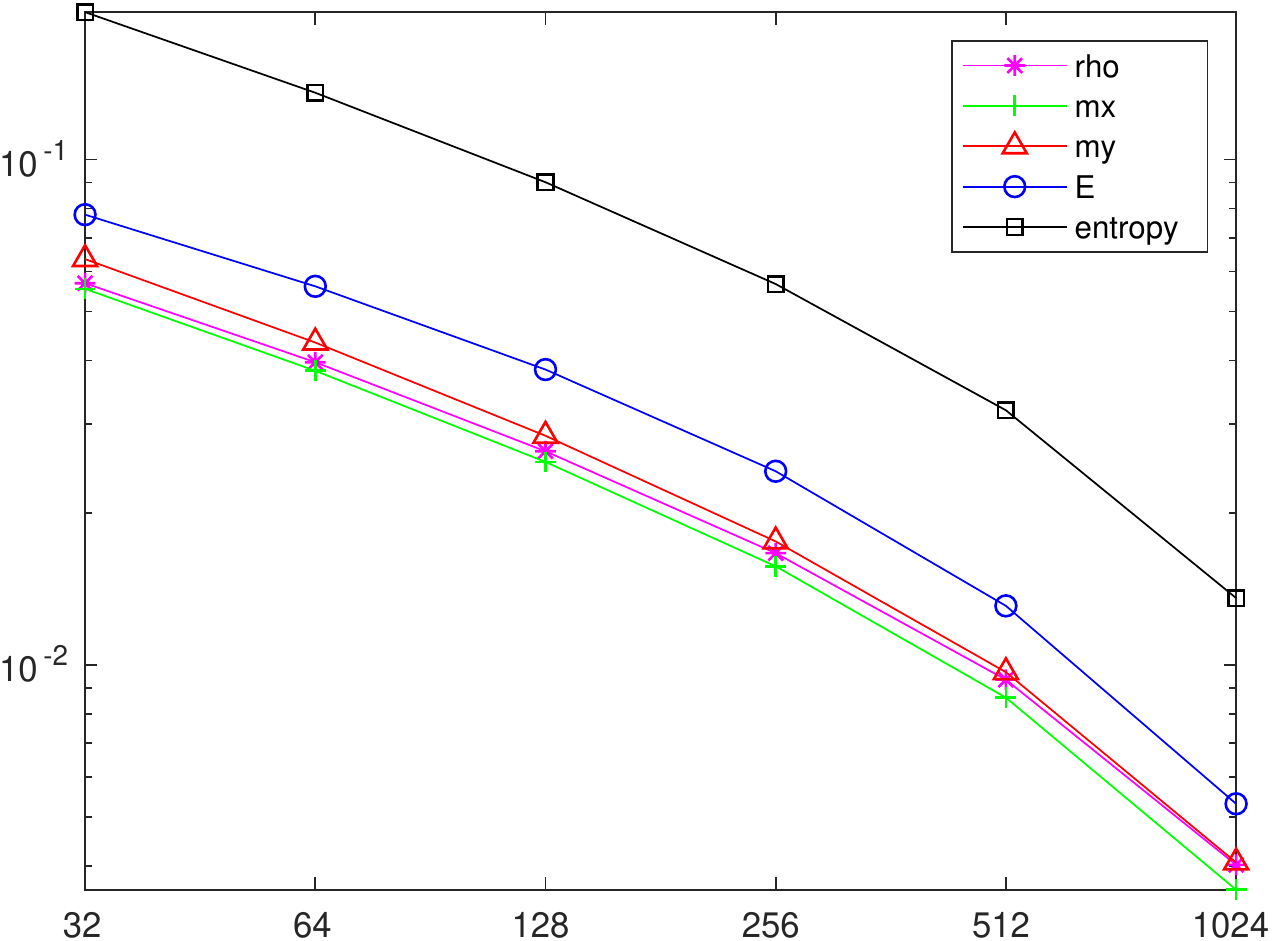}
		\caption{$E_1$}
	\end{subfigure}	
	\begin{subfigure}{0.243\textwidth}
		\includegraphics[width=\textwidth]{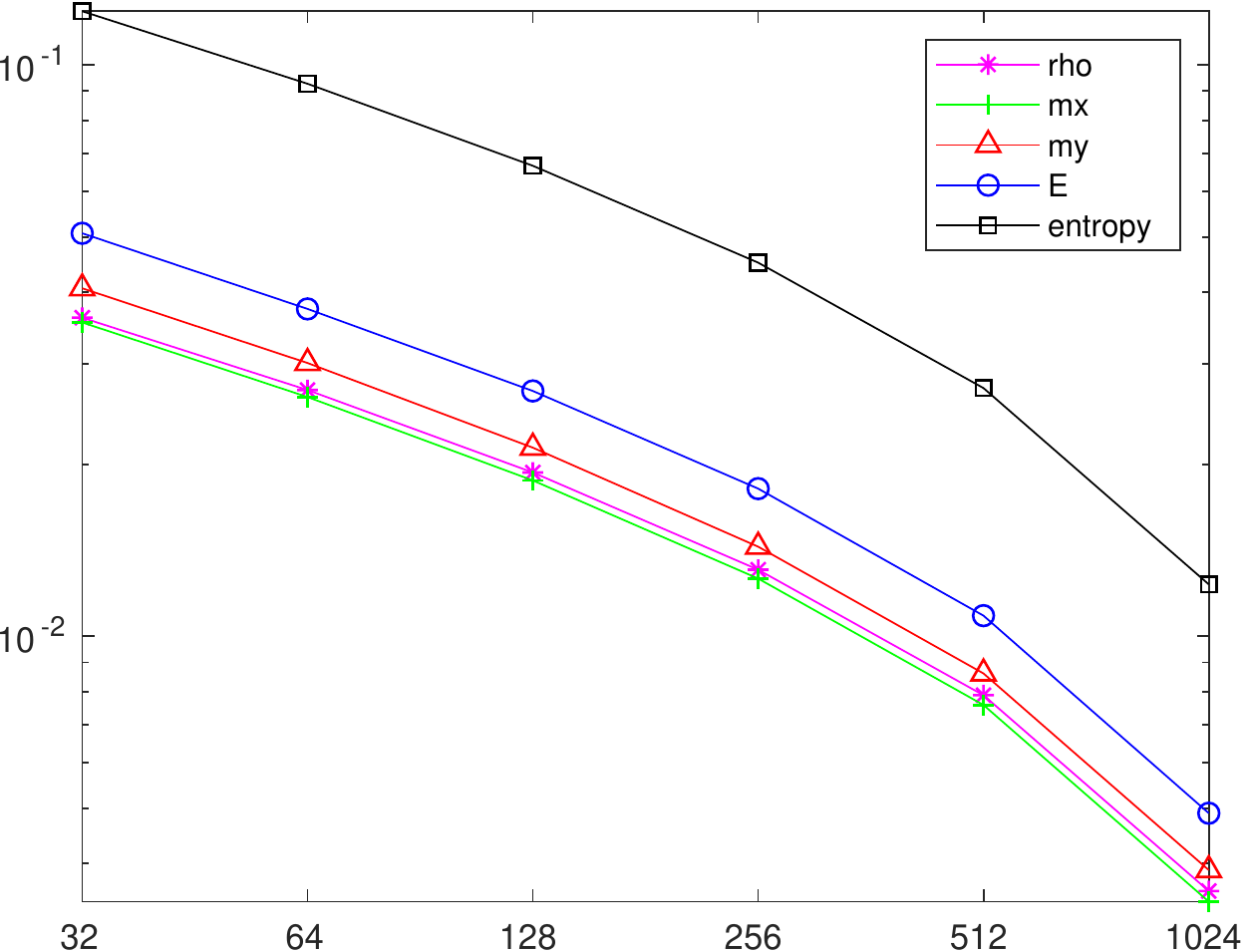}
		\caption{$E_2$}
	\end{subfigure}	
	\begin{subfigure}{0.243\textwidth}
		\includegraphics[width=\textwidth]{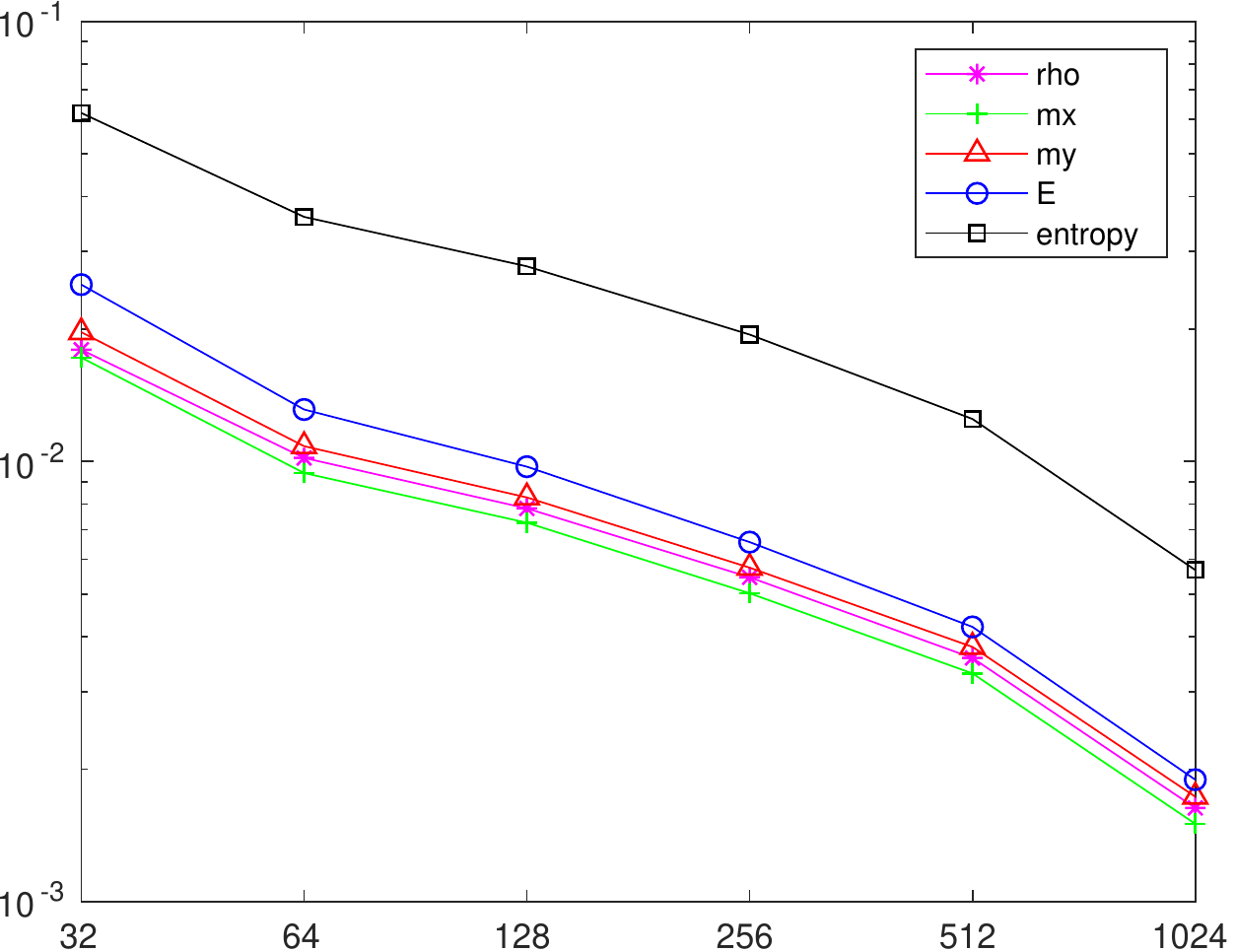}
		\caption{$E_3$}
	\end{subfigure}	
	\begin{subfigure}{0.243\textwidth}
		\includegraphics[width=\textwidth]{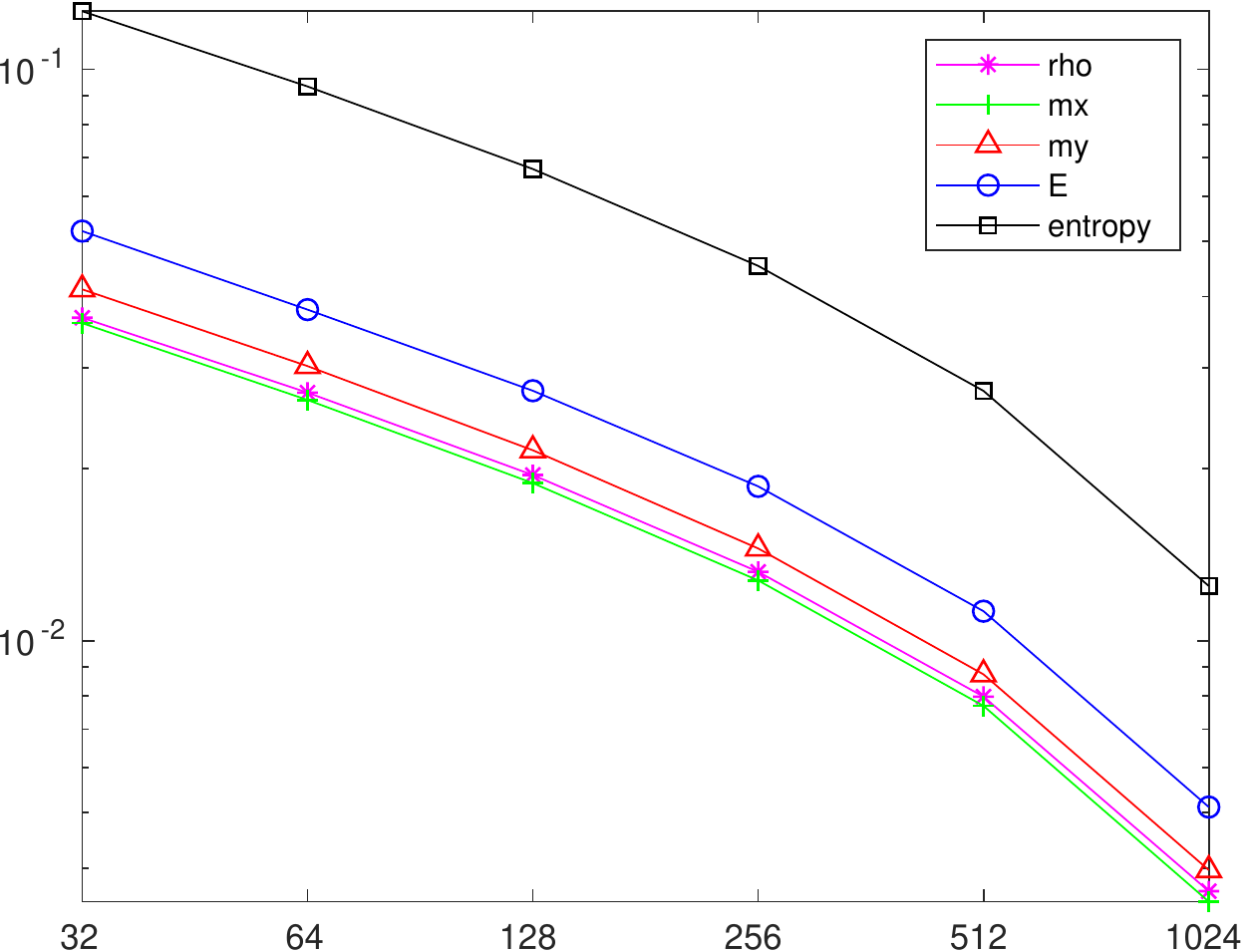}
		\caption{$E_4$}
	\end{subfigure}	
	\caption{\small{Example \ref{example:2D-Spiral}: the errors obtained on different meshes.}}\label{figure:2D-Spiral-error}
\end{figure}

\begin{table}[htbp]
	\centering
	\caption{Example \ref{example:2D-Spiral}: the errors and convergence rates for density.  } \label{tabel:2D-Spiral-density}
	\begin{tabular}{|c|cc|cc|cc|cc|}
		\hline
		\multirow{2}{*}{$n$} & \multicolumn{2}{c|}{$E_1$ } & \multicolumn{2}{c|}{$E_2$ } & \multicolumn{2}{c|}{$E_3$ }  & \multicolumn{2}{c|}{$E_4$ } \\
		\cline{2-9}
		& error & order &  error & order &  error & order  &  error & order \\
		\hline
		\hline
		32 & 0.0569 & - & 0.0361 & - & 0.0179 & - & 0.0367 & - \\

		64 & 0.0397 & 0.5189 & 0.0270 & 0.4206 & 0.0102 & 0.8165 & 0.0272 & 0.4354\\

		128 & 0.0265 & 0.5854 & 0.0193 & 0.4795 & 0.0078 & 0.3810 & 0.0195 & 0.4790\\

		256 & 0.0166 & 0.6698 & 0.0131 & 0.5655 & 0.0055 & 0.5199 & 0.0132 & 0.5626\\

		512 & 0.0093 & 0.8316 & 0.0079 & 0.7299 & 0.0036 & 0.6089 & 0.0080 & 0.7258\\

		1024 & 0.0040 & 1.2172 & 0.0036 & 1.1388 & 0.0016 & 1.1319 & 0.0036 & 1.1311\\
		\hline
	\end{tabular}
\end{table}
\begin{table}[htbp]
	\centering
	\caption{Example \ref{example:2D-Spiral}: the errors and convergence rates for entropy.  } \label{tabel:2D-Spiral-entropy}
	\begin{tabular}{|c|cc|cc|cc|cc|}
		\hline
		\multirow{2}{*}{$n$} & \multicolumn{2}{c|}{$E_1$ } & \multicolumn{2}{c|}{$E_2$ } & \multicolumn{2}{c|}{$E_3$ }  & \multicolumn{2}{c|}{$E_4$ } \\
		\cline{2-9}
		& error & order &  error & order &  error & order  &  error & order \\
		\hline
		\hline
		32 & 0.1956 & - & 0.1242 & - & 0.0619 & - & 0.1263 & - \\

		64 & 0.1355 & 0.5297 & 0.0927 & 0.4216 & 0.0360 & 0.7841 & 0.0933 & 0.4372\\

		128 & 0.0901 & 0.5882 & 0.0666 & 0.4771 & 0.0278 & 0.3734 & 0.0669 & 0.4800\\

		256 & 0.0567 & 0.6692 & 0.0451 & 0.5634 & 0.0194 & 0.5150 & 0.0453 & 0.5632\\

		512 & 0.0319 & 0.8297 & 0.0272 & 0.7294 & 0.0125 & 0.6370 & 0.0274 & 0.7267\\

		1024 & 0.0136 & 1.2341 & 0.0123 & 1.1417 & 0.0057 & 1.1396 & 0.0125 & 1.1346\\
		\hline
	\end{tabular}
\end{table}

\begin{figure}[htbp]
		\setlength{\abovecaptionskip}{-0.1cm}
	\setlength{\belowcaptionskip}{-0.2cm}
	\centering
	\includegraphics[width=\textwidth]{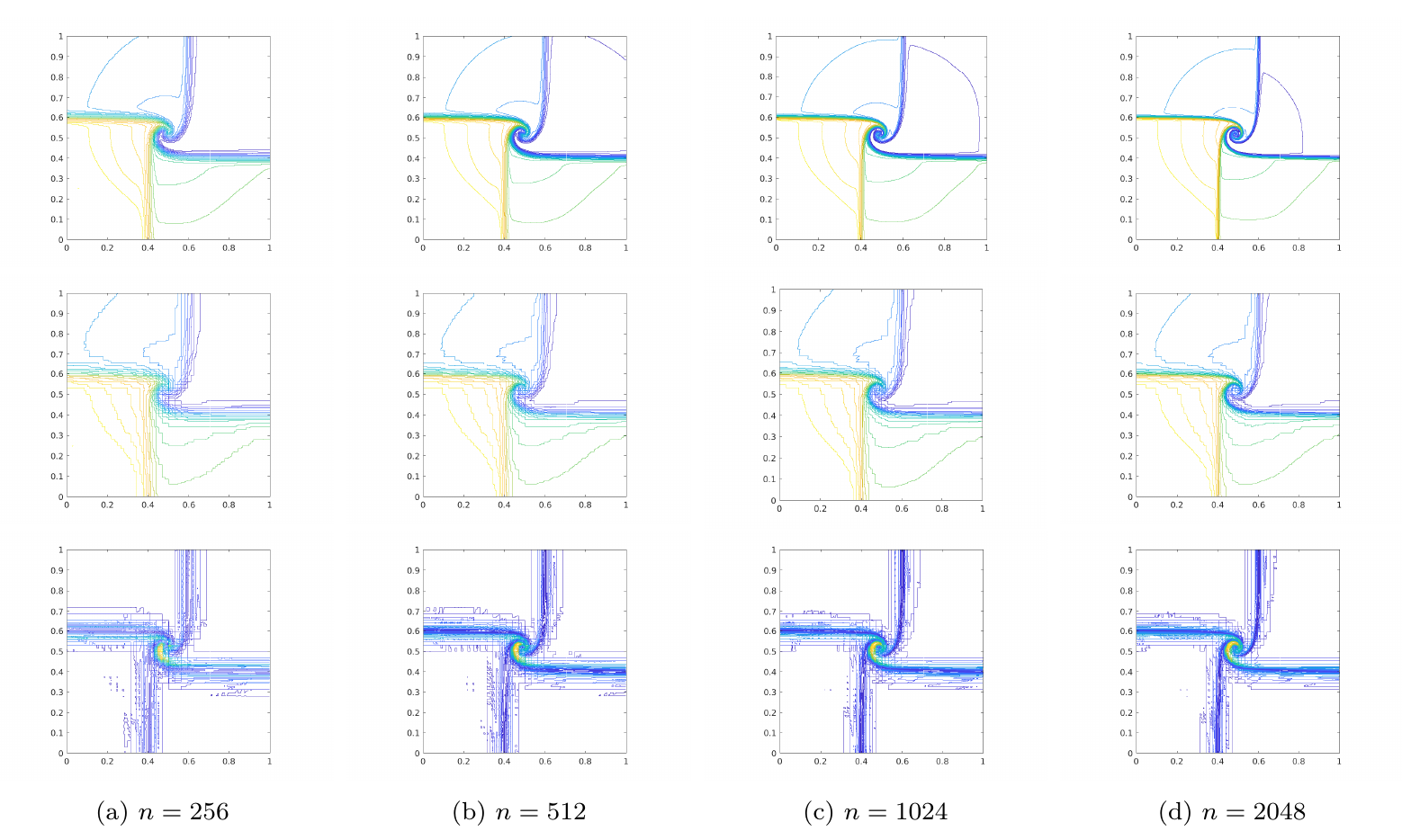}
	\caption{\small{Example \ref{example:2D-Spiral}: the density contours obtained on meshes with $n\times n$ cells. From top to bottom: density (top); Ces\'{a}ro averages of density (middle); the first variance of density (bottom). }}\label{figure:2D-Spiral-densitycontour}
\end{figure}


\begin{figure}[t]
		\setlength{\abovecaptionskip}{-0.1cm}
	\setlength{\belowcaptionskip}{-0.2cm}
	\centering
	\includegraphics[width=\textwidth]{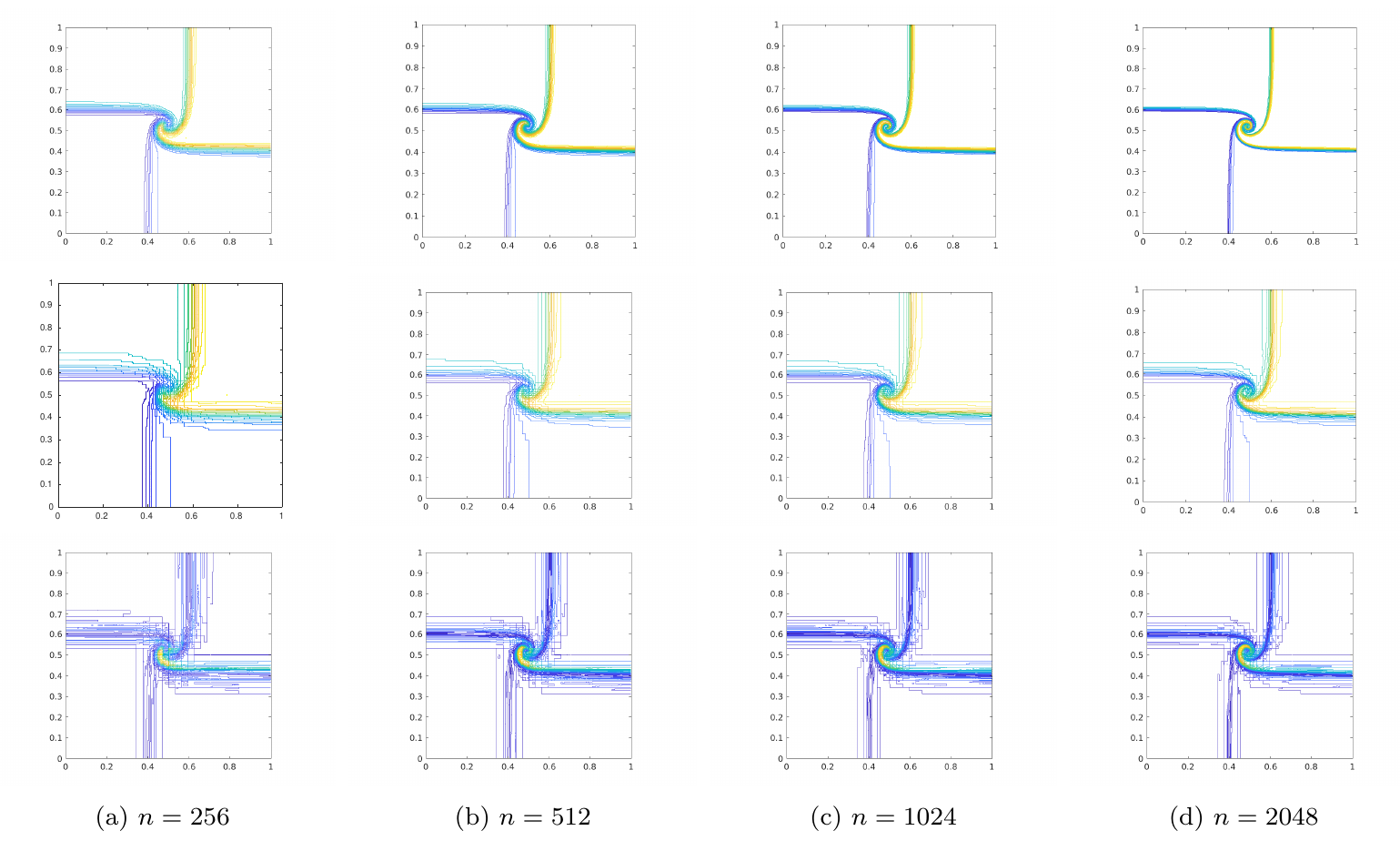}
	\caption{\small{Example \ref{example:2D-Spiral}: the entropy contours obtained on meshes with $n\times n$ cells. From top to bottom: entropy (top); Ces\'{a}ro averages of entropy (middle); the first variance of entropy (bottom). }}\label{figure:2D-Spiral-entropycontour}
\end{figure}

\begin{example}[\textbf{Kelvin-Helmholtz problem}]\label{example:2D-KH}\rm
	We consider a shear flow of three fluid layers with different densities.
	The initial data are given by
	\begin{equation*}
	(\rho ,  u , v, p)(x,0)
	=  \begin{cases}
	(2 ,\, -0.5 ,\, 0 ,\, 2.5 ) , &  \mbox{if}~ I_1 < x_2 < I_2, \\
	(1 ,\,  0.5,\, 0 , \, 2.5) , & \mbox{otherwise},
	\end{cases}
	\end{equation*}
	where the interface profiles
	\begin{equation*}
	I_j = I_j(\vec x) := J_j + \epsilon Y_j(\vec x), \quad j = 1,2
	\end{equation*}
	are chosen to be small perturbations around the lower $J_1 = 0.25$ and the upper $J_2 = 0.75$ interfaces, respectively.
	Moreover,
	\begin{equation*}
	Y_j = \sum_{m = 1}^{M} a_j^m \cos( b_j^m+ 2n\pi x_1 ), \quad j = 1,2,
	\end{equation*}
	where $a_j^m \in [0,1]$ and $b_j^m \in [-\pi, \pi], j = 1,2, m = 1, \cdots, M$ are arbitrary, but fixed numbers.
	The coefficients $a_j^m$ have been normalized such that $\sum_{m = 1}^{M} a_j^m = 1$ to guarantee that $|I_j - J_j | < \epsilon$ for $j = 1,2$.
	In the simulation we have $M = 10, \epsilon = 0.01$, $T = 2$ and $N  = 2048$.
	
	Figure \ref{figure:2D-KH-error} shows the errors $E_1, E_2, E_3, E_4$ of $\rho, m_1, m_2, E, S$ obtained on different meshes.
	Tables \ref{tabel:2D-KH-density}, \ref{tabel:2D-KH-entropy} present the errors of $\rho, S$,  respectively.
	Moreover, Figures \ref{figure:2D-KH-densitycontour}, \ref{figure:2D-KH-entropycontour} show the contours of $\rho$ and $S$ obtained on different meshes, respectively.
	
	The numerical results show that the numerical solutions obtained by the finite volume method \eqref{eq:semi-discrete-RPflux-phi} seem to converge in the sense of $L^1$. Though, the concergence rates of $E_2,\,E_3,\,E_4$ are much higher than that of $E_1$.
\end{example}

\begin{figure}[htbp]
	\setlength{\abovecaptionskip}{0.cm}
	\setlength{\belowcaptionskip}{-0.cm}
	\centering
	\begin{subfigure}{0.243\textwidth}
		\includegraphics[width=\textwidth]{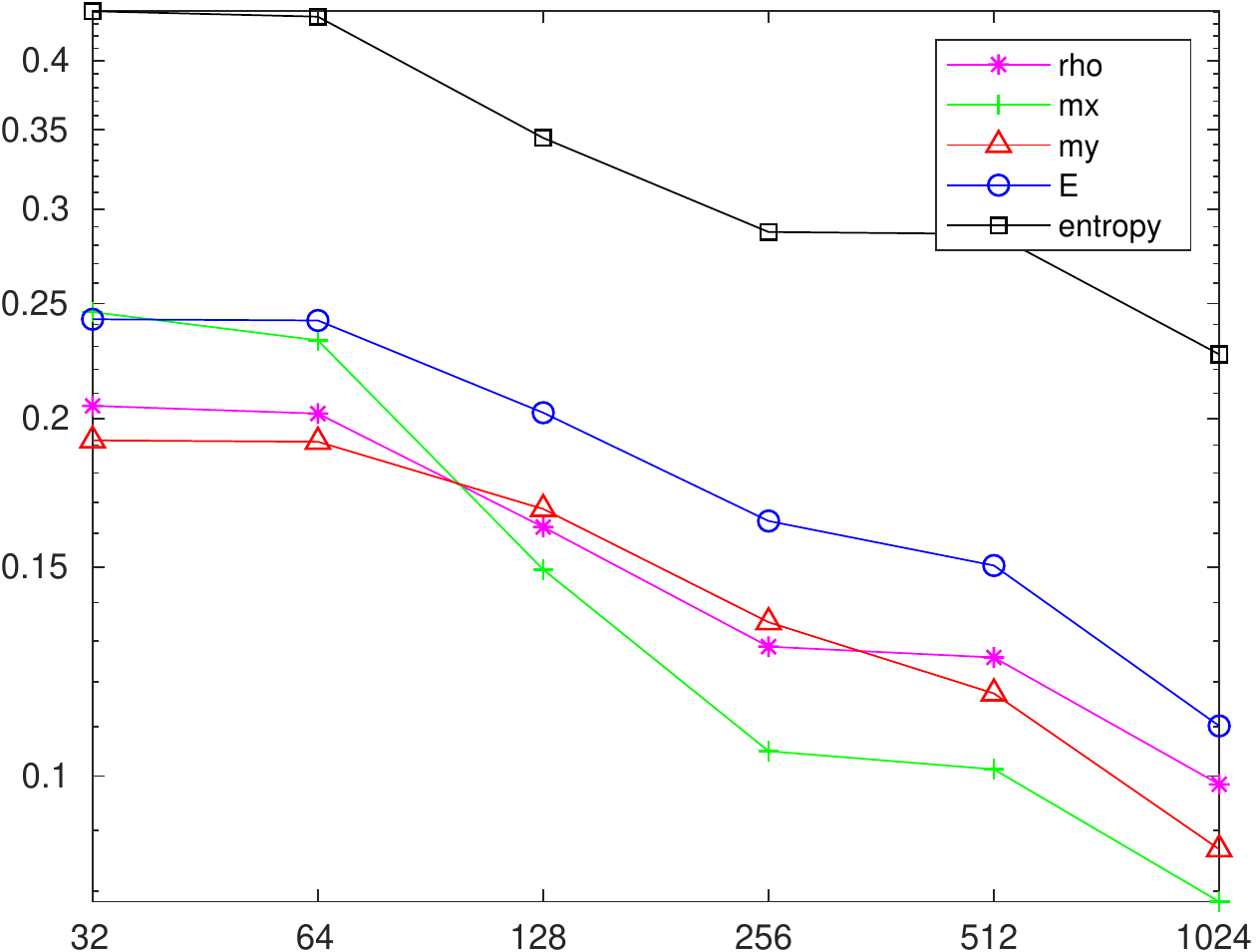}
		\caption{$E_1$}
	\end{subfigure}	
	\begin{subfigure}{0.243\textwidth}
		\includegraphics[width=\textwidth]{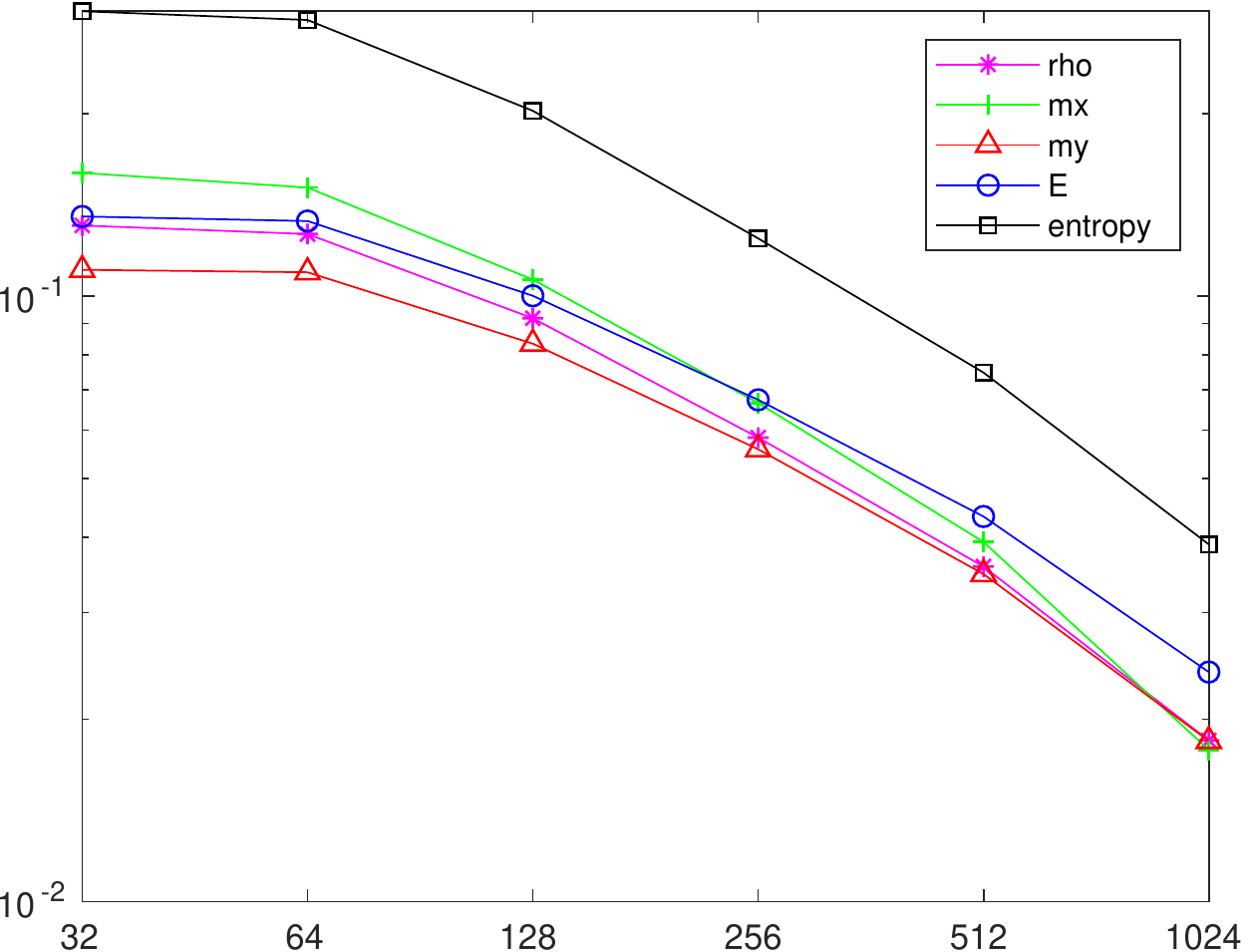}
		\caption{$E_2$}
	\end{subfigure}	
	\begin{subfigure}{0.243\textwidth}
		\includegraphics[width=\textwidth]{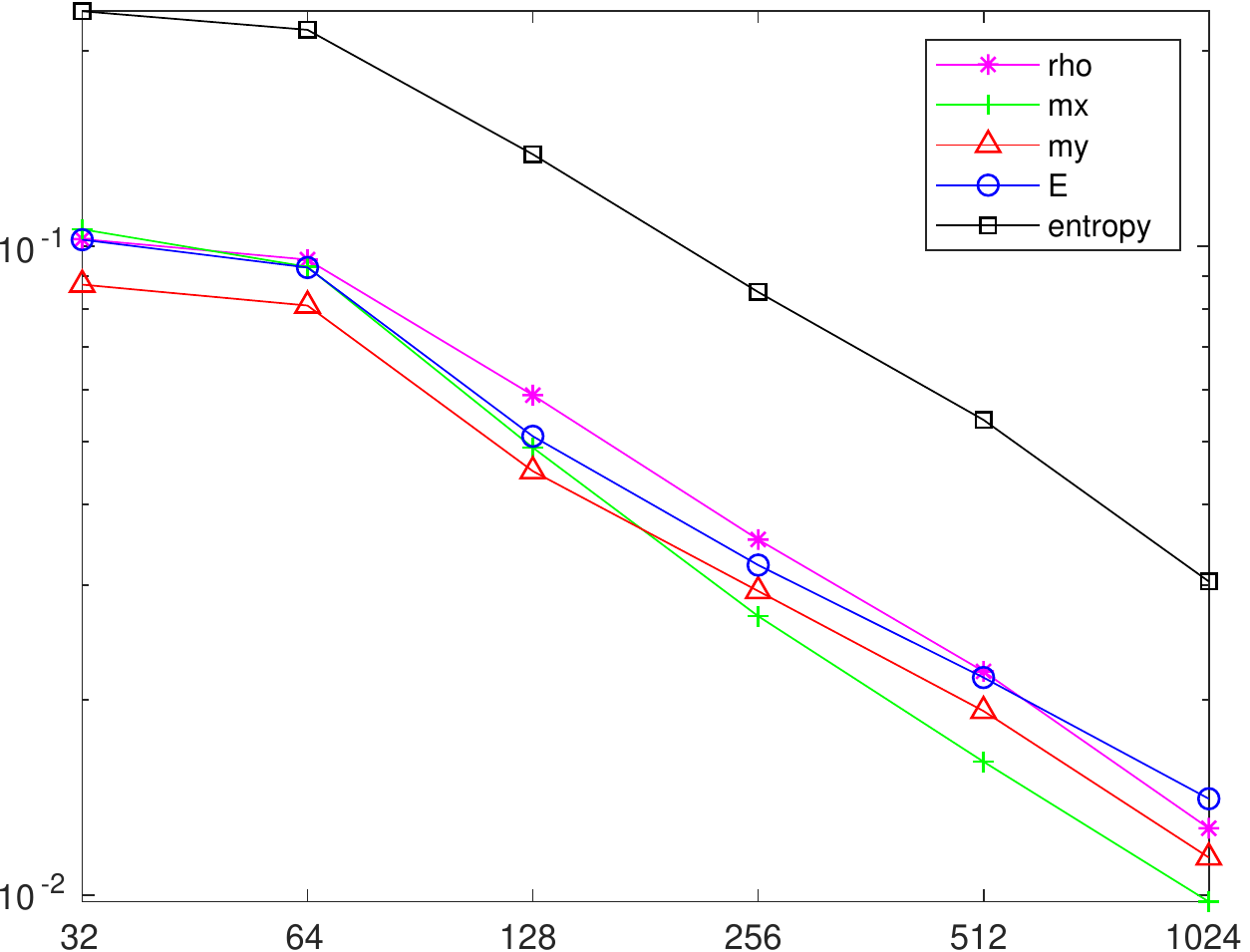}
		\caption{$E_3$}
	\end{subfigure}	
	\begin{subfigure}{0.243\textwidth}
		\includegraphics[width=\textwidth]{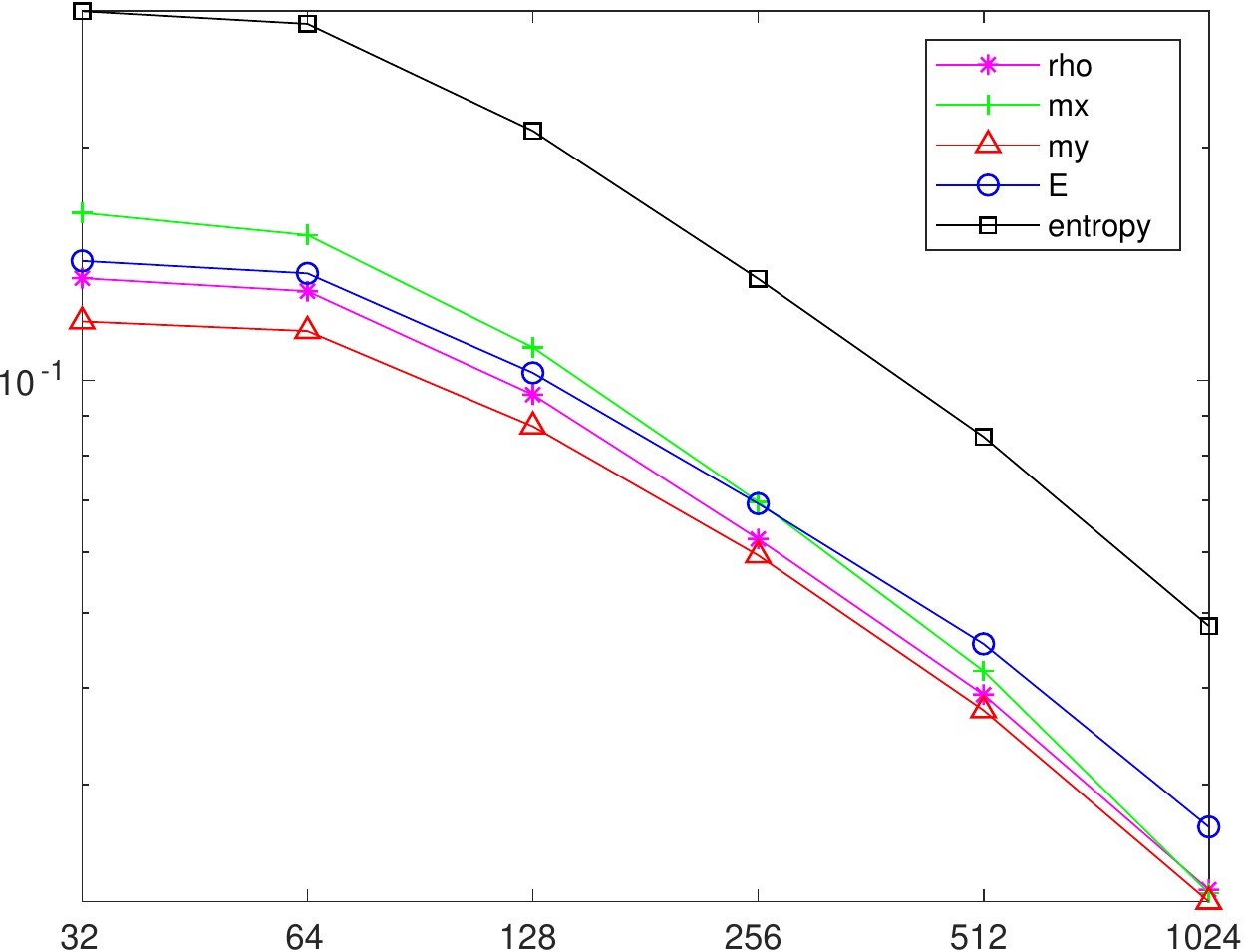}
		\caption{$E_4$}
	\end{subfigure}	
	\caption{\small{Example \ref{example:2D-KH}: the errors obtained on different meshes.}}\label{figure:2D-KH-error}
\end{figure}

\begin{table}[htbp]
	\centering
	\caption{Example \ref{example:2D-KH}: the errors and convergence rates for density.  } \label{tabel:2D-KH-density}
	\begin{tabular}{|c|cc|cc|cc|cc|}
		\hline
		\multirow{2}{*}{$n$} & \multicolumn{2}{c|}{$E_1$ } & \multicolumn{2}{c|}{$E_2$ } & \multicolumn{2}{c|}{$E_3$ }  & \multicolumn{2}{c|}{$E_4$ } \\
		\cline{2-9}
		& error & order &  error & order &  error & order  &  error & order \\
		\hline
		\hline
		32 & 0.2050 & - & 0.1307 & - & 0.1026 & - & 0.1355 & - \\
		
		64 & 0.2020 & 0.0215 & 0.1265 & 0.0468 & 0.0953 & 0.1052 & 0.1303 & 0.0561\\
		
		128 & 0.1621 & 0.3175 & 0.0918 & 0.4620 & 0.0589 & 0.6939 & 0.0958 & 0.4437\\
		
		256 & 0.1285 & 0.3345 & 0.0584 & 0.6542 & 0.0353 & 0.7393 & 0.0623 & 0.6202\\
		
		512 & 0.1259 & 0.0303 & 0.0357 & 0.7072 & 0.0221 & 0.6750 & 0.0392 & 0.6682\\
		
		1024 & 0.0984 & 0.3549 & 0.0185 & 0.9533 & 0.0127 & 0.8033 & 0.0219 & 0.8381\\
		\hline
	\end{tabular}
\end{table}

\begin{table}[htbp]
	\centering
	\caption{Example \ref{example:2D-KH}: the errors and convergence rates for entropy.  } \label{tabel:2D-KH-entropy}
	\begin{tabular}{|c|cc|cc|cc|cc|}
		\hline
		\multirow{2}{*}{$n$} & \multicolumn{2}{c|}{$E_1$ } & \multicolumn{2}{c|}{$E_2$ } & \multicolumn{2}{c|}{$E_3$ }  & \multicolumn{2}{c|}{$E_4$ } \\
		\cline{2-9}
		& error & order &  error & order &  error & order  &  error & order \\
		\hline
		\hline
		32 & 0.4405 & - & 0.2950 & - & 0.2302 & - & 0.3000 & - \\
		
		64 & 0.4359 & 0.0152 & 0.2854 & 0.0479 & 0.2155 & 0.0952 & 0.2889 & 0.0543\\
		
		128 & 0.3446 & 0.3392 & 0.2021 & 0.4980 & 0.1385 & 0.6375 & 0.2102 & 0.4593\\
		
		256 & 0.2871 & 0.2631 & 0.1244 & 0.6998 & 0.0850 & 0.7052 & 0.1353 & 0.6355\\
		
		512 & 0.2860 & 0.0055 & 0.0747 & 0.7356 & 0.0540 & 0.6544 & 0.0846 & 0.6779\\
		
		1024 & 0.2265 & 0.3366 & 0.0389 & 0.9423 & 0.0304 & 0.8274 & 0.0481 & 0.8135\\
		\hline
	\end{tabular}
\end{table}

\newgeometry{left=2.5cm,right=2.5cm,top=3cm,bottom=3cm}
\begin{figure}[t]
		\setlength{\abovecaptionskip}{-0.1cm}
	\setlength{\belowcaptionskip}{-0.2cm}
	\centering
	\includegraphics[width=\textwidth]{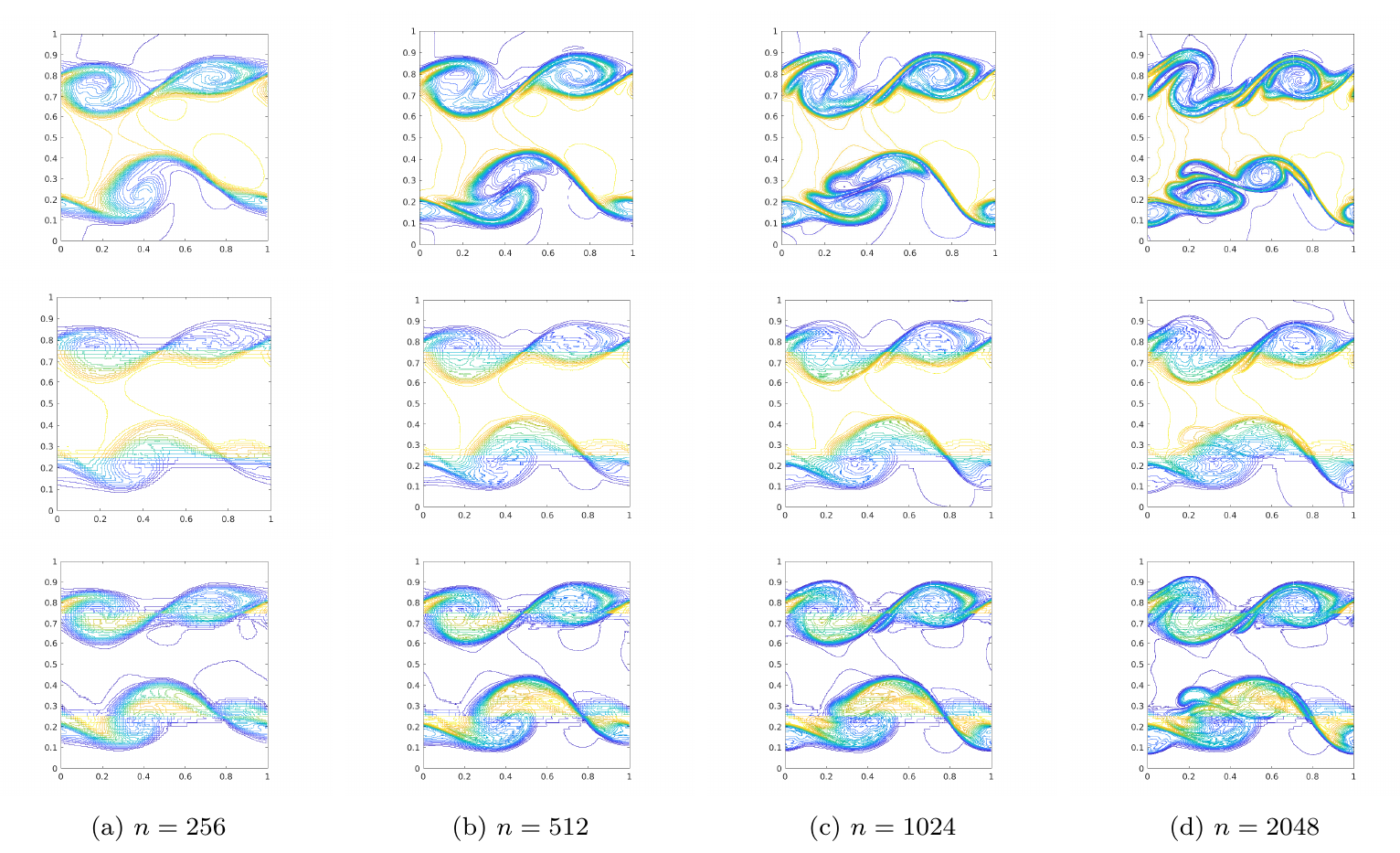}
	\caption{\small{Example \ref{example:2D-KH}: the density contours obtained on meshes with $n\times n$ cells. From top to bottom: density (top); Ces\'{a}ro averages of density (middle); the first variance of density (bottom). }}\label{figure:2D-KH-densitycontour}
\end{figure}

\begin{figure}[b]
		\setlength{\abovecaptionskip}{-0.1cm}
	\setlength{\belowcaptionskip}{-0.2cm}
	\centering
	\includegraphics[width=\textwidth]{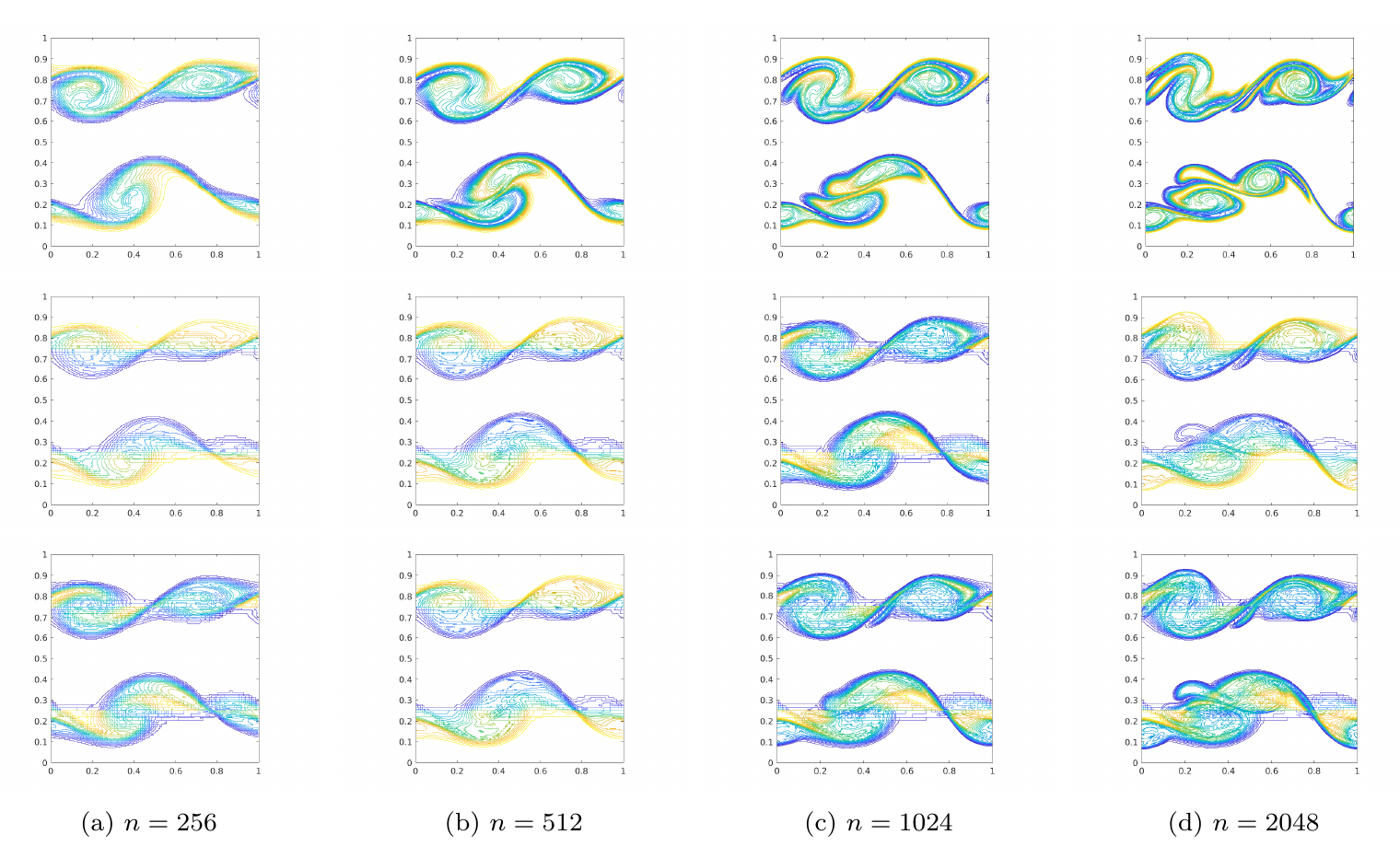}
	\caption{\small{Example \ref{example:2D-KH}: the entropy contours obtained on meshes with $n\times n$ cells. From top to bottom: entropy (top); Ces\'{a}ro averages of entropy (middle); the first variance of entropy (bottom). }}\label{figure:2D-KH-entropycontour}
\end{figure}
\restoregeometry

\begin{example}[\textbf{Richtmyer-Meshkov problem}]\label{example:2D-RM}\rm
	This example describes complex interactions of strong shocks with unstable interfaces.
	The initial data are given by
	\begin{equation*}
	u_1 = 0, \quad u_2 = 0, \quad
	p(x,0)
	=  \begin{cases}
	20, &  \mbox{if}~ r < 0.1, \\
	1, & \mbox{otherwise},
	\end{cases} \quad
	\rho(x,0)
	=  \begin{cases}
	2, &  \mbox{if}~ r < I(\vec x), \\
	1, & \mbox{otherwise},
	\end{cases}
	\end{equation*}
	where $r = \sqrt{(x_1 - 0.5)^2 + (x_2 - 0.5)^2}$ and the radial interface $I(\vec x) = 0.25 + \epsilon Y(\vec x)$ is perturbed by
	\begin{equation*}
	Y(\vec x) =  \sum_{m = 1}^{M} a^m \cos(\phi +  b^m ),
	\end{equation*}
	with $\phi = \arccos( (x_2 - 0.5)/ r)$.
	The parameters $a^m, b^m$ are arbitrary, but fixed numbers chosen such that
	 $a^m \in [0,1],\,\sum_{m = 1}^{M} a^m = 1$ and $b^m \in [-\pi, \pi],\,m = 1, \cdots, M$.
	In the simulation we set $\epsilon = 0.01, T = 4$ and $N = 2048$.
	
	Figure \ref{figure:2D-RM-error} shows the errors $E_1, E_2, E_3, E_4$ of $\rho, m_1, m_2, E, S$ obtained using different meshes $n = 32, \cdots, 1024$, see also Tables \ref{tabel:2D-RM-density}, \ref{tabel:2D-RM-entropy}.
	Moreover,  the contours of $\rho$ and $S$ are shown in Figures \ref{figure:2D-RM-densitycontour}, \ref{figure:2D-RM-entropycontour}, respectively.
	
	The figures and tables clearly indicate only a weak converge of single simulations obtained by the finite volume method \eqref{eq:semi-discrete-RPflux-phi}.
	Indeed, the numerical solutions do not converge strongly in $L^1$-norm.
	On the other hand, the Ces\`{a}ro  average and the first variance of the numerical solutions, as well as the Wasserstein distance of the corresponding Dirac distributions converge strongly in $L^1$-norm.
	These results again confirm our theoretical analysis on the convergence of the finite volume method \eqref{eq:semi-discrete-RPflux-phi}.
	In the Richtmyer-Meshkov case the limiting solution is not a weak solution but a dissipative measure-valued solution.
\end{example}

\begin{figure}[htbp]
	\setlength{\abovecaptionskip}{0.cm}
	\setlength{\belowcaptionskip}{-0.cm}
	\centering
	\begin{subfigure}{0.243\textwidth}
		\includegraphics[width=\textwidth]{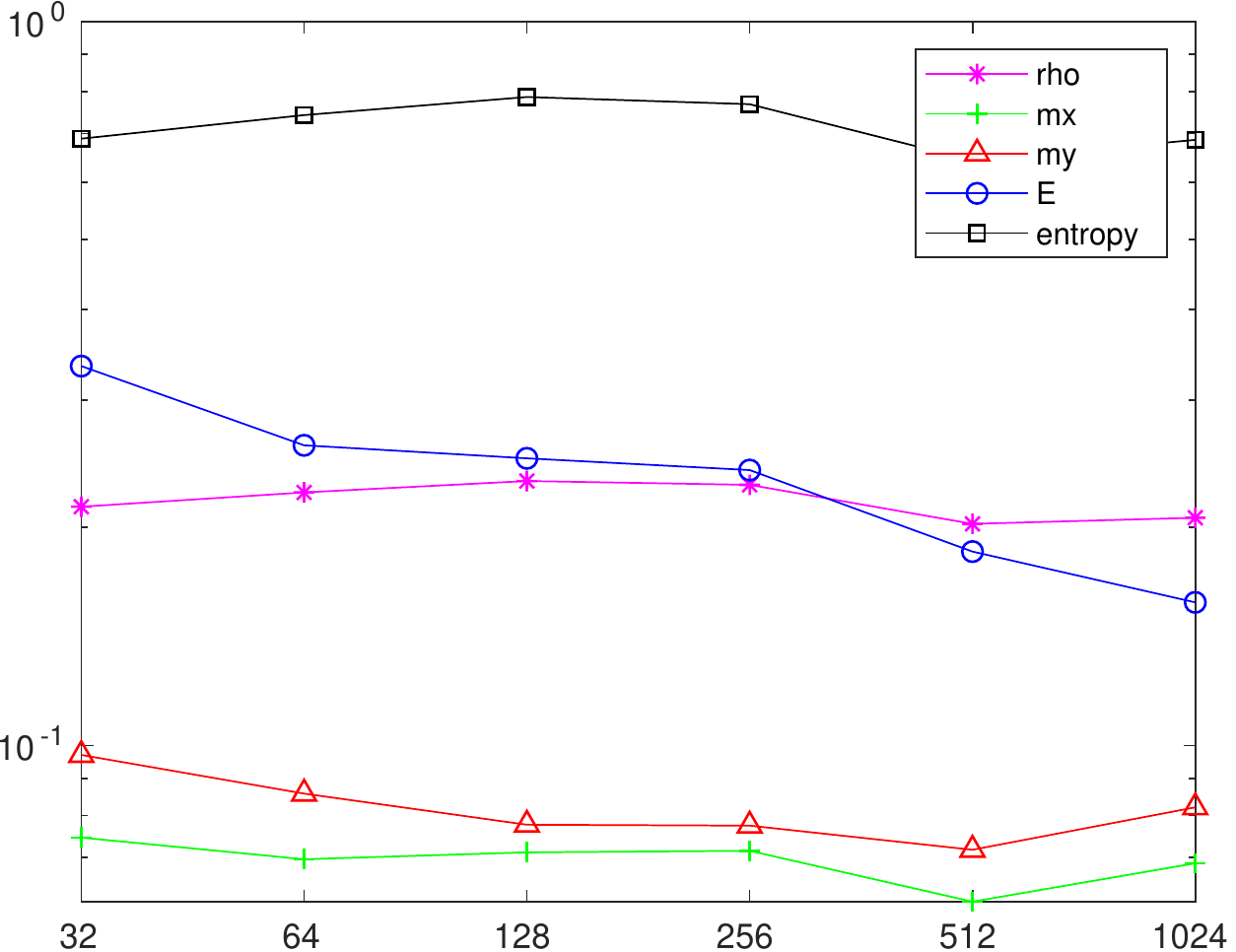}
		\caption{$E_1$}
	\end{subfigure}	
	\begin{subfigure}{0.243\textwidth}
		\includegraphics[width=\textwidth]{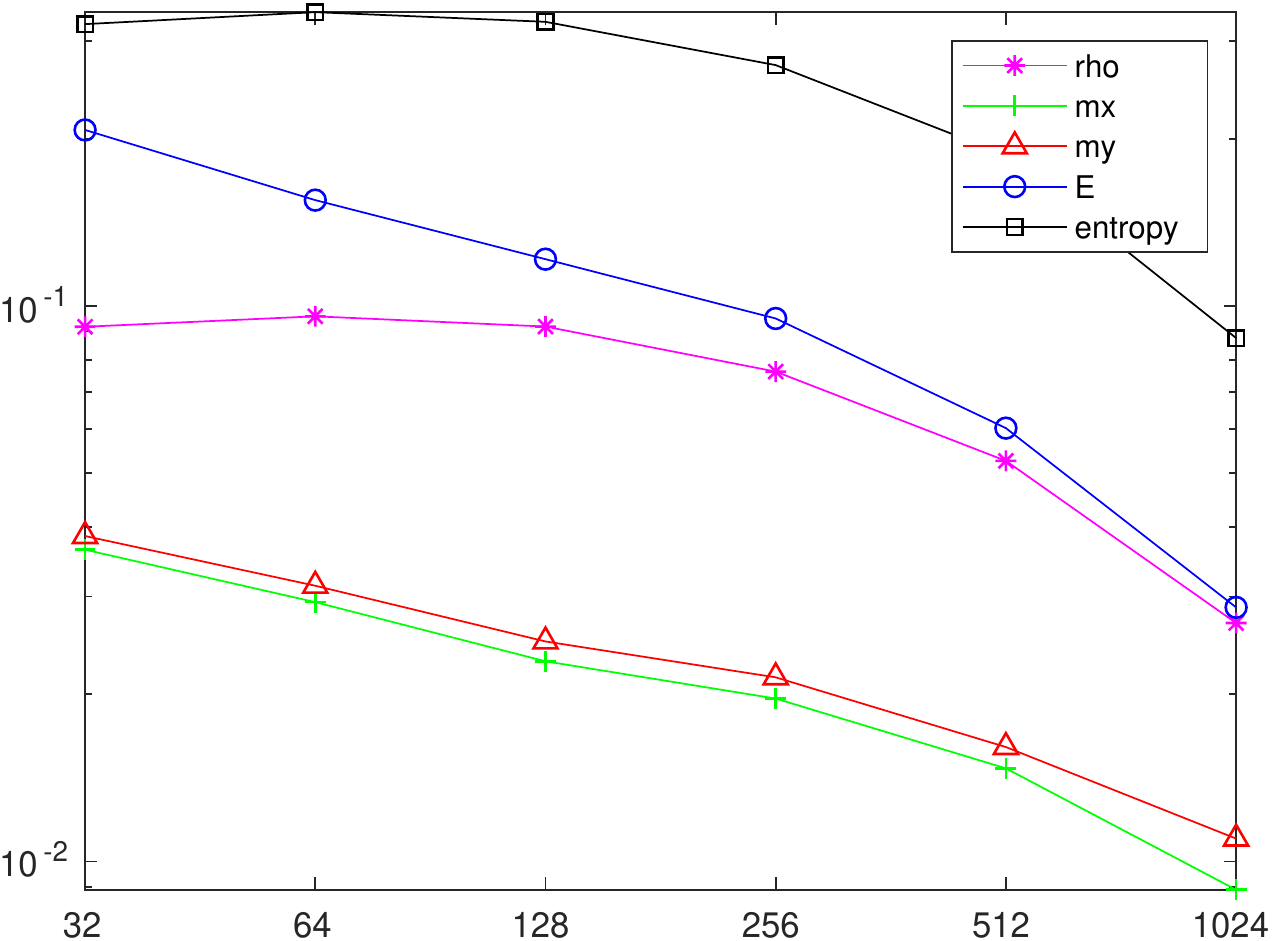}
		\caption{$E_2$}
	\end{subfigure}	
	\begin{subfigure}{0.243\textwidth}
		\includegraphics[width=\textwidth]{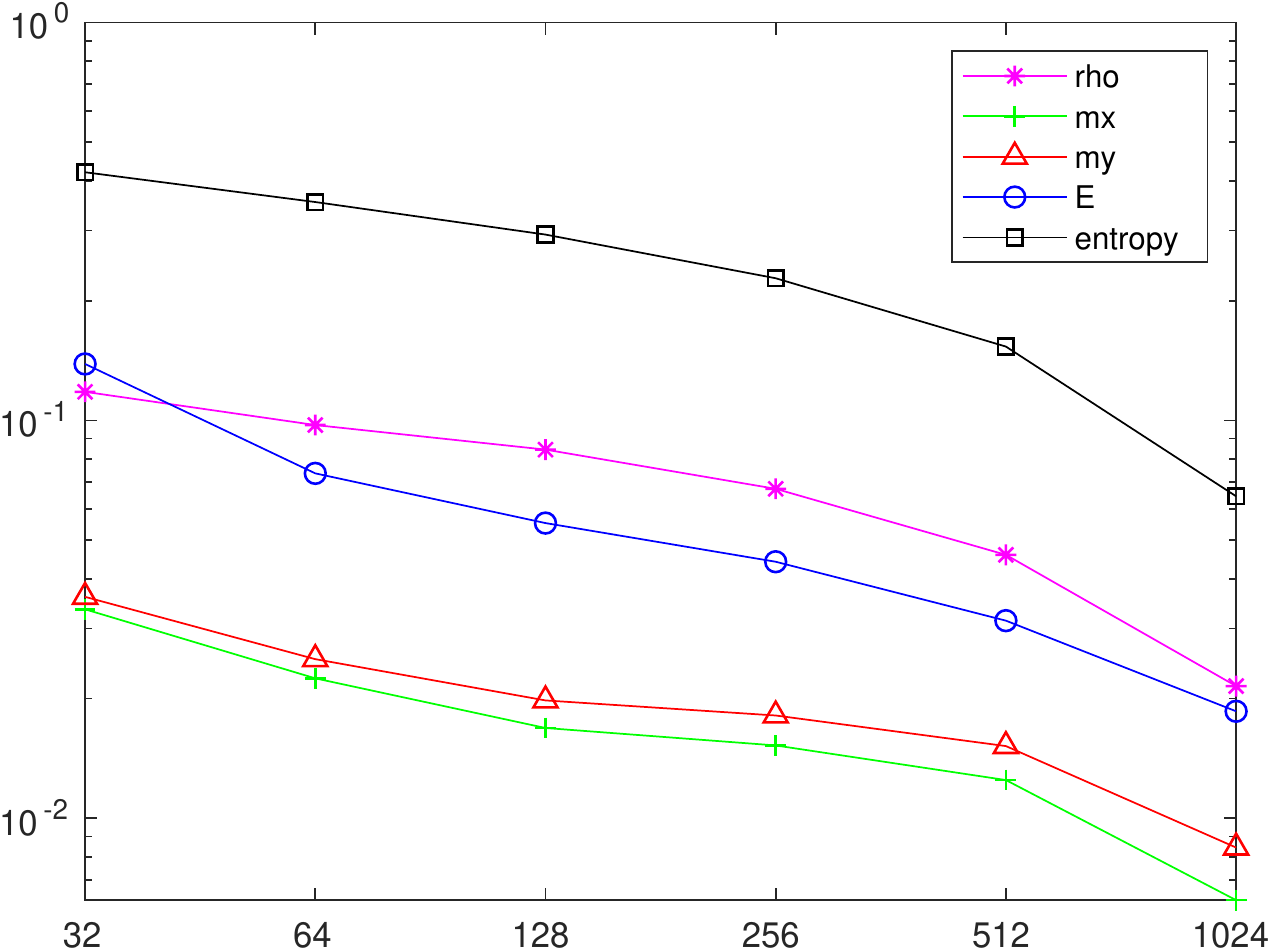}
		\caption{$E_3$}
	\end{subfigure}	
	\begin{subfigure}{0.243\textwidth}
		\includegraphics[width=\textwidth]{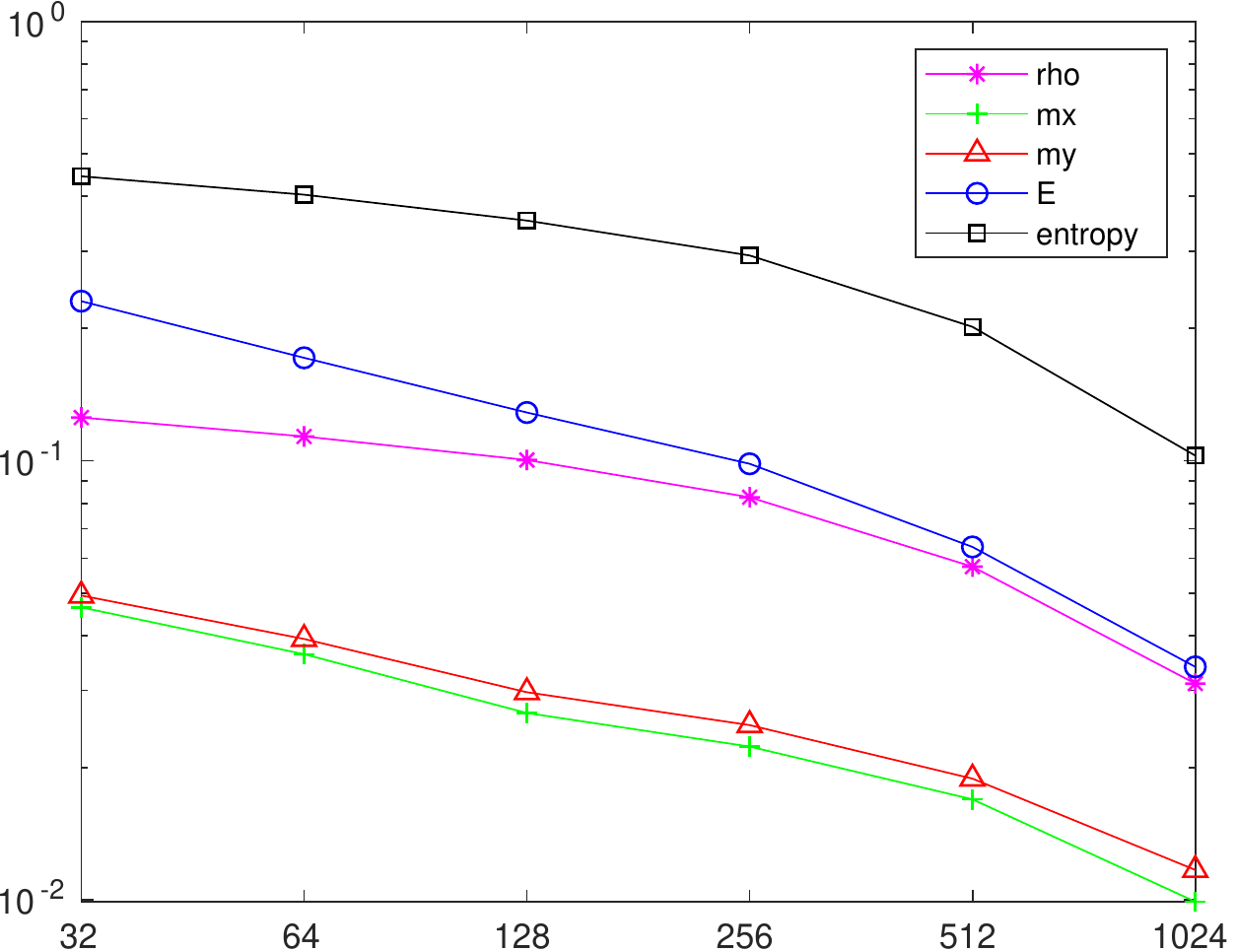}
		\caption{$E_4$}
	\end{subfigure}	
	\caption{\small{Example \ref{example:2D-RM}: the errors obtained  on different meshes.}}\label{figure:2D-RM-error}
\end{figure}

\newgeometry{left=2.5cm,right=2.5cm,top=3cm,bottom=3cm}
\begin{table}[t]
	\centering
	\caption{Example \ref{example:2D-RM}: the errors and convergence rates for density.  } \label{tabel:2D-RM-density}
	\begin{tabular}{|c|cc|cc|cc|cc|}
		\hline
		\multirow{2}{*}{$n$} & \multicolumn{2}{c|}{$E_1$ } & \multicolumn{2}{c|}{$E_2$ } & \multicolumn{2}{c|}{$E_3$ }  & \multicolumn{2}{c|}{$E_4$ } \\
		\cline{2-9}
		& error & order &  error & order &  error & order  &  error & order \\
		\hline
		\hline
		32 & 0.2136 & - & 0.0917 & - & 0.1180 & - & 0.1254 & - \\
		
		64 & 0.2236 & -0.0656 & 0.0958 & -0.0626 & 0.0973 & 0.2774 & 0.1136 & 0.1433\\
		
		128 & 0.2318 & -0.0523 & 0.0918 & 0.0613 & 0.0844 & 0.2051 & 0.1004 & 0.1782\\
		
		256 & 0.2290 & 0.0175 & 0.0762 & 0.2692 & 0.0673 & 0.3274 & 0.0826 & 0.2816\\
		
		512 & 0.2023 & 0.1790 & 0.0526 & 0.5337 & 0.0459 & 0.5512 & 0.0574 & 0.5249\\
		
		1024 & 0.2063 & -0.0280 & 0.0269 & 0.9694 & 0.0215 & 1.0959 & 0.0311 & 0.8835\\
		\hline
	\end{tabular}
\end{table}

\begin{table}[t]
	\centering
	\caption{Example \ref{example:2D-RM}: the errors and convergence rates for entropy.  } \label{tabel:2D-RM-entropy}
	\begin{tabular}{|c|cc|cc|cc|cc|}
		\hline
		\multirow{2}{*}{$n$} & \multicolumn{2}{c|}{$E_1$ } & \multicolumn{2}{c|}{$E_2$ } & \multicolumn{2}{c|}{$E_3$ }  & \multicolumn{2}{c|}{$E_4$ } \\
		\cline{2-9}
		& error & order &  error & order &  error & order  &  error & order \\
		\hline
		\hline
		32 & 0.6887 & - & 0.3215 & - & 0.4205 & - & 0.4440 & - \\
		
		64 & 0.7423 & -0.1082 & 0.3378 & -0.0716 & 0.3539 & 0.2487 & 0.4032 & 0.1391\\
		
		128 & 0.7862 & -0.0830 & 0.3246 & 0.0578 & 0.2931 & 0.2718 & 0.3521 & 0.1957\\
		
		256 & 0.7684 & 0.0330 & 0.2713 & 0.2588 & 0.2277 & 0.3646 & 0.2932 & 0.2639\\
		
		512 & 0.6425 & 0.2582 & 0.1869 & 0.5373 & 0.1533 & 0.5704 & 0.2017 & 0.5396\\
		
		1024 & 0.6860 & -0.0945 & 0.0876 & 1.0943 & 0.0645 & 1.2502 & 0.1027 & 0.9737\\
		\hline
	\end{tabular}
\end{table}


\begin{figure}[b]
		\setlength{\abovecaptionskip}{-0.1cm}
	\setlength{\belowcaptionskip}{-0.2cm}
	\centering
	\includegraphics[width=\textwidth]{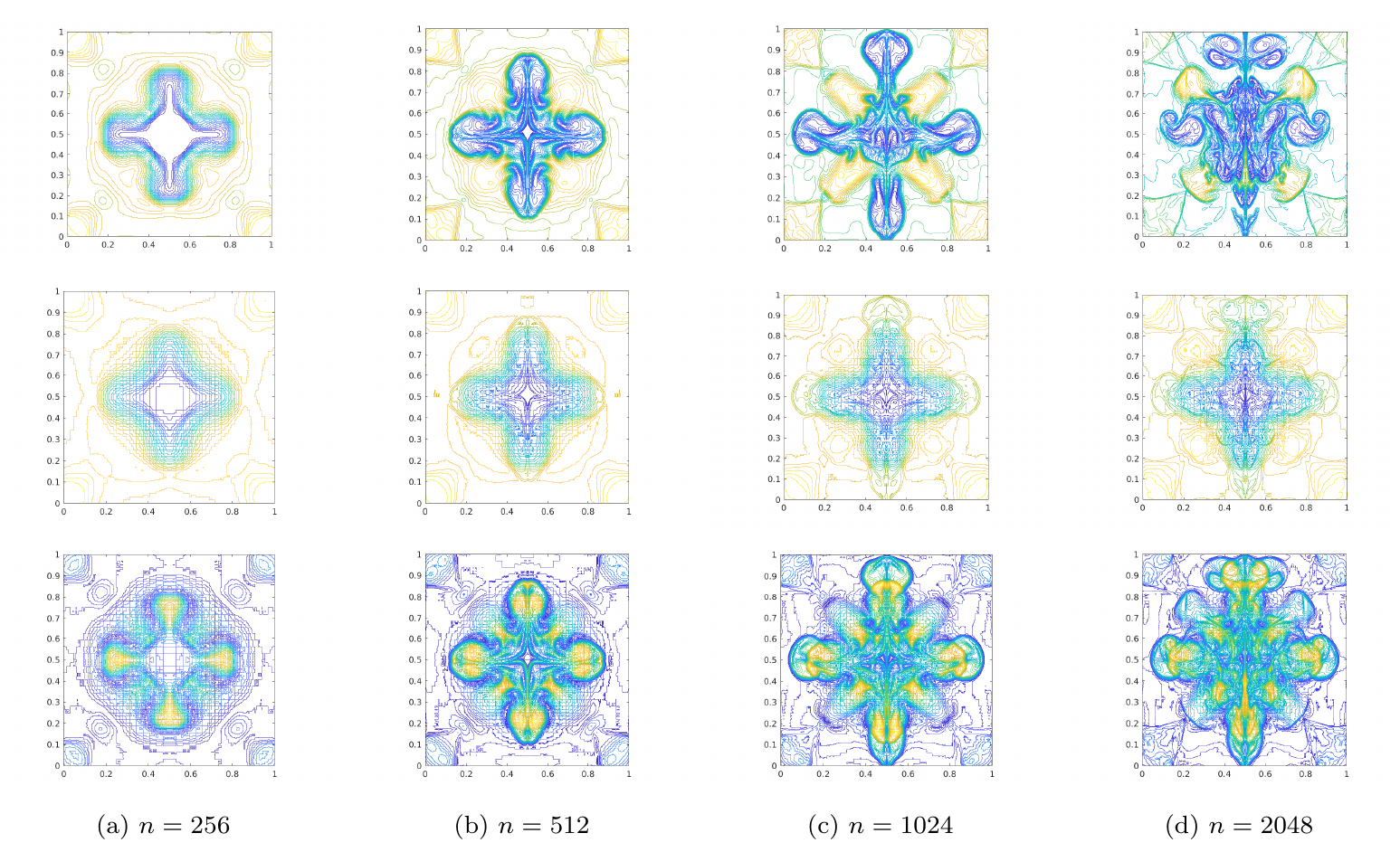}
	\caption{\small{Example \ref{example:2D-RM}: the density contours obtained on meshes with $n \times n$ cells. From top to bottom: density (top); Ces\'{a}ro averages of density (middle); the first variance of density (bottom). }}\label{figure:2D-RM-densitycontour}
\end{figure}
\restoregeometry

\begin{figure}[htpb]
		\setlength{\abovecaptionskip}{-0.1cm}
	\setlength{\belowcaptionskip}{-0.2cm}
	\centering
	\includegraphics[width=\textwidth]{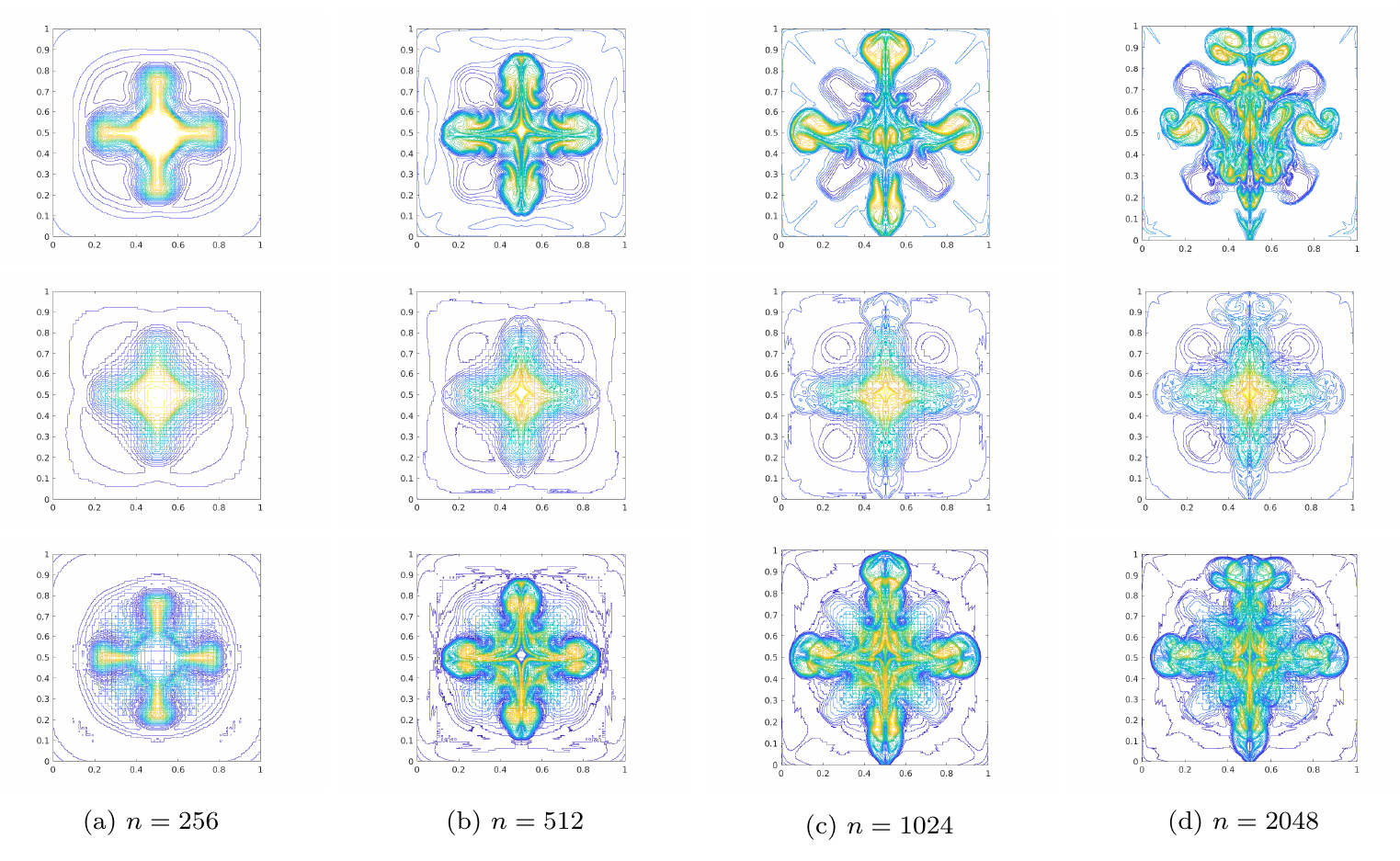}
	\caption{\small{Example \ref{example:2D-RM}: the entropy contours obtained on meshes with $n \times n$ cells. From top to bottom: entropy (top); Ces\'{a}ro averages of entropy (middle); the first variance of entropy (bottom). }}\label{figure:2D-RM-entropycontour}
\end{figure}


\section{Conclusions}
\noindent We have shown the convergence of the Godunov finite volume method \eqref{eq:semi-discrete-RPflux-phi}, which is based on the exact solution of a local Riemann problem for the complete compressible Euler system.
Hereby we have only assumed that our numerical solutions belong to a physically non-degenerate region, i.e.~we have a uniform lower bound on density and an upper bound on energy. The latter is equivalent to the existence of a uniform lower bound on entropy Hessian matrix, see  Lemma~\ref{lemma:equivalent-statement}.
Using the fact that the finite volume method \eqref{eq:semi-discrete-RPflux-phi} is entropy stable, the entropy inequality together with the explicit lower bound of the entropy Hessian matrix yield the weak BV estimate \eqref{eq:U-l2norm-bound}.
Further, we have shown that the difference between the exact solution of the local Riemann problem and its initial data, i.e.~numerical solution at time $t_n$, can be controlled by the jump of numerical solution itself.
Consistency of the method is proved in Theorem~\ref{theorem:consistency}.
Weak convergence to a generalized solution, the dissipative measure-valued (DMV) solution, is showed in Theorem~\ref{thm:weak-convergence}.
Applying the novel tool of $\mathcal{K}$-convergence we obtained strong convergence to the expected value and first variance of a dissipative measure-valued solution, see Theorem~\ref{thm:K-convergence}.
Theorem~\ref{thm:strong-convergence} summarizes the cases of strong convergence of numerical solutions of \eqref{eq:semi-discrete-RPflux-phi}.
The latter happens if the limit is a weak or $C^1$ solution to the Euler system \eqref{eq:multiD-Euler}.
In addition, applying  the DMV-strong uniqueness principle \cite{Brezina-Feireisl:2018a} we also have strong convergence of our numerical solutions on the lifespan of the strong solution.

Numerical results for the spiral problem, the Kelvin-Helmholtz problem and the Richtmyer-Meshkov problem confirm results of theoretical analysis.
In particular, we observe that for highly oscillatory limit, as is the case of the Richtmyer-Meshkov test, single numerical solutions do not  converge strongly.
On the other hand, we obtain the strong convergence of observable quantities, i.e. the expected values and first variances.

Although the Godunov finite volume method is one of the most classical schemes for hyperbolic conservation laws, its numerical analysis for multidimensional systems still remains open in general.
The main goal of this paper was to fill this gap and illustrate the application of the recently developed concepts of generalized DMV solution and $\mathcal{K}$-convergence on an iconic example, the multidimensional Euler system.
In future it will be interesting to extend the obtained results to general multidimensional hyperbolic conservation laws.

\section*{Acknowledgments}
\noindent M.L. has been funded by the German Science Foundation (DFG) under the collaborative research projects TRR SFB 165 (Project A2) and TRR SFB 146 (Project C5).
Y.Y. has been funded by Sino-German (CSC-DAAD) Postdoc Scholarship Program in 2020 - Project number 57531629.

\appendix
\renewcommand{\appendixname}{}

\phantomsection

\section{Lipschitz continuity of $F$} \label{section:Lipschitzcontinuity}
\noindent The aim of this section is to give the estimates of $\|\defjump{(\vec F_1)_h}\|$, also denoted by $\|\defjump{F_h}\|$.

\begin{lemma} \label{lemma:useful-001}
	It holds
	\begin{align}
		& \defave{|\vec u|^2} \leq (\|\defave{\vec u} \| +  \|\defjump{\vec u}\|)^2, \quad
		 \defave{a^2} \leq 2 \defave{a}^2, \quad
		 \defave{p} \defave{\rho}^{-1}  \leq  \frac{4}{\gamma} \defave{a}^2,\\
		& \defave{\rho |\vec u|^2} \defave{\rho}^{-1} \leq  2(\|\defave{\vec u} \| +  \|\defjump{\vec u}\|)^2, \quad
		\| \defave{\vec m} \| \defave{\rho}^{-1} \leq \|\defave{\vec u} \| + \frac12 \|\defjump{\vec u}\|.
	\end{align}
\end{lemma}

\begin{proof}
	If $f \geq 0$ then $|\defjump{f}| \leq \min(f_L,f_R) \leq 2\defave{f}$.
	After some calculations we obtain
	\begin{align*}
		& \defave{a^2} = \defave{a}^2 + \frac14\defjump{a}^2 \leq 2 \defave{a}^2,\quad
		 \defave{|\vec u|^2} = \defave{\vec u}\cdot \defave{\vec u} + \frac14 \defjump{\vec u}\cdot  \defjump{\vec u} \leq (\|\defave{\vec u} \| +  \|\defjump{\vec u}\|)^2.
	\end{align*}
	Moreover, we obtain
	\begin{align*}
		& \defave{p} = \defave{\frac{p}{\rho} \cdot \rho} = \defave{\frac{p}{\rho} } \defave{\rho} + \frac14 \defjump{\frac{p}{\rho} } \defjump{\rho} \leq 2\defave{\frac{p}{\rho} } \defave{\rho} = \frac{2}{\gamma}\defave{a^2} \defave{\rho} , \\
		& \defave{\rho |\vec u|^2} = \defave{\rho} \defave{|\vec u|^2} + \frac14 \defjump{\rho} \defjump{|\vec u|^2} \leq 2\defave{\rho} \defave{|\vec u|^2} , \\
		& \| \defave{\vec m} \| = \|\defave{\rho} \defave{\vec u} + \frac14 \defjump{\rho} \defjump{\vec u}\|  \leq \defave{\rho} \|\defave{\vec u}\| + \frac12 \defave{\rho} \|\defjump{\vec u}\|,
	\end{align*}
	which concludes the proof.
\end{proof}

\begin{lemma} \label{lemma:jump-U-V}
It holds
\begin{align}
	& \|\defjump{\vec u}\|
	\leq   \defave{\rho}^{-1} \| \defave{\vec u}\| \cdot |\defjump{\rho}| + \defave{\rho}^{-1} \|\defjump{\vec m}\|,  \label{eq:jump-u-U}\\
	&  |\defjump{p}|  \lesssim   ( \|\defave{\vec u} \| + \|\defjump{\vec u}\|)^2 \cdot |\defjump{\rho} | + ( \|\defave{\vec u} \| + \|\defjump{\vec u}\|) \cdot \|\defjump{\vec m}\| + |\defjump{E}|. \label{eq:jump-p-U}
\end{align}
\end{lemma}

\begin{proof}
Some manipulations give
\begin{align*}
&\defjump{\vec m} = \defave{\rho} \defjump{\vec u} + \defjump{\rho} \defave{\vec u},
\end{align*}
which implies 
\begin{align*}
	& \|\defjump{\vec u}\|
	= \left\|\frac{\defjump{\vec m}- \defjump{\rho} \defave{\vec u}}{\defave{\rho}} \right\|
	\leq   \defave{\rho}^{-1} \| \defave{\vec u}\| \cdot |\defjump{\rho}| + \defave{\rho}^{-1} \|\defjump{\vec m}\| .
\end{align*}

On the other hand, we have
\begin{align*}
\defjump{\frac{|\vec m|^2}{\rho} \cdot \rho}
& = \defjump{|\vec m|^2} = \defave{\vec m} \cdot  \defjump{\vec m}
	= \defave{\frac{|\vec m|^2}{\rho}} \cdot \defjump{\rho} +  \defjump{\frac{|\vec m|^2}{\rho}} \cdot \defave{\rho},
\end{align*}
which gives
\begin{equation*}
\left|\defjump{ \frac{|\vec m|^2}{\rho}} \right| \leq \defave{\rho}^{-1}  \defave{\frac{|\vec m|^2}{\rho}} \cdot |\defjump{\rho} | +  \defave{\rho}^{-1} \|\defave{\vec m}\| \cdot \|\defjump{\vec m}\|.
\end{equation*}
Consequently, we obtain
\begin{align*}
 |\defjump{p}|
 & =  (\gamma-1) |\defjump{E - \frac{|\vec m|^2}{2\rho}}| \leq  (\gamma-1) |\defjump{E}| +\frac{\gamma-1}{2} |\defjump{ \frac{|\vec m|^2}{\rho}}| \nonumber \\
  & \leq   \frac{\gamma-1}{2} \defave{\rho}^{-1}   \defave{\frac{|\vec m|^2}{\rho}} \cdot |\defjump{\rho} | + \frac{\gamma-1}{2} \defave{\rho}^{-1} \|\defave{\vec m}\| \cdot \|\defjump{\vec m}\| +(\gamma-1) |\defjump{E}|,
\end{align*}
which concludes the proof.
\end{proof}

\begin{lemma}
It holds
{\small
\begin{align*}
\left|\defjump{F^{(1)}}\right| & \leq \|\defjump{\vec m}\|, \\
\left|\defjump{F^{(2)}}\right| & \lesssim  (\|\defave{\vec u} \| +  \|\defjump{\vec u}\|)^2 \cdot |\defjump{\rho} |
									+ (\|\defave{\vec u} \| +  \|\defjump{\vec u}\|) \cdot \|\defjump{\vec m}\|
									+  |\defjump{E}|,\\
\left|\defjump{F^{(3)}}\right| & \lesssim  (\|\defave{\vec u} \| +  \|\defjump{\vec u}\|)^2 \cdot |\defjump{\rho} |
									+ (\|\defave{\vec u} \| +  \|\defjump{\vec u}\|) \cdot \|\defjump{\vec m}\|,\\
\left|\defjump{F^{(4)}}\right| & \lesssim  (\|\defave{\vec u} \| +  \|\defjump{\vec u}\|)^2 \cdot |\defjump{\rho} |
									+ (\|\defave{\vec u} \| +  \|\defjump{\vec u}\|) \cdot \|\defjump{\vec m}\|,\\
\left|\defjump{F^{(5)}}\right| & \lesssim  \left( (\|\defave{\vec u} \| +  \|\defjump{\vec u}\|)^2 + \defave{a}^2 \right)  \| \defave{\vec u}\|\cdot  |\defjump{\rho}|
									 +  (\|\defave{\vec u} \| +  \|\defjump{\vec u}\|)^2\cdot  \|\defjump{\vec m}\|
									+ \|\defave{\vec u}\| \cdot \defjump{E}.	
\end{align*}
}%
\end{lemma}

\begin{proof}
	Some calculations give
	{\small
		\begin{align*}
		\left|\defjump{F^{(1)}}\right| & = |\defjump{\rho u}| \leq \| \defjump{ \vec m} \|, \\
		\left|\defjump{F^{(2)}}\right| & = |\defjump{\rho u^2+p}| \leq \defave{u^2} \cdot  |\defjump{\rho}| + 2  \defave{\rho} \cdot |\defave{u}| \cdot |\defjump{u}| + |\defjump{p}|\\
		& \leq \defave{|\vec u|^2} \cdot  |\defjump{\rho}| + 2  \defave{\rho} \cdot \|\defave{\vec u}\| \cdot \|\defjump{\vec u}\| + |\defjump{p}|,\\
		\left|\defjump{F^{(3)}}\right| & = |\defjump{\rho uv}| \leq |\defave{uv}| \cdot |\defjump{\rho}| + \defave{\rho} \cdot |\defave{u}| \cdot |\defjump{v}|  + \defave{\rho} \cdot |\defave{v}| \cdot |\defjump{u}|   \\
		& \leq  \defave{|\vec u|^2} \cdot |\defjump{\rho}| + 2\defave{\rho} \cdot \|\defave{\vec u}\| \cdot \|\defjump{\vec u} \|,   \\
		\left|\defjump{F^{(4)}}\right| &  = |\defjump{\rho uw}| \leq |\defave{uw}| \cdot |\defjump{\rho}| + \defave{\rho} \cdot |\defave{u}| \cdot |\defjump{w}|  + \defave{\rho} \cdot |\defave{w}| \cdot |\defjump{u}|   \\
		& \leq  \defave{|\vec u|^2} \cdot |\defjump{\rho}| + 2\defave{\rho} \cdot \|\defave{\vec u}\| \cdot \|\defjump{\vec u} \|
		\end{align*}
	}%
	and
	{\small
		\begin{align*}
		\left|\defjump{F^{(5)}}\right| & = \left|\defjump{ u(E+p) } \right| \leq \defave{E+p} \cdot |\defjump{u}| + |\defave{u}| \cdot |\defjump{E}| + |\defave{u}| \cdot |\defjump{p}|	\\
		& \leq \defave{E+p} \cdot \|\defjump{\vec u}\| + \|\defave{\vec u}\| \cdot |\defjump{E}| + \|\defave{\vec u}\| \cdot |\defjump{p}|.
		\end{align*}
	}%
	Combining Lemma~\ref{lemma:useful-001} and Lemma~\ref{lemma:jump-U-V} concludes the proof.
\end{proof}

\section{Lower bound of the entropy Hessian matrix} \label{section:lowerbound}
\noindent In this section we give the lower bound of entropy Hessian matrix $\frac{d^2 \eta}{d \vec U^2}$ for $d = 3$.
Harten \cite{Harten:1983a} has already given the explicit expression of $\frac{d^2 \eta}{d \vec U^2}$ in two-dimensional case.
Analogously, we can obtain its three-dimensional explicit expression with $\eta = -\rho \chi(S), \chi \in C^2(\mathbb{R})$
\begin{align*}
	&	\frac{d^2 \eta}{d \vec U^2} =  \frac{d \vec \nu}{d \vec U}  =  \left(\frac{\gamma-1}{p}\right)^2  \rho \chi'(S)\cdot \vec D,
\end{align*}
where 
{
\begin{align*}
& \vec D^{(1)} :=
	\begin{pmatrix}
		\frac{|\vec u|^4 }4 + \frac{a_*^4}{\gamma} - R(\frac{|\vec u|^2}2   - a_*^2)^2 		\\
		-u \left(\frac{|\vec u|^2}{2}(1-R) + Ra_*^2 \right)	\\
		-v \left( \frac{|\vec u|^2}{2}(1-R) + Ra_*^2 \right) \\
		-w \left(\frac{|\vec u|^2}{2}(1-R) + Ra_*^2 \right)		\\
		\frac{|\vec u|^2}{2} (1-R) -  a_*^2 \left( \frac1{\gamma}  - R\right)
	\end{pmatrix},  \quad a_*^2 = \frac{\gamma}{\gamma-1} \cdot \frac{p}{\rho}, \quad R = \frac{\chi{''}(S)}{\chi'(S)},\\
& \vec D^{(2)} :=
	\begin{pmatrix}
		-u \left(\frac{|\vec u|^2}{2}(1-R) + Ra_*^2 \right) 	\\
		u^2(1-R) + a_*^2 /\gamma \\
		uv (1-R) \\
		uw (1-R) \\
		-u(1-R)
	\end{pmatrix}, \quad
	\vec D^{(3)} :=
	\begin{pmatrix}
		-v \left(\frac{|\vec u|^2}{2}(1-R) + Ra_*^2 \right) 	\\
		uv (1-R) \\
		v^2(1-R) + a_*^2 /\gamma \\	
		vw (1-R) \\
		-v(1-R)
	\end{pmatrix},\\
&	\vec D^{(4)} :=
	\begin{pmatrix}
		-w \left(\frac{|\vec u|^2}{2}(1-R) + Ra_*^2 \right) 	\\
		uw (1-R) \\
		vw (1-R) \\
		w^2(1-R) + a_*^2 /\gamma \\
		-w(1-R)
	\end{pmatrix},\quad
	\vec D^{(5)} :=
	\begin{pmatrix}
		\frac{|\vec u|^2}{2} (1-R) -  a_*^2 \left( \frac1{\gamma}  - R\right) 	\\
		-u(1-R) \\
		-v(1-R) \\
		-w(1-R)\\
		1-R
	\end{pmatrix}.
\end{align*}}%
Moreover, the symmetric matrix $\vec D$ is positive definite if and only if $R < 1/\gamma$.

Consider the special case $\eta = -\rho S/(\gamma-1)$, i.e., $\chi'(S) = 1/(\gamma-1)$ and $\chi''(S) = 0$.
After some calculations we obtain the characteristic function $f(\lambda)$ of the matrix $\vec D$
\begin{align*}
	f(\lambda) & = \left(\lambda - \frac{a^2}{\gamma} \right)^2 \cdot g(\lambda) ,
\end{align*}
where
\begin{align*}
	g(\lambda)
	& = \lambda^3 - \frac{1}{4\gamma}  \bigg( 4 a^4 + 4 a^2 + \gamma  (|\vec u|^2 +2)^2\bigg) \lambda^2   \\
	& \quad +\frac{a^2}{4 \gamma^2} \bigg( 4 a^4 - 4 a^2 + \gamma \big[4( |\vec u|^2 + 1) a^2 + (|\vec u|^2 +2)^2\big]\bigg) \lambda -\frac{a^{6} (\gamma-1)}{\gamma^3}.
\end{align*}
Taking derivation of $g(\lambda)$ gives
\begin{align*}
	g'(\lambda)
	& = 3\lambda^2 - \frac{1}{2\gamma}  \bigg( 4 a^4 + 4 a^2 + \gamma  (|\vec u|^2 +2)^2\bigg) \lambda   \\
	& \quad +\frac{a^2}{4 \gamma^2} \bigg( 4 a^4 - 4 a^2 + \gamma \big[4( |\vec u|^2 + 1) a^2 + (|\vec u|^2 +2)^2\big]\bigg) .
\end{align*}

Denote the roots of $g(\lambda) = 0$ by $\lambda_1, \lambda_2, \lambda_3$ with $\lambda_1 < \lambda_2 < \lambda_3$, the roots of $g'(\lambda) = 0$ by $\lambda_1^*$ and $\lambda_2^*$ with 
 \begin{equation} \label{eq:eigen-df}
 \lambda_1^* \leq  \frac{1}{12 \gamma} \left(4 a^4 + 4 a^2 + \gamma  (|\vec u|^2 +2)^2 \right) \leq \lambda_2^*.
 \end{equation}
As known, $\frac{d^2 \eta}{d \vec U^2} $ is positive definite \cite{Harten:1983a},
which implies
\begin{equation}
0 < \lambda_1 \leq \lambda_1^* \leq \lambda_2 \leq \lambda_2^* \leq \lambda_3.
\end{equation}

\begin{lemma} \label{lemma:eigen-upbound}
	It holds
	\begin{equation}
		\lambda_1 \leq \frac{a^2}{\gamma} (\gamma-1)^{1/3} < \lambda_2.
	\end{equation}
\end{lemma}

\begin{remark}
With the help of \eqref{eq:eigen-df} and the fact
\begin{equation*}
\begin{cases}
g(\lambda)  < 0,  ~\mbox{if}~ \lambda \in (0,\lambda_1); \quad &
g(\lambda)  > 0,  ~\mbox{if}~ \lambda \in (\lambda_1,\lambda_2);\\
g(\lambda)  < 0,  ~\mbox{if}~ \lambda \in (\lambda_2,\lambda_3);\quad &
g(\lambda)  > 0,  ~\mbox{if}~ \lambda \in (\lambda_3,0);
\end{cases}
\end{equation*}
we will prove Lemma~\ref{lemma:eigen-upbound} by showing
\begin{align*}
g\left( \frac{a^2}{\gamma} (\gamma-1)^{1/3} \right) \geq 0, \quad \frac{a^2}{\gamma} (\gamma-1)^{1/3}  < \frac{1}{12\gamma} \left(4 a^4 + 4 a^2 + \gamma  (|\vec u|^2 +2)^2\right)\leq \lambda_2^* \leq \lambda_3.
\end{align*}
\end{remark}

\begin{proof}
Denote  $(\gamma-1)^{1/3}$ by $\gamma_*$, i.e., $\gamma = \gamma_*^3 + 1$.
Then $\gamma \in (1,2]$ implies
\begin{equation}
\gamma_* \in (0,1].
\end{equation}
It is easy to show that
{\small
\begin{align*}
g\left( \frac{a^2 \gamma_*}{\gamma}  \right)
& =  \frac{a^4\gamma_{*}}{4\gamma^3} \bigg( (1-\gamma_{*}) \cdot (\gamma |\vec u|^4 + 4 \gamma |\vec u|^2) + 4 \gamma a^2 |\vec u|^2+ 4(1-\gamma_*) \cdot (a^4 - \gamma_{*}(1+\gamma_{*}) a^2 + \gamma) \bigg)\\
& \geq \frac{a^4\gamma_{*}(1-\gamma_*)}{\gamma^3} \bigg(  a^4 - \gamma_{*}(1+\gamma_{*}) a^2 + \gamma \bigg).
\end{align*}}%
Since
\begin{equation*}
\Delta = ( \gamma_{*}(1+\gamma_{*}) )^2 - 4 \gamma = \gamma_*^2(1-\gamma_*)^2 - 4 < 0,
\end{equation*}
we have $g\left( \frac{a^2 \gamma_*}{\gamma}  \right)  \geq 0$.

On the other hand, we have
\begin{align*}
	& \frac{4 a^4 + 4 a^2 + \gamma  (|\vec u|^2 +2)^2}{12\gamma} - \frac{a^2}{\gamma} \gamma_*
	\geq   \frac{4 a^4 + 4 a^2 + 4 \gamma - a^2 \gamma_*}{12\gamma}
	= \frac{1}{3\gamma} \left(a^4 +  (1-3\gamma_*)a^2 + \gamma \right).
\end{align*}
Hence,
\begin{align*}
\Delta & = (1-3\gamma_*)^2 - 4 \gamma = -3 - 6 \gamma_* + 9 \gamma_*^2 - 4 \gamma_*^3
			 	=  -3 (1-\gamma_*^2) - 6 (\gamma_* - \gamma_*^2) - 4 \gamma_*^3 \leq 0,
\end{align*}
which concludes the proof.
\end{proof}

\begin{lemma} \label{lemma:eigen-lowbound}
	It holds
	{
	\begin{equation}
	\lambda_1 \geq \frac{a^2(\gamma-1)}{\gamma} \min \left( \frac{a^2}{\gamma(|\vec u|^2+2)^2}, \, \frac{1}{4(a^2+1)},\, \frac{1}{4\gamma(|\vec u|^2+1)} \right).
	\end{equation}}%
\end{lemma}

\begin{proof}
Let $\lambda = \frac{a^2}{\gamma} x$.
From Lemma~\ref{lemma:eigen-upbound} we know that the lower boundary of $\lambda_1$ can only be found within $ 0< x \leq 1$.
With the definition of $g(\lambda)$ the following decomposition gives
{
\begin{align*}
g\left(\frac{a^2}{\gamma} x\right)
= ~& \bigg(\frac{a^6x^3}{\gamma^3}  - \frac{a^6(\gamma-1)}{4\gamma^3} \bigg)  + \frac{a^4}{4\gamma^3} \bigg( -3a^2(\gamma-1) -  8 a^2x \\
	&\quad + \gamma(|\vec u|^2+2)^2x(1-x) + 4 a^2(a^2+1)x(1-x)  + 4\gamma a^2(|\vec u|^2+1)x\bigg) \\
< ~	& \bigg(\frac{a^6x^3}{\gamma^3}  - \frac{a^6(\gamma-1)}{4\gamma^3} \bigg)  + \frac{a^4}{4\gamma^3} \bigg( \Big[ \gamma(|\vec u|^2+2)^2x(1-x)-a^2(\gamma-1)  \Big] \\
			&\quad + \Big[ 4 a^2(a^2+1)x(1-x)-a^2(\gamma-1) \Big]  + \Big[4\gamma a^2(|\vec u|^2+1)x-a^2(\gamma-1) \Big] \bigg)	\\
< ~	& \bigg(\frac{a^6x^3}{\gamma^3}  - \frac{a^6(\gamma-1)}{4\gamma^3} \bigg)  + \frac{a^4}{4\gamma^3} \bigg( \Big[ \gamma(|\vec u|^2+2)^2x-a^2(\gamma-1)  \Big] \\
			&\quad + \Big[ 4 a^2(a^2+1)x-a^2(\gamma-1) \Big]  + \Big[4\gamma a^2(|\vec u|^2+1)x-a^2(\gamma-1) \Big] \bigg),	
\end{align*}}%
where the first inequality holds since {\small$- 8 a^2x < 0$}, and the second inequality due to {\small$1-x < 1$}.
Hence, if $x$ satisfies
{
\begin{equation*}
x \leq \min \left( \left(\frac{\gamma-1}{4}\right)^{1/3} , \,  \frac{a^2(\gamma-1)}{\gamma (u^2+2)^2}, \,
								\frac{\gamma-1}{4 (a^2+1)},\, \frac{\gamma-1}{4\gamma (u^2+1)}\right),
\end{equation*}}%
then $g(\lambda) < 0$.
This concludes the proof since $\left(\frac{\gamma-1}{4}\right)^{1/3} > \frac{\gamma-1}{4}$.
\end{proof}

Summarizing, we have the estimate for the smallest eigenvalue of the entropy Hessian matrix  $\frac{d^2 \eta}{d \vec U^2}$.
\begin{lemma} \label{lem:lowerbound}
	It holds  
	{\small
		\begin{equation}\label{eq:lowerbound}
		\min\left( \lambda\left( \frac{d^2 \eta}{d \vec U^2}\right) \right) \in \left[ \frac{(\gamma-1)^2}{p} \min \left( \frac{a^2}{\gamma(|\vec u|^2+2)^2}, \, \frac{1}{4(a^2+1)},\, \frac{1}{4\gamma(|\vec u|^2+1)} \right), \frac{(\gamma-1)^{4/3}}{p} \right],
		\end{equation}}%
	where $\lambda\left( \frac{d^2 \eta}{d \vec U^2}\right)$ represents the eigenvalue of matrix  $\frac{d^2 \eta}{d \vec U^2}$ and $\eta = -\frac{\rho S}{\gamma-1}$.
\end{lemma}

Specially, $\underline{\cdot}, \overline{\cdot}$ stand here for the infimum and  supremum, respectively. 
For example, 
\begin{equation}
\underline{\rho} = \inf (\rho), \quad \overline{E} = \sup  (E), \quad \underline{\eta} = \inf \left(\min\left( \lambda\left( \frac{d^2 \eta}{d \vec U^2}\right) \right)\right),
\end{equation}
where $\lambda\left( \frac{d^2 \eta}{d \vec U^2}\right)$ represents the eigenvalue of matrix  $\frac{d^2 \eta}{d \vec U^2}$.
\begin{lemma}\label{lemma:equivalent-statement}
	The following is equivalent:
	\begin{itemize}
		\item[(i)] $\underline{\rho} > 0, \overline{E} < \infty$  for a.e. $(t,\vec{x}) \in (0,T) \times \Omega  $; 
		\item[(ii)]  $\underline{\eta} > 0$   for a.e. $(t,\vec{x})\in (0,T) \times \Omega  $.
	\end{itemize}
\end{lemma}

\begin{proof}
	\textbf{Step 1:}  Suppose that (i) holds.
	Combining Lemma~\ref{lemma:lemma-U-bounded} and Lemma~\ref{lem:lowerbound} we have
	\begin{equation*}
		\underline{\eta} \geq \frac{(\gamma-1)^2}{ \overline{p}} \min \left( \frac{\underline{a}^2}{\gamma(\overline{u}^2+2)^2}, \, \frac{1}{4(\overline{a}^2+1)},\, \frac{1}{4\gamma(\overline{u}^2+1)} \right) > 0,
	\end{equation*}
	where $\overline{a}:= \sqrt{\frac{\gamma\overline{p}}{\underline{\rho}}}, \underline{a}:= \sqrt{\frac{\gamma\underline{p}}{\overline{\rho}}}$ and $\overline{\rho}, \underline{p}, \overline{p}, \overline{u}$ are defined in the proof of  Lemma~\ref{lemma:lemma-U-bounded}.
	
	\textbf{Step 2:} Let (ii) hold. From \eqref{eq:lowerbound} we obtain
	\begin{equation}
	\overline{p} < \infty.
	\end{equation}
	In the following, with the relationship 
	\begin{equation}\label{eq:relation-eta-lam}
	\min\left( \lambda\left( \frac{d^2 \eta}{d \vec U^2}\right) \right) = \lambda_1 \cdot \frac{\gamma(\gamma-1)}{p a^2} = \lambda_1 \cdot \frac{\gamma^2(\gamma-1)}{\rho a^4} 
	\end{equation}
	we show (i) by contradiction.
	\begin{itemize}
		\item $\overline{\rho} < \infty$. 
		
		Let $\vec u, a$ fixed. It is the fact that $g(\lambda)$ only depends on $a^2, |\vec u|^2$, which implies $\lambda_1$ fixed. 
		Hence, we can derive $\overline{\rho} < \infty$ from \eqref{eq:relation-eta-lam}.
		
		\item $\overline{u} < \infty$.
		
		Let $a, p$ fixed. With $\underline{\eta} > 0$ we obtain $\underline{\lambda_1}> 0$. 
		Letting $|\vec u| \rightarrow \infty$ gives the behavior  of $g(\lambda)$ 
		\begin{equation*}
			g(\lambda) = \frac{(|\vec u|^2 +2)^2}{4 \gamma} \cdot \lambda \cdot (a^2 - \gamma \lambda) + \frac{a^2( |\vec u|^2 + 1) }{4\gamma} \lambda.
		\end{equation*}
		Combining $\lambda_1 > \underline{\lambda_1}> 0$ implies $g(\lambda_1) \rightarrow \infty$ or $g(\lambda_1) \rightarrow -\infty$, which is a contradiction with $g(\lambda_1)  = 0$. 
		
		\item $\underline{\rho} > 0$. 
		
		Let $\vec u \neq 0, p$ fixed. With $\underline{\eta} > 0$ we obtain $\underline{ \frac{\lambda_1}{a^2}}> 0$. 
		Call back that
		\begin{align*}
			g\left(\frac{a^2}{\gamma} x\right)
			= ~& \frac{a^6}{\gamma^3}  \bigg( (x-1)^2(x+1) + \gamma[(|\vec u|^2+1)x - 1] \bigg)   \\
			&\quad + \frac{a^4}{4\gamma^2} \gamma(|\vec u|^2+2)^2x(1-x)  + \frac{a^8}{\gamma^3} x(1-x). 
		\end{align*}
		We let $\lambda_1 = \frac{a^2}{\gamma} x$, which implies that $x$ satisfies $\overline{x} > 0$ and $\overline{x} \leq 1$.
		Passing to the limit $\rho \rightarrow 0$, i.e. $a^2 \rightarrow \infty$, we obtain $g\left( \lambda_1\right) \rightarrow \infty$, which is a contradiction.
	\end{itemize}
	Consequently, it holds $E \leq \frac{1}{2} \overline{\rho} \overline{u}^2 + \frac{1}{\gamma-1} \overline{p}$, 
	which concludes the proof.
\end{proof}

\end{document}